\documentclass[11pt]{article} 

\usepackage{amsmath}
\usepackage{lmodern}
\usepackage[T1]{fontenc}
\usepackage[babel=true]{microtype}
\usepackage{amsxtra}%
\usepackage{amsfonts}%
\usepackage{amssymb}%
\usepackage{amsthm}
\usepackage{graphicx}
\usepackage{epstopdf}
\usepackage{color,hyperref}
\hypersetup{colorlinks,breaklinks,
             linkcolor=blue,urlcolor=blue,
           anchorcolor=blue,citecolor=blue}
       
\usepackage{subfigure}
\usepackage{appendix}

\usepackage[margin=2.54cm]{geometry}

\newcommand{\eps}{\varepsilon}

\newcommand{\R}{\mathbb{R}}

\newcommand{\C}{\mathbb{C}}

\renewcommand{\phi}{\varphi}

\newcommand{\mcl}{\mathcal{L}}

\newcommand{\NL}{\mathrm{F}_\mathrm{res}}
\newcommand{\cNL}{\mathcal{F}_\mathrm{res}}

\usepackage{enumitem}
\setitemize{noitemsep,topsep=0pt,parsep=2pt,partopsep=0pt}
\renewcommand{\Re}{\mathrm{Re} \,}
\renewcommand{\Im}{\mathrm{Im} \,}

\def\XXint#1#2#3{{\setbox0=\hbox{$#1{#2#3}{\int}$ }
		\vcenter{\hbox{$#2#3$ }}\kern-.6\wd0}}

\newtheorem{thm}{Theorem}

\newtheorem*{thm*}{Theorem}
\newtheorem{prop}{Proposition}
\newtheorem{lemma}[prop]{Lemma}
\newtheorem{corollary}[prop]{Corollary}

\newtheorem{thmlocal}[prop]{Theorem}

\newtheorem{hyp}{Hypothesis}

\newtheorem{defi}{Definition}

\theoremstyle{definition}
\newtheorem{remark}[prop]{Remark}

\numberwithin{equation}{section}
\numberwithin{prop}{section}
\newcommand{\hl}[2][black]{\textcolor{#1}{#2}}

\newcommand\blfootnote[1]{%
	\begingroup
	\renewcommand\thefootnote{}\footnote{#1}%
	\addtocounter{footnote}{-1}%
	\endgroup
}

\begin{document}
% \maketitle
\begin{center}
{\fontsize{15}{15}\fontseries{b}\selectfont{Universal selection of pulled fronts}}\\[0.2in]
Montie Avery and Arnd Scheel \\[0.1in]
\textit{\footnotesize 
University of Minnesota, School of Mathematics,   206 Church St. S.E., Minneapolis, MN 55455, USA}
\end{center}

% \begin{abstract}
% 	We study invasion processes in semilinear parabolic equations on the real line of arbitrary order, addressing the fundamental question: \textit{when is the nonlinear propagation speed determined from linear information only}? We identify a set of general assumptions on existence and stability of traveling fronts under which we construct open classes of initial data for which the solution propagates with the linear spreading speed, up to a logarithmic delay which has been previously identified in the Fisher-KPP equation. Prior results in this direction rely heavily on comparison principles, while our results apply in particular to fourth order equations without comparison principles. Our proof is based on constructing a good approximate solution and then using a direct stability argument to find close by solutions to the full equation. To carry out this stability argument, we need detailed decay estimates for the linear problem, which we obtain through a careful analysis of the resolvent near the essential spectrum. 
% \end{abstract}
% % 

\begin{abstract}%fontsize{10}{10}\selectfont
We establish selection of critical pulled fronts in invasion processes as predicted by the marginal stability conjecture. Our result shows convergence to a pulled front with a logarithmic shift for open sets of steep initial data, including one-sided compactly supported initial conditions. We rely on robust, conceptual assumptions, namely existence and marginal spectral stability of a front traveling at the linear spreading speed and demonstrate that the assumptions  hold for open classes of spatially extended systems. Previous results relied on comparison principles or probabilistic tools with implied non-open conditions on initial data and structure of the equation. Technically, we describe the invasion process through the interaction of a Gaussian leading edge with the pulled front in the wake. Key ingredients are sharp linear decay estimates to control errors in the nonlinear matching and corrections from initial data. \blfootnote{2010 Mathematics Subject Classifications: primary 35B40, 35K25, 35B35, 35Q92, 35K55; secondary 35B36.} \blfootnote{Keywords: pulled fronts, traveling waves, pattern formation, marginal stability conjecture, diffusive stability.}
\end{abstract}

\section{Introduction}

The onset of structure formation in spatially extended physical systems is often mediated by an invasion process, in which a pointwise stable state invades a pointwise unstable state. \hl{One observes that an initially localized perturbation to the pointwise unstable background state grows in amplitude and spreads spatially. In its wake, this spreading process may select a spatially constant state, a periodic pattern, or more complicated dynamics; see \cite{vanSaarloosReview} for a thorough review of experimental observations of invasion processes across the sciences.} The fundamental objective then is to describe this process, in particular by predicting \hl{the spreading speed and the selected state in the wake}. 

A mathematical study usually focuses first on a one-sided invasion process, describing  convergence of solutions to a front that connects the unstable state in the leading edge to the selected state in the wake. Existence of such fronts and stability information is often available through a variety of analytical and computational techniques, ranging from comparison methods and monotonicity \cite{FayeHolzerLotKaVolterra,weinberger} to integrability \cite{MR2578681,vshohenberg}, to perturbative techniques in small-amplitude settings \cite{colleteckmann,MR1458064,MR1970238,SchneiderEckmann,MR4097934,MR1631295}  or in the presence of scale separation \cite{cartersch,MR2256842,RottschaferWayne}, to topological techniques \cite{MR2231782,sch} and to rigorous computational approaches \cite{beckjaquette}; see also the review \cite{vanSaarloosReview}. The \emph{marginal stability conjecture} connects such existence and stability information with the physically most interesting question as to which fronts are actually \emph{selected}, that is, observed when starting from \emph{steep}, particularly compactly supported initial conditions; see for instance \cite{deelanger,colleteckmann,vanSaarloosReview}.  In essence, the conjecture predicts that out of a family of fronts, the ``marginally stable'' front is selected and thus reduces, whenever proven, the study of invasion processes to the study of existence and stability of front solutions. Our work here makes this concept of marginal stability precise and establishes the conjecture.

\textbf{Linear marginal stability.}
Motivating this conjecture are predictions based on the linearized problem at the trivial unstable state. Spatially localized disturbances in this linearized problem grow exponentially and spread spatially. Analyzing stability in comoving frames of speed $c$, one finds that perturbations exhibit pointwise exponential decay for all speeds above a critical speed, $c>c_*$. The pointwise decay  or growth is typically encoded in the complex linear dispersion relation through \emph{pinched double roots} $\lambda_\mathrm{dr}(c)$: one finds pointwise exponential behavior $\sim \exp(\lambda_\mathrm{dr}t)$ with $\Re\lambda_\mathrm{dr}(c)<0$ for $c>c_*$ and marginal stability  $\Re\lambda_\mathrm{dr}(c_*)=0$ at the ``critical'' \emph{linear spreading speed} $c_*$  \cite{bers1983handbook,brevdo,HolzerScheelPointwiseGrowth}; \hl{see Section \ref{s: overview} for a brief review}. In addition to predictions for the speed, the linear analysis also offers a prediction for the nature of the invasion process: zero versus nonzero frequency $\Im\lambda_\mathrm{dr}(c_*)$ 
%of the marginally stable pinched double root 
predicts invasion that is
\[
\text{(S) stationary \quad or  \quad (P) time-periodic.}
\]
in the comoving frame.

\textbf{Nonlinear stability.}
Nonlinear saturation of this linear growth process leads to formation of fronts. Describing the evolution towards fronts, one invokes subtle stability arguments. Pointwise stability is mostly intractable in nonlinear equations, and one therefore resorts to  function spaces with exponential weights, penalizing in particular perturbations in the leading edge of the front that would grow due to the inherent instability of the invaded state. In the leading edge, the linearized evolution is indeed close to the linearization at the unstable state and one finds that fronts with speeds $c<c_*$ are unstable in any weighted space. Fronts with speed $c\geq c_*$ however can often be shown to be stable in exponentially weighted spaces, based on properties of the linearization at the front. One therefore analyzes the spectrum of the linearization, which is composed of essential spectrum related to stability in the leading edge, essential spectrum related to stability in the wake, and point spectrum associated with stability of the front interface. In suitably weighted spaces, one establishes or assumes that essential spectrum associated with the leading edge is stable for $c>c_*$ and \emph{marginally stable} for $c=c_*$, while point spectrum and essential spectrum associated with the state in the wake are  stable or marginally stable; see also Figure~\ref{f: spectrum and initial data}. 
% %   
% 

This perspective recovers some notion of stability but does not address the question of speed selection, since all fronts with speeds $c \geq c_*$ are stable in this sense. It turns out that permissible perturbations of the front in such stability arguments are very localized so that initial conditions given by the sum of front and perturbation  necessarily preserve the exponential decay rate in the leading edge of the front. The arguments do  in particular not describe the behavior of the  most relevant \emph{steep},  for instance step-function like, initial conditions.

\textbf{The marginal stability conjecture.}
The marginal stability conjecture states that, nevertheless,  information on the existence and stability of fronts does yield a selection criterion:
\[
 \text{\emph{Marginally stable fronts are selected by steep initial conditions.}}
\]
% Neither marginal stability nor selection are defined precisely here. 
The hypothesis is understood to hold \emph{universally} across equations for open classes of steep initial conditions, under appropriate notions of marginal stability; see \cite{vanSaarloosReview,deelanger,SchneiderEckmann,EckmannWayne} for statements of the conjecture in many specific examples. In this paper, we will define marginal stability through spectral properties of the linearization of the front in norms with exponentially growing weights in the leading edge, encoding marginal stability in the leading edge,  absence of unstable or marginally stable spectrum associated with the front interface  or the wake; see Hypotheses \ref{hyp: spreading speed}--\ref{hyp: resonance}, below. Fronts where marginal stability is caused by the leading edge are commonly referred to as \emph{pulled fronts}. In that regard, our assumptions exclude pushed fronts, with instability in the front interface, and staged invasion with instabilities in the wake. Selection of pushed fronts is easier to establish than that of pulled fronts, since for pushed fronts a spectral gap may be recovered with exponential weights; see Section \ref{s: discussion} for a discussion of pushed fronts and other secondary instability mechanisms. We define selection in Definition \ref{d: selection} as allowing for \emph{open} classes of steep initial conditions.
Our main result, precisely stated in Theorem \ref{t: main}, establishes the marginal stability conjecture under these conceptual assumptions. 

\begin{thm*}
	The marginal stability conjecture holds in the case (S) of stationary invasion assuming existence of a rigidly propagating (marginally stable)  pulled front, that is, open classes of steep initial conditions converge to an appropiately shifted front profile with speed $c_*$.
\end{thm*}
In addition, we address the \emph{universality} aspect of the marginal stability conjecture in Theorem \ref{t: robustness}, which shows that our assumptions hold for open classes of equations. Contrasted with previous work, our result does not rely on the structure of the equation or sign conditions on initial data. We rephrase the front selection problem as a stability problem for fronts similar to \cite{Gallay,FayeHolzer,AveryScheel,EckmannWayne},
but, crucially, with supercritical localization of perturbations to the front. In fact, as we shall see later, dynamics near a critical front profile exhibit diffusive decay in suitable norms for localized perturbations of the front. However, again in suitable norms, steep initial data induces   perturbations of fronts that grow linearly in the leading edge and therefore do not decay diffusively. In this respect, our results can be compared to efforts toward establishing diffusive decay in pattern-forming systems, where neutral modes decay diffusively, and where modulation equations attempt to capture dynamics when perturbations are not spatially localized \cite{dsss,GallayScheel,iyersand,jnrz2,jnrz1,zumbruninventiones,uecker}.
% gallay, eckmann wayne

\textbf{A brief history and examples.} % add some references to diffusive stability, critical decay in L^1, then derivative in L^1, small L^infty,... and derivatives in L^infty
% gallayscheel, zumbrun, howardcritical, ueckerschneider, sandstede, dsss
% 
Specific to front selection, the mathematical literature originates with work in the 1930s on the Fisher-KPP equation \cite{Fisher,Kolmogorov},%
% reference to Fisher's paper on alleles
\begin{equation}
\partial_t u = \partial_{xx}u + u - u^2, \quad x \in \R, \quad t > 0, \label{e: KPP}
\end{equation}
where fronts connecting the stable state $u \equiv 1$ to the unstable state $u \equiv 0$ are linearly stable for speeds $c\geq 2$, the linear spreading speed. Kolmogorov, Petrovskii, and Piskunov \cite{Kolmogorov} proved in 1937 that for step function initial data, $u = 1, x<0$ and  $u = 0$, $x>0$, the solution to \eqref{e: KPP} indeed converges to shifted pulled fronts,
\begin{equation}
\lim_{t \to \infty} u(x + \sigma(t), t) = q(x; 2), \label{e: KPP result}
\end{equation}
for some shift $\sigma(t) = 2t + \mathrm{o}(t)$, uniformly in space, where $q(\cdot; 2)$ is a front solution to \eqref{e: KPP} satisfying $u(x, t) = q(x-2t; 2)$, unique up to spatial translation. In particular, the speed $\sigma'(t)=2+\mathrm{o}(1)$ converges to the linear spreading speed as $t\to\infty$.
The basic idea of the proof relies on using comparison principles, with (unstable) fronts at speeds $c\lesssim 2$ as subsolution building blocks, and was subsequently adapted to a plethora of systems that allow comparison principles, sometimes in a more hidden fashion, on the real line and also in higher space dimensions; see e.g. \cite{aronson,berestyckinirenberg,hamelnadirashvili,heinze,roquejoffre,weinberger}.
% berestycki, Heinze in strips; nadirashvili in R^n, 
% weinberger's general results on order-preserving systems
In a celebrated series of papers, Bramson \cite{Bramson1, Bramson2} showed that convergence of the speed is quite slow with a universal leading-order correction that induces a $\log$-shift in the position, independent of initial conditions, 
\begin{align}
\sigma(t) = 2t - \frac{3}{2} \log t + \mathrm{O}(1).
\end{align}
The approach there relies on a probabilistic interpretation of \eqref{e: KPP} as an evolution of distributions in a branched random walk. Proofs were greatly simplified  later  using comparison principles with refined subsolutions in  \cite{Lau,Comparison1, Comparison2}. The new techniques introduced also led to refined asymptotics, allowed adaptations to other systems,  and analysis in higher space dimensions; see e.g. \cite{Comparison3, Graham, BouinHendersonRyzhik1,BouinHendersonRyzhik2,roussier}. 

% more on selection with comparison principles and maybe less on the log speed? 

% Bramson later refined this result \cite{Bramson1, Bramson2}, proving that for large classes of initial data which decay to zero sufficiently rapidly as $x \to \infty$, \eqref{e: KPP result} holds with 
% \begin{align}
% \sigma(t) = 2t - \frac{3}{2} \log t + \mathrm{O}(1).
% \end{align}
% Bramson's proof relies on the relationship of the Fisher-KPP equation to branched Brownian motion, and uses technical probabilistic arguments. The logarithmic delay result was shortly thereafter reobtained by Lau, using results on the decrease in zero number of solutions to linear parabolic equations \cite{Lau}. More recently, these results have been reobtained and refined by Hamel, Nolen, Roquejoffre, and Ryzhik to include higher order terms in the front position using direct PDE arguments \cite{Comparison1, Comparison2, Comparison3}.  While there are analogous results on front propagation in a number of other equations \cite{CalvezEtAl, BouinHendersonRyzhik1, BouinHendersonRyzhik2, HamelPeriodic}, the proofs all rely on direct application of the comparison principle or a reduction to an equation for which the comparison principle is available. 

Inspecting the comprehensive review of experimental observations and theoretical studies of front propagation into unstable states \cite{vanSaarloosReview}, almost all experimental settings and associated models including for instance fluid instabilities, crystal growth, and phase separation, do not admit a probabilistic interpretation or comparison principles, nor do they preserve positivity of initial data. In fact key examples in \cite{vanSaarloosReview} are higher-order parabolic equations that do not admit comparison principles. A prototypical case of stationary propagation (S) is the extended Fisher-KPP equation
\begin{align}
\partial_t u = - \delta^2 \partial_{xxxx} u + \partial_{xx} u + f(u), \quad x \in \R, \quad t > 0, \label{e: eKPP}
\end{align}
for $\delta$ small, where $f$ is a smooth function satisfying $f(0) = f(1) = 0$, $f'(0) > 0, f'(1) < 0$, and (for instance) $f''(u) < 0$ for all $u \in (0,1)$.  A basic example for time-periodic propagation, case (P), is the Swift-Hohenberg equation, a prototypical model for pattern formation in contexts such as Rayleigh-B\'enard convection, 
\begin{align}
\partial_t u = -\partial_{xxxx}u - 2\partial_{xx}u + (\mu-1) u - u^3, \quad x \in \R, \quad t > 0, \quad \mu > 0; \label{e: SH}
\end{align}
see \cite{SchneiderEckmann} for a formulation of the marginal stability conjecture in this particular case. Fourth order (and even sixth) order scalar equations such as \eqref{e: eKPP} and \eqref{e: SH} arise in many physical circumstances, including in phenomenological models for convection rolls \cite{sh}, in crystal nucleation and growth \cite{PhysRevB.75.064107,PhysRevE.79.051404}, in models for spatial localization of patterns across many physical systems \cite{knobloch15}, as models for phyllotaxis \cite{phyllo}, as amplitude equations derived from reaction-diffusion or fluid systems \cite{NewellWhitehead, RottschaferDoelman}, and even in Turing's late work on morphogenesis \cite{dawes}. 

Motivated by these examples, we focus on a setting of higher-order parabolic equations in which we establish selection of pulled fronts and the marginal stability conjecture in the case of stationary invasion (S). We believe that the techniques introduced here will also prove useful  in understanding front propagation in the time-periodic case (P), particularly in pattern-forming systems such as \eqref{e: SH}. We further expect the linear theory which we develop here and in \cite{AveryScheel} to be useful in understanding diffusive decay near coherent structures in other contexts. 

\subsection{Setup and main results}\label{s: setup}

We consider scalar, spatially homogeneous parabolic equations of arbitrary order, of the form 
\begin{align}
u_t = \mathcal{P}(\partial_x) u + f(u), \quad u = u (x,t) \in \R,  \quad x \in \R, \quad t > 0, \label{e: eqn}
\end{align}
with $f$ smooth, $f(0) = f(1) = 0, f'(0) > 0$, and $f'(1) < 0$, and polynomial differential operator 
\begin{align}
\mathcal{P}(\nu) = \sum_{k=0}^{2m} p_k \nu^{k}, \quad (-1)^m p_{2m} < 0, \quad p_0 = 0, \label{e: P def}
\end{align}
so that $\mathcal{P}(\partial_x)$ is elliptic of order $2m$, although not necessarily symmetric, that is, we allow nonzero coefficients of odd derivatives.  

% We restrict our consideration to scalar equations $u\in\R$, but we expect our methods to readily adapt to systems of equations. 

We pass to a comoving frame of speed $c$ and linearize at the unstable rest state $u \equiv 0$ to find
\begin{align}
u_t = \mathcal{P}(\partial_x) u + c \partial_x u + f'(0) u. \label{e: lin about 0}
\end{align}
Informally, the \textit{linear spreading speed} $c_*$ is a distinguished speed so that solutions to \eqref{e: lin about 0} with compactly supported initial data grow exponentially pointwise for $c \lesssim c_*$ and decay pointwise for $c \gtrsim c_*$. To characterize $c_*$, we  substitute $u = e^{\nu x + \lambda t}$ into \eqref{e: lin about 0} and find the \emph{dispersion relation},
\begin{align}
d^+_c (\lambda, \nu) := \mathcal{P}(\nu) + c \nu + f'(0) - \lambda. 
\end{align}
The dispersion relation determines the spectrum of the linearization about the unstable state $\mathcal{P}(\partial_x) + c_* \partial_x + f'(0)$ via the Fourier transform, in the sense that the spectrum of this operator (on, for instance, $L^p (\R)$) is given by 
\begin{align}
\Sigma^+ = \{ \lambda \in \C : d^+(\lambda, ik) = 0 \text{ for some } k \in \R \}. \label{e: Sigma plus}
\end{align}
\begin{hyp}[Linear spreading speed] \label{hyp: spreading speed}
	We assume there exists a speed $c_*$ and an exponential rate $\eta_* > 0$ such that 
	\begin{itemize}
		\item[(i)] (Simple pinched double root) For $\nu, \lambda$ near 0, we have
		for some $\alpha > 0$, 
		\begin{align}
		d^+_{c_*} (\lambda, \nu - \eta_*) = \alpha \nu^2 - \lambda + \mathrm{O} (\nu^3); \label{e: right dispersion curve}
		\end{align}
		\item[(ii)] (Minimal critical spectrum) If $d^+_{c_*} (i \kappa, ik - \eta_*) = 0$ for some $k, \kappa \in \R$, then $k = \kappa = 0$;
		\item[(iii)] (No unstable spectrum) $d^+_{c_*} (\lambda, i k - \eta_*) \neq 0$ for any $k \in \R$ and any $\lambda \in \C$ with $\Re \lambda > 0$.
	\end{itemize}
\end{hyp}
\begin{remark}
	\hl{Hypothesis \ref{hyp: spreading speed} implies that a transition from pointwise growth to pointwise decay occurs at $c=c_*$.  }
		\hl{The rate $\eta_*$ naturally arises in this stability computation, and characterizes both the exponential decay rate of pulled fronts and the natural choice of an exponentially weighted space for perturbations. Indeed, assumption (i) implies by Fourier transform in $x$  that, in a space with weight $e^{\eta_* x}$,  the essential spectrum of the linearization about $u \equiv 0$ has a branch that touches the imaginary axis at the origin and is otherwise contained in the left half plane (see Figure \ref{f: spectrum and initial data}). It is then natural to assume that the essential spectrum is otherwise stable, which is captured in (ii) and (iii), so that dynamics are governed by the marginal pointwise stability captured by the simple pinched double root criterion. On the other hand, choosing $\Re\nu\neq 0$, small in  \eqref{e: right dispersion curve} shows that the origin is unstable in weights $\eta \approx \eta_*$, $\eta\neq \eta_*$. Using (ii) and (iii) one can show that instability holds for all weights $\eta\neq \eta_*$. Similarly, for speeds $c>c_*$, we have $\Re\lambda\leq \delta<0$ when $\Re\nu=-\eta_*$, implying exponential decay of perturbations in a weighted norm, while for speeds $c<c_*$, one finds $\Re\lambda>0$ for some $\nu\in -\eta+i\R$, for \emph{any} $\eta$ fixed; see Section \ref{s: overview} for further details. It is in this sense that Hypothesis \ref{hyp: spreading speed} specifies marginal stability at speed $c_*$.}
	
\end{remark}

From now on we fix $c = c_*$ and write $d^+ = d^+_{c_*}$. 
We also define the left dispersion relation
\begin{align}
d^- (\lambda, \nu) = \mathcal{P} (\nu) + c_* \nu + f'(1) - \lambda
\end{align}
which determines the spectrum of the linearization  $\mathcal{P}(\partial_x) + c_* \partial_x + f'(1)$ at $u=1$ through
\begin{align}
\Sigma^- = \{ \lambda \in \C : d^-(\lambda, ik) = 0 \text{ for some } k \in \R \}. \label{e: Sigma minus}
\end{align}
% We are interested in \textit{pulled fronts}, for which the relevant dynamics should be driven by the linearization about the unstable state. We therefore assume that the spectrum associated to the stable state is strictly contained in the left half plane. 
We focus on dynamics in the leading edge and therefore assume (strict) stability in the wake.
\begin{hyp}[Stability in the wake]\label{hyp: stable on left}
	We assume that $\Re (\Sigma^-) < 0$. 
\end{hyp}
According to the marginal stability conjecture, one expects the propagation dynamics to be governed by traveling wave solutions, and we therefore assume existence of such a front.
\begin{hyp}[Existence of a critical front]\label{hyp: front existence}
	We assume there exists a solution to \eqref{e: eqn} of the form
		\begin{align}
	u(x,t) = q_* (x - c_* t),\qquad \lim_{\xi \to -\infty} q_* (\xi) = 1, \quad \lim_{\xi \to \infty} q_*(\xi) = 0. 
	\end{align}
	We refer to $q_*$ as the critical front. Moreover, we assume that 
	for some $a, b \in \R$ with $b \neq 0$ and some $\eta_0 > 0$,
	\begin{align}
	q_*(\xi) = (a + b \xi) e^{-\eta_* \xi} + \mathrm{O} (e^{- (\eta_*+\eta_0) \xi}).\label{e: front asymptotics}
	\end{align}
\end{hyp}
\begin{figure}
	\centering
	\includegraphics[width=1\textwidth]{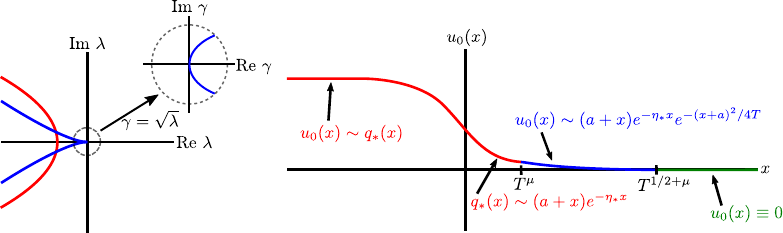}
	\caption{Left: the Fredholm borders of $\mcl$ associated to the asymptotic operators on the right (blue, marginally stable) and on the left (red, stable), together with inset showing an image of the origin under the map $\lambda \mapsto \sqrt{\lambda}$. The Fredholm borders determine the boundary of the essential spectrum; see Section \ref{s: resolvent estimates} for details. Right: schematic of a representative initial datum $u_0$ to which our selection result applies.}
	\label{f: spectrum and initial data}
\end{figure}
Possibly translating in space and reflecting $q_*\mapsto -q_*$ if necessary, we assume $b=1$. 

\begin{remark} 
% homotopy in lambda to infinity since not crossing essential spectrum, all roots stay separated by -eta_*. For large lambda there's precisely m nu's with Re nu <0, so at lambda=0, for the linearization of the traveling-wave ODE, there's an m+1-dimensional strong stable manifold W^ss of solutions with decay exp((-eta_*+eps)x) and an m-1-dimensional (very) strong manifold W^sss.  Since the wake is stable, it has an m-dimensional unstable manifold, which generically does not intersect W^sss. On the other hand, W^u and W^ss generically intersect along a one-dimensional curve, the front we are looking for. ODE theory shows that solutions have asymptotics (a+bx) exp(-eta_* x), and a,b depend smoothly on the initial condition at x=0, so generically b neq 0. The double eigenvalue is a Jordan block since otherwise the dispersion relation would start with lambda^2... I don't think we need that much detail here. 
	\hl{The assumption \eqref{e: front asymptotics} on the asymptotics of the critical front is generic for pulled fronts. Writing the traveling wave ODE as a first-order system, the front connects two equilibria. Eigenvalues at the linearization at $u=0$ are roots of $d_{c_*}^+(0,\nu)=0$. Homotoping from $\lambda=0$ to $\lambda=+\infty$, we notice that roots $\nu$ do not cross $\Re\nu=-\eta_*$ by (iii). Using \eqref{e: right dispersion curve}, one quickly concludes that the linearization at the origin has $m-1$ eigenvalues with real part less than $-\eta_*$, and a Jordan block of length 2 at $-\eta_*$.  ODE theory shows that solutions with decay at most $e^{-\eta_*x}$ then form a smooth $m$-dimensional strong stable manifold, solutions with decay at most  $x e^{-\eta_*x}$ form a smooth $m+1$-dimensional strong stable manifold. Similarly, the equilibrium corresponding to $u=1$ is hyperbolic and possesses an $m$-dimensional unstable manifold.
	Counting dimensions, we find generically a discrete set of %locally unique
	heteroclinic orbits with  decay rate $x e^{-\eta_*x}$ but no solutions with decay  $e^{-\eta_*x}$, confirming the asymptotics \eqref{e: front asymptotics} in a generic situation. }
% 	For $\lambda\gtrsim 0$, there are precisely two roots near $\Re\nu=0$, $\nu_\pm=-\eta_*\pm\sqrt{\lambda/\alpha}$. Since for $|lambda\gg 1$, 	the critical front is an intersection between the stable manifold of the origin and the unstable manifold of the rest state $u \equiv 1$. Hypothesis \ref{hyp: spreading speed} implies that the linearization about the origin in the traveling wave ODE has a 2-by-2 Jordan black at $\nu = -\eta_*$ \cite{HolzerScheelPointwiseGrowth}, corresponding to the decay modes $e^{-\eta_* \xi}$ and $(a+b\xi) e^{-\eta_* \xi}$. If there are no other eigenvalues $\nu$ of this linearization with $-\eta_* < \Re \nu < 0$, then generically the front converges to the weak stable manifold along its strong stable foliation, and hence has the asymptotics \eqref{e: front asymptotics}. If there is an eigenvalue $\nu$ with $-\eta_* < \Re \nu < 0$, then by dimension counting fronts with speed $c_*$ generically come in a two-parameter family (with one parameter arising from translation invariance). There is generically a distinguished member of this family which converges along the Jordan block [\cmt{to Arnd: why?}], and it is this front which governs the propagation dynamics.}
\end{remark}

Finally, we need to exclude the possibility that the fronts are \textit{pushed} in the sense that the nonlinearity accelerates the speed of propagation. Typically, in the presence of pushed fronts, the linearization about the critical front has an unstable eigenvalue. This linearization is given by 
\begin{align}
\mathcal{A} = \mathcal{P}(\partial_x) + c \partial_x + f'(q_*).
\end{align}
The assumption $f'(0) > 0$ implies that the essential spectrum of $\mathcal{A}$ on $L^2$ is unstable \cite{FiedlerScheel,Palmer1, Palmer2}, but Hypotheses \ref{hyp: spreading speed} and \ref{hyp: stable on left} imply that, with a smooth positive exponential weight $\omega$ satisfying 
\begin{align}
\omega(x) = \begin{cases}
e^{\eta_* x}, & x \geq 1, \\
1, & x \leq -1,
\end{cases} \label{e: weight critical}
\end{align}
the essential spectrum of the conjugate operator $\mcl = \omega \mathcal{A} \omega^{-1}$ is marginally stable; see Figure \ref{f: spectrum and initial data} for a schematic, and the beginning of Section \ref{s: resolvent estimates} for further details. To exclude pushed fronts, we assume the following. 
\begin{hyp}[No resonance or unstable point spectrum]\label{hyp: resonance}
	We assume that $\mcl: H^{2m} (\R) \subset L^2(\R) \to L^2(\R)$ does not have any eigenvalues $\lambda$ with $\Re \lambda \geq 0$. We further assume that there is no bounded solution to $\mcl u = 0$. 
\end{hyp}
In the Fisher-KPP setting, fronts with the linear spreading speed $c=2$ satisfy Hypotheses \ref{hyp: spreading speed}--\ref{hyp: resonance}. Absence of a bounded solution to $\mcl u = 0$ in the Fisher-KPP equation is a consequence of the weak exponential decay \eqref{e: front asymptotics} of the critical front. Instabilities as excluded in Hypothesis \ref{hyp: resonance} occur for instance when considering an asymmetric cubic nonlinearity and lead to the selection of \emph{pushed fronts} that propagate at a speed $c_\mathrm{pushed}>c_*$; see \cite{HadelerRothe}. Separating pulled fronts with speed $c=2$ from pushed fronts is the case of of a bounded solution to $\mathcal{L}u=0$, excluded in Hypothesis \ref{hyp: resonance}. In the case that there is such a bounded solution, we say $\lambda = 0$ is a \textit{resonance} of $\mcl$. 

Before we can state our main results, we address the question: \textit{what features should a meaningful front selection result have}? First, such a result should establish that, for an open class of initial data, the front interface is located approximately at $x = c_* t$ for large times, so that the asymptotic speed of propagation is the linear spreading speed. Second, the class of initial data should include some data that is compactly supported in the leading edge, commonly the most interesting case.  Such a setting rules out selection in this sense of faster-traveling, \textit{supercritical fronts} \cite{EbertvanSaarloos}, $c>c_*$, which attract open sets of initial data but only with well-prepared, slowly decaying exponential tails.

\begin{defi}\label{d: selection}
	We say that a front $q_*$ with speed $c_*$  in \eqref{e: eqn} is a \textbf{selected front} if an open class of steep initial conditions propagates with asymptotic speed $c_*$ and stays close to translates of the front $q_*$. More precisely, we require that there exists a non-negative continuous weight $\rho : \R \to \R$ and,  for any $\eps>0$, a set of initial data $\mathcal{U}_\eps \subseteq L^\infty (\R)$ such that:
	\begin{itemize}
		\item[(i)] for any $u_0 \in \mathcal{U}_\eps$, there exists a function $h(t) = \mathrm{o}(t)$  such that the solution $u(x, t)$ to \eqref{e: eqn} with initial data $u_0$ satisfies, for $t$ sufficiently large,
		\begin{align}
		\| u(\cdot + c_* t + h(t), t) - q_* (\cdot) \|_{L^\infty(\R)} < \eps; \label{e: basin propagation}
		\end{align}
		\item[(ii)] there exists $u_0 \in \mathcal{U}_\eps$ such that $u_0 (x) = 0$ for all $x>0$ sufficiently large;
		\item[(iii)]  $\mathcal{U}_\eps$ is open in the topology induced by the norm $\| g \|_{\rho} = \| \rho g \|_{L^\infty}$.
	\end{itemize}
\end{defi}

Our result includes a specific choice of algebraic weight with smooth, positive weight function $\rho_r$, $r\in\R$, satisfying 
\begin{align}
\rho_r (x) = \begin{cases}
\langle x \rangle^r, & x \geq 1, \\
1, & x \leq -1, \label{e: rho one sided}
\end{cases}
\end{align}
where $\langle x \rangle = (1+x^2)^{1/2}$. 
\begin{thm}\label{t: main}
	Assume Hypotheses \ref{hyp: spreading speed} through \ref{hyp: resonance} hold. Then the critical front $q_*$ with speed $c_*$ is a selected front, with weight $\rho = \rho_r \omega$ for $r = 2+\mu$, $0 < \mu < \frac{1}{8}$. Furthermore, for each $\eps>0$ small there exists $\mathcal{U}_\eps \subseteq L^\infty(\R)$ satisfying Definition~\ref{d: selection}, (ii)--(iii), such that  the refined estimate
	\begin{align}
	 \sup_{x \in \R} | \rho_{-1} (x) \omega(x) [u(x + \sigma (t), t) - q_* (x) ]| < \eps
	\end{align}
	holds for solutions with initial data  $u_0 \in \mathcal{U}_\eps$ and $t\geq t_*(u_0)$, sufficiently large, where, for some $x_\infty(u_0)$,
	\begin{align}
	\sigma(t) = c_* t - \frac{3}{2\eta_*} \log t + x_\infty(u_0).
	\end{align}
\end{thm}
% In particular, Theorem \ref{t: main} says that our Hypotheses \ref{hyp: spreading speed} through \ref{hyp: resonance} guarantee that the fronts are pulled, so that propagation occurs in the nonlinear equation with the linear spreading speed $c_*$.

\begin{remark}
	\hl{In addition to universal selection of pulled fronts, Theorem \ref{t: main} establishes universality of the logarithmic delay $-\frac{3}{2 \eta_*} \log t$ as identified by Ebert and van Saarloos \cite{EbertvanSaarloos}: the delay is present in all equations satisfying our conceptual assumptions and is independent of the initial data. The matched asymptotics in \cite{vanSaarloosReview} suggest that there are further universal terms in the expansion of the position of the front: in particular, $u(x + \tilde{\sigma}(t), t) \to q_* (x)$, where
	\begin{align*}
		\tilde{\sigma}(t) = c_* t - \frac{3}{2 \eta_*} \log t + \tilde{x}_\infty + \frac{\sigma_2}{\sqrt{t}} + \mathrm{o}(t^{-1/2})
	\end{align*}
	for a constant $\sigma_2$ with explicit, universal form depending only on the leading order expansion of the dispersion relation. Such expansions, including to higher orders, have been verified for the Fisher-KPP equation in \cite{Comparison3, Graham}. We do not pursue such higher expansions in the more general setting here, instead focusing on the most relevant questions of selection of the speed and the state in the wake.
% 	, instead focusing on selection of the speed and the state in the wake. We also do not identify the precise translate of the critical front to which the solution converges, since we only require that the solution remains $\eps$-close to $q_*$ in a suitable norm. We expect that in our setting $u(x + \sigma(t) + x_0 (\eps), t) \to q_* (x)$, where $x_0 (\eps) \to 0$ as $\eps \to 0$. It should be possible to prove this in our framework but with substantial additional technical difficulties. The most important quantities in practice are typically the propagation speed and the selected state in the wake, so we focus on these in Theorem \ref{t: main}.
}
\end{remark}
% 
% Stepping back, Theorem \ref{t: main} validates Hypotheses  \ref{hyp: spreading speed}--\ref{hyp: resonance} as largely model-independent sufficient criteria for selection of \emph{pulled fronts}, that is, for fronts where dynamics  and selection are determined by properties of the linearization in the leading edge.

% From this perspective, one can interpret Hypotheses  \ref{hyp: spreading speed}--\ref{hyp: resonance} as making precise the concept of marginal stability and Theorem \ref{t: main} as establishing the marginal stability conjecture, that is, selection of the marginally stable front in a roughly model-independent fashion, in the case of stationary invasion. 

We emphasize that we do not require any structure of the equation beyond Hypotheses \ref{hyp: spreading speed} through \ref{hyp: resonance} --- in particular, our results apply to equations without comparison principles. The first author together with Gar\'enaux recently proved \cite{AveryGarenaux} that the extended Fisher-KPP equation \eqref{e: eKPP} satisfies Hypotheses \ref{hyp: spreading speed} through \ref{hyp: resonance}, so that Theorem \ref{t: main} applies immediately in that setting. Adapting the ideas therein, we show that the class of equations we consider here is open, thereby emphasizing the universality across different equations. We make this precise in the following theorem. 
\begin{thm}\label{t: robustness}
	Assume that $\mathcal{P}(\partial_x; \delta)$ is a family of operators of degree $2m$ of the form \eqref{e: P def} for $\delta \in (-\delta_0, \delta_0)$ for some $\delta_0 > 0$, with coefficients $p_k (\delta)$ smooth in $\delta$, $p_{2m} (0) \neq 0$, and that $f = f(u; \delta)$ is smooth in both $u$ and $\delta$. Suppose that $f(0; 0) = f(1; 0) = 0,$ $f'(0; 0) > 0, f'(1; 0) < 0$, and $\mathcal{P}(\partial_x; 0)$ and $f(u; 0)$ are such that Hypotheses \ref{hyp: spreading speed} through \ref{hyp: resonance} are satisfied. Assume further that $f(0; \delta) = f(1; \delta) = 0$. Then Hypotheses \ref{hyp: spreading speed} through \ref{hyp: resonance} also hold for $\delta$ sufficiently small, and hence Theorem \ref{t: main} holds for all $\delta$ sufficiently small. 
\end{thm}
We note in passing that the assumptions in Theorem \ref{t: robustness} on the nonlinearity at $\delta = 0$ imply that $0$ and $1$ perturb smoothly to nearby zeros of $f$ for $\delta$ small. Shifting and rescaling $u$, we may then assume without loss of generality that $f(0; \delta) = f(1; \delta) = 0$. 

Together, Theorems \ref{t: main} and \ref{t: robustness} establish propagation at the linear spreading speed for open classes of initial data and open classes of equations. 

% We suspect with some notable exceptions pertaining to stability in a fixed exponentially weighted space and possible neutral or unstable modes in the wake that we list in Section \ref{s: discussion}, our assumptions are in fact not only sufficientalso necessary. 
% \end{remark}
% 

\subsection{Overview and preliminaries}\label{s: overview}

\noindent \textbf{Pointwise stability and pinched double roots.} We give a brief review of concepts of pointwise decay driving our understanding of invasion processes, and the role of pinched double roots, based on the presentation in \cite{HolzerScheelPointwiseGrowth}. To determine pointwise growth or decay for solutions to the linearization about the unstable state, we use the inverse Laplace transform to write the solution to \eqref{e: lin about 0} as 
\begin{align}
	u(x, t) = \frac{1}{2 \pi i} \int_\Gamma e^{\lambda t} \int_\R G_\lambda (x-y) u_0 (y) \, dy \, d \lambda,
\end{align}
where $u_0$ is the initial data, $\Gamma$ is a suitable contour to the right of the essential spectrum, and $G_\lambda (\xi)$ is the resolvent kernel, which solves
\begin{align}
	(\mathcal{P}(\partial_\xi) + c \partial_\xi + f'(0)  - \lambda) G_\lambda = -\delta_0. 
\end{align}
We restrict to strongly localized, for instance compactly supported, $u_0$. To obtain optimal decay of $u(x, t)$, one aims to shift the integration contour $\Gamma$ as far to the left as possible. Because the initial data is strongly localized and because we are interested in the \emph{pointwise} growth or decay of the solution, rather than in a fixed norm, the obstruction to shifting the contour $\Gamma$ is \emph{not} the essential spectrum of $P(\partial_\xi) + c \partial_\xi + f'(0)$, which depends on the choice of function space and may be moved with exponential weights. Instead, the only obstructions to shifting the contour are singularities of $\lambda \mapsto G_\lambda (\xi)$ for \emph{fixed} $\xi$. 

To track singularities of $G_\lambda(\xi)$, we recast the equation as a first order system, and solve for the matrix Green's function $T_\lambda(\xi)$, which solves
\begin{align*}
	(\partial_\xi - M(\lambda)) T_\lambda = -\delta_0 \cdot I
\end{align*}
for some matrix $M(\lambda)$ which is polynomial in $\lambda$. The resolvent kernel $G_\lambda$ may be recovered from $T_\lambda$, and both have precisely the same pointwise singularities \cite[Lemma 2.1]{HolzerScheelPointwiseGrowth}. The matrix Green's function $T_\lambda(\xi)$ has the explicit expression
\begin{align*}
	T_\lambda (\xi) = \begin{cases}
		-e^{M(\lambda) \xi} P^\mathrm{s}(\lambda), &x > 0,\\
		e^{M(\lambda) \xi} P^\mathrm{u}(\lambda), &x < 0,
	\end{cases}
\end{align*}
where $P^{\mathrm{s}/\mathrm{u}} (\lambda)$ are, for $\lambda$ sufficiently large, the projections onto the stable and unstable eigenspaces of $M(\lambda)$, which are well separated for $\Re \lambda \gg 1$ due to the fact that the underlying equation is parabolic. Since $M(\lambda)$ is polynomial in $\lambda$, singularities of $T_\lambda (\xi)$ are precisely the singularities of $P^{\mathrm{s}/\mathrm{u}} (\lambda)$. Using the Dunford integral, these projections may be analytically continued from $\Re \lambda \gg 1$ until an eigenvalue of $M(\lambda)$ which was stable for $\Re \lambda \gg 1$ collides with an eigenvalue of $M(\lambda)$ which was unstable for $\Re \lambda \gg 1$. Eigenvalues $\nu$ of $M(\lambda)$ are precisely roots $\nu$ of the dispersion relation, and such a collision of stable and unstable eigenvalues is a \emph{pinched double root} of the dispersion relation, by definition; see \cite[Definition 4.2]{HolzerScheelPointwiseGrowth}. The term ``pinched'' refers to the fact that the colliding roots come from opposite (that is, stable and unstable) directions for $\Re \lambda \gg 1$. 

To summarize, the contour $\Gamma$ can be shifted to the left until we reach a pointwise singularity of $G_\lambda(\xi)$, and all such singularities are pinched double roots of the dispersion relation. The pinched double root with maximal real part therefore gives an upper bound on the pointwise exponential decay rate of $u(x, t)$. It is possible to have a pinched double root that is not a singularity of $G_\lambda(\xi)$, if the eigenvalues collide but have distinct limiting eigenspaces \cite[Remark 4.5]{HolzerScheelPointwiseGrowth}. We exclude this possibility in Hypothesis \ref{hyp: spreading speed} by restricting to \emph{simple} pinched double roots, which are robust (see Lemma \ref{l: robustness of simple PDR}) and always produce pointwise growth modes \cite[Lemma 4.4]{HolzerScheelPointwiseGrowth}. 

The linear spreading speed $c_*$ in case (S) is then characterized by a simple pinched double root at the origin, as in Hypothesis \ref{hyp: spreading speed}(i); compare \cite[Section 6]{HolzerScheelPointwiseGrowth}. Assumptions (ii)-(iii) of Hypothesis \ref{hyp: spreading speed} on minimality of critical spectrum guarantee that this is the most unstable pinched double root, so that the linearization precisely exhibits marginal pointwise stability, as required by the marginal stability conjecture. A short calculation shows that the double root $\lambda_\mathrm{dr}$ at the origin moves to the right as $c$ decreases, $\lambda_\mathrm{dr}'(c_*)<0$, so that the origin is pointwise exponentially stable for $c>c_*$ and exponentially unstable for $c<c_*$ \cite[Remark 6.6]{HolzerScheelPointwiseGrowth}. Heuristically, marginal pointwise stability in the leading edge allows solutions evolving from steep initial data to develop a Gaussian tail which does not decay rapidly in time, allowing for matching with the front interface on an intermediate length scale, thus explaining the selection of $c_*$; see Figure \ref{f: spectrum and initial data} and the discussion below.

Note that spreading speeds with marginally stable pinched double roots are quite generally well-defined \cite[Corollary 6.5]{HolzerScheelPointwiseGrowth}, motivating the conceptual, equation-independent setup here. 
% could give precise references to [41], say lemma...

\noindent \textbf{Marginal stability as a selection mechanism.} Hypothesis \ref{hyp: spreading speed} guarantees that, in the frame moving with $c_*$, the state $u \equiv 0$ is marginally pointwise stable, and in particular marginally stable in a weighted space with weight $e^{\eta_* x}$. Hypothesis \ref{hyp: front existence} gives us a front to perturb from, although perturbations that cut off the tail of the front, which are the most relevant for front selection, are large perturbations growing linearly in $x$, due to the front asymptotics $xe^{-\eta_*x}$. For initial data that vanish for $x$ sufficiently large, the dynamics for large $x$ are governed by the linearization about the unstable state, which, in the co-moving frame with speed $c_*$ and in a weighted space with weight $e^{\eta_* x}$, is given by
\begin{align}
    \mcl^+ = \mathcal{P}(\partial_x - \eta_*) + c_* (\partial_x - \eta_*) + f'(0) = \alpha \partial_{xx} + \mathrm{O}(\partial_x^3), \label{e: Lplus def}
\end{align}
by Hypothesis \ref{hyp: spreading speed}. This leads to diffusive dynamics in the leading edge: the front rebuilds its tail with Gaussian asymptotics
\begin{align*}
    u\left(x+\frac{3}{2 \eta_*} \log t,t\right) \sim x e^{-\eta_* x} e^{-x^2/(4 \alpha t)}. 
\end{align*}
See Section \ref{s: approximate solution} for further details on the precise form of this Gaussian tail. These diffusive dynamics, with no temporal decay, allow for matching with the front $q_*(x) \sim x e^{-\eta_* x}$ on the intermediate length scale $x +\frac{3}{2 \eta_*} \log t \sim  t^\mu, 0 < \mu \ll 1$. This is the intuition for the speed selection here: for steep initial data, the dynamics are initially driven by the diffusive repair at $+\infty$, which then pulls the front forward at the natural speed $c_*$ associated to this diffusive repair. 

For speeds $c \neq c_*$, the discussion on pointwise growth above implies that the linearization in the leading edge will either grow or decay exponentially in time, precluding matching with the front, which is constant in time in the comoving frame. Also, instabilities beyond the one associated with the pointwise growth in the leading edge would induce temporal growth in the frame with speed $c_*$, again preventing matching with the front on the intermediate length scale. In this sense, the marginal stability of the front, that is, choosing $c=c_*$ and excluding other instabilities as made precise in Hypotheses \ref{hyp: spreading speed}--\ref{hyp: resonance}, is necessary for matching and selection of the pulled front in the sense that failure of marginal stability would select a different profile. An important boundary case, excluded here by Hypothesis \ref{hyp: stable on left}, is \emph{diffusively stable} essential spectrum in the wake, touching the imaginary axis in a parabolic fashion, rather than exponentially stable as assumed in Hypothesis \ref{hyp: stable on left}, a scenario typical  for pattern-forming fronts such as those in \eqref{e: SH}. While this scenario is excluded here, preliminary sharp stability results providing a basis for selection were recently obtained in  \cite{AveryScheelGL}. 

\noindent \textbf{Sketch of the main proof.} Absent a comparison principle but equipped with assumptions on the linearization at a given front profile, one would like to phrase the selection problem as a stability problem. Initial conditions with vanishing support for large $x>0$ can be thought of as perturbations of size $xe^{-\eta_* x}$, which however are not small perturbations in a suitable function space. Indeed, the weighted front satisfies $\omega(x) q_*(x) \sim x$ as $x \to \infty$ by Hypothesis \ref{hyp: front existence}, so that a perturbation which cuts off the front tail is only small in a function space such as $L^\infty_{-r} (\R)$ for $r > 1$ (after already including the exponential weight; see below for definitions of weighted spaces). However, by the argument of \cite[Proposition 7.6]{AveryScheel}, one can show that the linear evolution to such a perturbation will typically grow like $t^\beta$ for some $\beta > 0$, precluding a nonlinear perturbative argument. 

\hl{As suggested in the above discussion}, we overcome this difficulty by perturbing instead from a refined profile, informed by the formal asymptotics in \cite{EbertvanSaarloos}, which resembles the critical front for $x + c_* t - \frac{3}{2\eta_*} \log t \ll \sqrt{t}$ with a Gaussian tail for $x + c_* t - \frac{3}{2 \eta_*} \log t \gg \sqrt{t}$. Such a construction was carried out for the Fisher-KPP equation in \cite{Comparison2, Comparison3} and used together with the comparison principle to establish a refined description of the asymptotics of the front position. The key insight in this construction is to match on an intermediate length scale $x \sim t^\mu$ for some $\mu > 0$ small.

As a first main ingredient to our result, we construct such an approximate solution in our conceptual setup in a way that guarantees small residuals. Based on this first step, most of our work is concerned with establishing stability in time of such an approximate solution. In order to guarantee small residuals, we let the approximate solution evolve for some large time $T$ to an initial profile, such that small perturbations to the approximate solution include initial conditions which vanish for $x$ sufficiently large; see Figure \ref{f: spectrum and initial data}. 

In the second step, we establish stability by closing a perturbative argument. The main difficulty here stems from the fact that the logarithmic shift introduces critical terms into the linear dynamics. We therefore need sharp estimates on the linearized evolution which we obtain by refining resolvent estimates originally derived in order to conclude stability of the critical front in \cite{AveryScheel}. In order to close the nonlinear argument, we rely on sharp characterizations of decay and nonlinear contributions in terms of $T$, the characteristic scale of the initial Gaussian tail.

To illustrate the still substantial difficulties in closing this perturbative argument, consider the heavily simplified model problem 
\begin{align}
\begin{cases}
w_t = w_{xx} - \frac{3}{2(t+T)} (w_x - w), & x > 0, t > 0, \\
w = 0, & x = 0, t>0,
\end{cases}
\end{align}
where the diffusive term captures spectral properties in the leading edge and the non-autonomous terms are induced by the logarithmic shift. 
The autonomous linear evolution, \hl{that is, ignoring the $-\frac{3}{2}(t+T)^{-1} (w_x - w)$ terms}, allows $t^{-3/2}$ algebraic decay in suitable norms provided the initial data is sufficiently localized.  However, it turns out that the term $\frac{3}{2} (t+T)^{-1} w$ is \textit{critical}, so that in fact $w$ does not decay but instead remains $\mathrm{O}(1)$.  This feature is explicit after the simple but insightful change of variables $z = (t+T)^{-3/2} w$, which eliminates the critical term, giving an equation
\begin{align*}
\begin{cases}
z_t = z_{xx} - \frac{3}{2(t+T)} z_x, & x > 0, t > 0 \\\
z = 0, & x = 0, t > 0. 
\end{cases}
\end{align*}
\hl{The term $-\frac{3}{2} (t+T)^{-1} z_x$ has improved decay properties compared to $\frac{3}{2} (t+T)^{-1} w$ due to the presence of an extra spatial derivative, as spatial derivatives of the heat kernel exhibit faster decay}. Using sharp estimates on decay of derivatives and several bootstrap steps, refining in particular the estimates in \cite{AveryScheel},  we find $T$-uniform decay estimates $(t+T)^{-3/2}$ for small initial data. In $z$-variables, the nonlinearity causes additional complications due to the factor $(t+T)^{3/2}$, which we account for using sharp $T$-dependent characterizations of decay. A significant part of our efforts is then concerned with establishing robust decay estimates, equivalent to those in the model problem but based only on our conceptual assumptions, for the linearized evolution near our approximate solution, which in turn we base on estimates on the  linearization at the critical front $\mcl$.
Throughout, we cannot and do not rely on comparison principles which are not available for the full problem.

% We obtain linear estimates for this principal part from sharp estimates on the behavior of the resolvent $(\mcl - \gamma^2)^{-1}$ near the essential spectrum of $\mcl$, so that we can deform our integration contours in the inverse Laplace transform to integrate \textit{along the essential spectrum}. 
%
% Resolvent estimates are established by exploiting the fact that the coefficients of $\mcl$ converge exponentially quickly to constants as $x \to \pm \infty$, so that we can use a partition of unity to decompose the equation $(\mcl - \gamma^2) u = f$ into a left piece, a right piece, and a center piece. We solve the left and right pieces with the resolvents for the asymptotic operators, using a precise description of the integral kernel for the asymptotic resolvent $(\mcl^+ - \gamma^2)^{-1}$ given in \cite[Section 2]{AveryScheel}.
%
% The resulting sharp linear decay  estimate
% \begin{align}
% \| e^{\mcl t} \|_{L^1_1 \to L^\infty_{-1}} \leq \frac{C}{t^{3/2}}, \quad t > 0, \label{e: intro t 3/2}
% \end{align}
% was essentially captured in \cite[Proposition 4.1]{AveryScheel} (see below for definitions of algebraically weighted norms), although here we also need precise estimates on derivatives. We in particular need to replace the $L^2$-based theory in \cite{AveryScheel} by sharp weighted $L^1$--$L^\infty$ estimates in order to  handle borderline cases.
% 
From this perspective, our results can be seen as an extension of stability results for critical fronts to actual selection mechanisms for fronts and invasion speeds, by placing the selection problem in a sufficiently broad perturbative framework. Indeed, our previous work \cite{AveryScheel} was motivated by stability results for critical  Fisher-KPP fronts by Gallay \cite{Gallay} and the more recent approach by Faye and Holzer \cite{FayeHolzer} using more direct pointwise semigroup methods. \hl{Our analysis in \cite{AveryScheel} establishes stability of pulled fronts against \emph{localized} perturbations that do not alter the exponential tail of the front, thus not sufficiently large in the tail to establish front selection, but also develops the fundamentals of the linear theory we rely on and adapt to our needs here}. This linear theory can be viewed as a robust functional analytic alternative to stability problems that have been successfully analyzed using pointwise resolvents, Evans functions, and  pointwise semigroup methods \cite{ZumbrunHoward,ZumbrunGardner,KapitulaSandstede}. On the level of the resolvent, we indeed replace the pointwise Evans function techniques with an equivalent functional analytic approach to tracking eigenvalues and resonances  based on farfield-core decompositions, initially developed in \cite{PoganScheel}. 

\noindent \textbf{Outline of the paper.} In Section \ref{s: approximate solution}, we use a matching procedure to construct a good approximate solution of \eqref{e: eqn} which moves with the expected speed. In Section \ref{s: resolvent estimates}, we use a far-field/core decomposition to prove sharp estimates on the resolvent $(\mcl - \gamma^2)^{-1}$. In Section \ref{s: linear estimates}, we use carefully chosen Laplace inversion contours as in \cite{AveryScheel} to translate these resolvent estimates into sharp linear decay estimates. We then carry out a nonlinear stability analysis in Section \ref{s: stability argument} to prove that certain classes of solutions to \eqref{e: eqn} resemble our approximate solution. We rephrase these results as statements on front propagation in Section \ref{s: propagation}, thereby proving Theorem \ref{t: main}. In Section \ref{s: robustness}, we use ideas from \cite{AveryGarenaux} to show that our assumptions hold for open classes of equations. We conclude in Section \ref{s: discussion} with a discussion of extensions of our results (including to systems of parabolic equations) and some of the challenges therein. 

\noindent \textbf{Function spaces.} We will need algebraic and exponential weights generalizing those defined in \eqref{e: rho one sided} and \eqref{e: weight critical}. For $r_-, r_+ \in \R$, we define a smooth positive algebraic weight 
\begin{align}
\rho_{r_-, r_+} (x) = \begin{cases}
\langle x \rangle^{r_+}, &x \geq 1, \\
\langle x \rangle^{r_-}, &x \leq -1. 
\end{cases}
\end{align}
For a non-negative integer $k$ and a real number $1 \leq p \leq \infty$, we define the corresponding algebraically weighted Sobolev space $W^{k, p}_{r_-, r_+} (\R)$ through 
\begin{align}
\| g \|_{W^{k,p}_{r_-, r_+}} = \| \rho_{r_-, r_+} g \|_{W^{k,p}}
\end{align}
where $W^{k,p} (\R)$ is the standard Sobolev space with differentiability index $k$ and integrability $p$. If $r_- = 0$ and $r_+ = r$, we write $\rho_{0, r_+} = \rho_r$ and $W^{k,p}_{0, r_+}(\R)=W^{k, p}_r(\R)$ with norm $\| \cdot \|_{W^{k, p}_r}$. For $k = 0$, we write $W^{0, p}_{r_-, r_+} (\R) = L^p_{r_-, r_+} (\R)$, and denote the norm by $\| \cdot \|_{L^p_{r_-, r_+}},$ or $\| \cdot \|_{L^p_{r}}$ in the case where $r_- = 0, r_+ = r$. 
Similarly, for $\eta_-, \eta_+ \in \R$ we let $\omega_{\eta_-, \eta_+}$ be a smooth positive exponential weight satisfying 
\begin{align}
\omega_{\eta_-, \eta_+} (x) = \begin{cases}
e^{\eta_+ x}, & x \geq 1, \\
e^{\eta_- x}, & x \leq -1,
\end{cases}
\end{align}
and define corresponding exponentially weighted Sobolev spaces $W^{k,p}_{\mathrm{exp}, \eta_-, \eta_+} (\R)$ through the norms 
\begin{align}
\| g \|_{W^{k,p}_{\mathrm{exp}, \eta_-, \eta_+}} = \| \omega_{\eta_-, \eta_+} g \|_{W^{k,p}}. 
\end{align}
Again, we write $\omega_{0, \eta_+} = \omega_{\eta}$ when $\eta_- = 0$ and $\eta_+ = \eta$, and  $W^{k,p}_{\mathrm{exp}, 0, \eta_+}(\R)=W^{k,p}_{\mathrm{exp}, \eta} (\R)$, and, for $k = 0$, we write $W^{0,p}_{\mathrm{exp},\eta_-, \eta_+} (\R) = L^p_{\mathrm{exp},\eta_-, \eta_+} (\R)$, with corresponding notation for the norms.

\noindent \textbf{Additional notation.} For two Banach spaces $X$ and $Y$, we let $\mathcal{B}(X,Y)$ denote the space of bounded linear operators from $X$ to $Y$ equipped with the operator norm topology. For $\delta > 0$, we let $B(0, \delta)$ denote the open ball in the complex plane with radius $\delta$. When the intention is clear, we may abuse notation slightly by writing a function $u(x,t)$ as $u(t) = u(\cdot, t)$, viewing it as an element of some function space for each $t$. 

\noindent \textbf{Acknowledgements.} This material is based upon work supported by the National Science
Foundation through the Graduate Research Fellowship Program under Grant No. 00074041, as
well as through NSF-DMS-1907391. Any opinions, findings, and conclusions or recommendations
expressed in this material are those of the authors and do not necessarily reflect the views of the
National Science Foundation.

\section{Construction of the approximate solution}\label{s: approximate solution}
We cast \eqref{e: eqn} in the co-moving frame with position 
\begin{align*}
\xi = x - c_* t + \frac{3}{2 \eta_*} \log (t+T) - \frac{3}{2 \eta_*} \log (T),
\end{align*}
i.e. in a frame that moves with the linear spreading speed up to the logarithmic delay. Here we write the logarithmic shift as $\frac{3}{2\eta_*} (\log(t+T) - \log(T))$ to capture the phase shift resulting from letting the approximate solution evolve for time $T$. After relabeling $\xi$ as $x$ again, we find
\begin{align}
u_t = \mathcal{P}(\partial_x) u + \left( c_* - \frac{3}{2 \eta_* (t+T)} \right) u_x + f(u). 
\end{align}
We next use an exponential weight to stabilize the linear part of the equation, defining $v = \omega u$ with $\omega$ from \eqref{e: weight critical}. The weighted variable $v$ solves $\NL[v] = 0$, where the nonlinear operator $\NL$ is
\begin{align}
\NL[v] = v_t - \omega \mathcal{P}(\partial_x) (\omega^{-1} v) - \left( c_* - \frac{3}{2 \eta_* (t+T)} \right) (\omega (\omega^{-1})' v + v_x) - \omega f(\omega^{-1} v). \label{e: NL def}
\end{align}
Our goal in this section is to construct an approximate solution $\psi$ such that $\NL[\psi] (x,t) = R(x,t)$ with $\| R (\cdot, t) \|_{L^\infty_r}$ small in a suitable sense. We follow the construction in \cite{Comparison3}, modifying it for the higher order equations considered here and only including the terms which are relevant for our analysis. The basic idea is to use an appropriate shift of the front to construct the ``interior'' of our approximate solution, and then glue this on the intermediate length scale $x \sim (t+T)^\mu$ to a diffusive tail which we construct in self-similar coordinates. \hl{The construction in this section relies only on the dynamics in the leading edge captured by Hypothesis \ref{hyp: spreading speed}(i) and the existence of a pulled front with generic asymptotics assumed in Hypothesis \ref{hyp: front existence}. Additional instabilities in the essential spectrum or point spectrum, excluded by Hypothesis \ref{hyp: spreading speed}(ii)-(iii) and Hypothesis \ref{hyp: resonance} respectively, would not prohibit the construction here but instead render the approximate solution constructed in this section unstable.}

\subsection{Interior of the approximate solution}
Fix $\mu > 0$ small. To construct the approximate solution for $x \in (-\infty, (t+T)^\mu)$, we define
\begin{align}
\psi^- (x,t) = \omega(x + \zeta(t+T)) q_* (x + \zeta(t+T)), \label{e: psi minus def}
\end{align}
where we will choose the shift $\zeta(t+T)$ in Lemma \ref{l: pointwise matching} to match with a diffusive tail on the length scale $x \sim (t+T)^\mu$. \hl{The matching conditions will imply that $\zeta$ is smooth, with $\zeta(t +T) = \mathrm{O}((t+T)^{\mu-1/2})$ and $\dot{\zeta}(t+T) = \mathrm{O}((t+T)^{\mu - 3/2})$, and we therefore assume for the remainder of this section that these conditions hold}. Since we choose $T$ large, $\zeta$ will be small uniformly in $t$, and so we expect that $\psi^-(x,t) \approx \omega(x) q_*(x)$, which we make precise in the following lemma. 

\begin{lemma}\label{l: interior shifted front vs front}
	Fix $0 < \mu < \frac{1}{8}$ and let $r = 2 + \mu$. Assume that $\zeta \in C^1[T, \infty)$, with $\zeta(t+T) = \mathrm{O}((t+T)^{\mu-1/2})$ for $t$ large. There exists a constant $C > 0$ such that for any integer $0 \leq k \leq 2m$,
	\begin{align}
	\left\| \partial_x^k \left[ \omega(\cdot + \zeta(t+T)) q_*(\cdot + \zeta(t+T)) - \omega(\cdot) q_* (\cdot) \right] 1_{\{x \leq (t+T)^\mu + 1\}} \right\|_{L^\infty_r} \leq \frac{C}{(t+T)^{1/2-4\mu}} \label{e: interior shifted front vs front} 
	\end{align}
	for $T$ sufficiently large and for all $t > 0$. 
\end{lemma}
\begin{proof}
	First we set $k = 0$. We use the fundamental theorem of calculus to write 
	\begin{align*}
	\omega(x+\zeta(t+T)) q_* (x + \zeta(t+T)) - \omega(x) q_* (x) = \int_x^{x + \zeta(t+T)} (\omega q_*)' (y) \, d y. 
	\end{align*}
	Fix $L > 0$ large. For $- \infty<y \leq L + |\zeta(t+T)$|, $(\omega q_*)' (y)$ is bounded uniformly in $y, t$ and $T$ for $T$ sufficiently large, and hence for $x \leq L$, we have 
	\begin{align*}
	\rho_r (x) \left| \int_x^{x+\zeta(t+T)} (\omega q_*)' (y) \, d y \right| \leq C L^r |\zeta(t+T)| \leq \frac{C L^r}{(t+T)^{1/2-\mu}}.
	\end{align*}
	For $y \geq L - |\zeta(t+T)|$, we can use the front asymptotics \eqref{e: front asymptotics} to write 
	\begin{align*}
	(\omega q_*)'(y) = 1 + \mathrm{o}(1),
	\end{align*}
	where the $\mathrm{o}(1)$ terms are with respect to the limit $y \to \infty$. In particular, for $L \leq x \leq (t+T)^\mu + 1$, we have 
	\begin{align*}
	\langle x \rangle^r \left| \int_x^{x + \zeta(t+T)} (\omega q_*)' (y) \, d y \right| &\leq C \langle x \rangle^r | \zeta(t+T)| \\&\leq C \langle (t+T)^{\mu} \rangle^{r} (t+T)^{\mu - 1/2} 
	\leq C (t+T)^{3 \mu + \mu^2 - 1/2} 
	\leq C (t+T)^{4 \mu - 1/2},
	\end{align*}
	recalling that $r = 2 + \mu$ and that $0 < \mu < \frac{1}{8}$. Together with the above estimate for $x \leq L$, this completes the proof of \eqref{e: interior shifted front vs front} for $k = 0$. The estimates for the derivatives are similar. 
\end{proof}

\begin{lemma}\label{l: interior residual estimate}
	Fix $0 < \mu < \frac{1}{8}$ and let $r = 2 + \mu$.  Assume that $\zeta \in C^1[T, \infty)$, with $\zeta(t+T) = \mathrm{O}((t+T)^{\mu-1/2})$ and $\dot{\zeta} (t+T) = \mathrm{O}((t+T)^{\mu - 3/2})$ for $t$ large. There exists a constant $C > 0$ such that 
	\begin{align}
	\| \NL [\psi^-] (\cdot, t) 1_{\{x \leq (t+T)^\mu + 1\}} \|_{L^\infty_r} \leq \frac{C}{(t+T)^{1/2-4\mu}}.
	\end{align}
\end{lemma}
\begin{proof}
	We write 
	\begin{align*}
	\NL[\psi^-] = \NL [\omega q_*] + \NL[\psi^- - \omega q_*] + \omega f(\omega^{-1} \psi^-) - \omega f(q_*) - \omega f(\omega^{-1} \psi^- - q_*).
	\end{align*}
	Arguing as in the proof of Lemma \ref{l: interior shifted front vs front}, we see that there is a constant $C > 0$ such that 
	\begin{align*}
	| (\omega(x))^{-1} \psi^-(x,t) - q_*(x)| \leq C,
	\end{align*}
	for all $x \leq (t+T)^\mu +1$, $t > 0$, and $T$ sufficiently large. By Taylor's theorem, we then have 
	\begin{align}
	\omega | f(\omega^{-1} \psi^-) - f(q_*) | \leq C \omega  |\omega^{-1} \psi^- - q_*| = C |\psi^- - \omega q_*|. 
	\end{align}
	By Lemma \ref{l: interior shifted front vs front}, we obtain 
	\begin{align*}
	\| [\omega f(\omega^{-1} \psi^-(\cdot, t)) - \omega f(q_*)] 1_{\{ x \leq (t+T)^\mu +1\}} \|_{L^\infty_r} &\leq C \| (\psi^-(\cdot, t) - \omega q_*) 1_{\{ x \leq (t+T)^\mu +1\}} \|_{L^\infty_r} \\
	&\leq \frac{C}{(t+T)^{1/2-4\mu}}
	\end{align*}
	for all $t > 0$. The term $\NL[\psi^- - \omega q_*] - \omega f(\omega^{-1} \psi^- - q_*)$ contains only terms which are linear in $\psi^- - \omega q_*$. All terms involving spatial derivatives of $\psi^- - \omega q_*$ can be estimated by Lemma \ref{l: interior shifted front vs front}, while the estimate for the time derivative is similar (in fact improved by the better decay of $\dot{\zeta} (t+T)$ compared to $\zeta(t+T)$). Altogether, we obtain 
	\begin{align*}
	\| [\NL [\psi^- - \omega q_*] - \omega f(\omega^{-1} \psi^- - q_*)] 1_{\{x \leq (t+T)^\mu + 1\}} \|_{L^\infty_r} \leq \frac{C}{(t+T)^{1/2-4\mu}}. 
	\end{align*}
	It remains only to estimate the term $\NL[\omega q_*]$. Using the fact that $q_*$ solves the traveling wave equation 
	\begin{align}
	\mathcal{P}(\partial_x) q_* + c_* \partial_x q_* + f(q_*) = 0, 
	\end{align}
	we find after a short computation
	\begin{align}
	\NL [\omega q_*] = \frac{3}{2 \eta_* (t+T)} [\omega^2 (\omega^{-1})' q_* + \omega' q_* + \omega q_*']. \label{e: interior NL omega q}
	\end{align}
	Since $\omega(x) \equiv 1$ for $x \leq -1$, we have 
	\begin{align}
	\| \NL [\omega q_*] 1_{\{x \leq -1\}}\|_{L^\infty_r} = \left\| \frac{3}{2 \eta_* (t+T)} q_*' 1_{\{x \leq -1\}} \right\|_{L^\infty} \leq \frac{C}{t+T}. \label{e: interior omega q x leq -1 estimate}
	\end{align}
	Fix $L > 0$ large. Since the interval $[-1, L]$ is compact and the term in brackets in \eqref{e: interior NL omega q} is continuous in $x$, we have 
	\begin{align}
	\| \NL [\omega q_*] 1_{\{-1 \leq x \leq L\}} \|_{L^\infty_r} \leq \frac{C}{t+T}, \label{e: interior omega q x moderate estimate}
	\end{align}
	for some constant $C$ depending on $L$. For $x \geq L$, we use the front asymptotics \eqref{e: front asymptotics} to write %after a short computation
	\begin{align*}
	\NL [\omega q_*] (x,t) 1_{\{L \leq x \leq (t+T)^\mu + 1\}} = \frac{3}{2 \eta_* (t+T)} [1 - \eta_* a - \eta_* x + \mathrm{o} (1)],
	\end{align*} 
	where the $\mathrm{o}(1)$ terms are with respect to the limit $x \to \infty$. Hence for $L$ fixed, these terms are bounded for $x \geq L$, and the worst behaved term in the expression grows like $x$. We thereby obtain the estimate
	\begin{align*}
	\| \NL [\omega q_*] 1_{\{L \leq x \leq (t+T)^\mu + 1\}} \|_{L^\infty_r} &\leq \frac{C}{t+T} \| x^{r+1} 1_{\{L \leq x \leq (t+T)^\mu + 1\}} \|_{L^\infty} \\&\leq \frac{C}{t+T} (t+T)^{(r+1)\mu} 
	= \frac{C}{(t+T)^{1-3\mu - \mu^2}} 
	\leq \frac{C}{(t+T)^{1- 4 \mu}},
	\end{align*}
	recalling that $r = 2 + \mu$ and $0 < \mu < \frac{1}{8}$. Together with \eqref{e: interior omega q x leq -1 estimate} and \eqref{e: interior omega q x moderate estimate}, this completes the proof of the lemma. 
\end{proof}

\subsection{Approximate solution in the leading edge}
For $x \geq 1$, $\omega(x) = e^{\eta_* x}$, and so the equation $\NL [v] = 0$ reduces to 
\begin{align}
v_t - \mathcal{P}(\partial_x - \eta_*) v - \left( c_* - \frac{3}{2 \eta_* (t+T)} \right) \left(-\eta_* v + v_x \right) - e^{\eta_* x} f(e^{-\eta_* x} v) = 0. \label{e: leading edge x large eqn}
\end{align}
Since $f(0) = 0$, we have for $x \geq 1$ by Taylor's theorem 
\begin{align}
|f(e^{-\eta_* x} v) - f'(0) e^{-\eta_* x} v | \leq C e^{-2 \eta_* x} |v|^2 \label{e: leading edge nonlinearity quadratic estimate}
\end{align}
provided $e^{-\eta_* x} v$ is bounded there. We define the quadratic remainder $\tilde{f} (u) := f(u) - f'(0) u$ and 
\begin{align}
\mathcal{S} (\partial_x) : = \mathcal{P} (\partial_x - \eta_*) + c_* (\partial_x - \eta_*) + f'(0),
\end{align}
so that \eqref{e: leading edge x large eqn} becomes 
\begin{align}
v_t - \mathcal{S}(\partial_x) + \frac{3}{2 \eta_* (t+T)} (v_x - \eta_* v) - e^{\eta_* x} \tilde{f} (e^{-\eta_* x} v) = 0. 
\end{align}
Hypothesis \ref{hyp: spreading speed} implies that 
\begin{align}
\mathcal{S} (\partial_x) = \alpha \partial_{xx} + \sum_{k = 3}^{2m} \alpha_k \partial_x^k \label{e: leading edge weighted operator}
\end{align}
for some constants $\alpha_k \in \R$, with $(-1)^m\alpha_{2m} < 0$. That is, the dynamics near $x = \infty$ are \textit{essentially diffusive}: since large spatial scales are most relevant for the long time behavior here, the dynamics are governed by the lowest derivative $\alpha \partial_{xx}$. To make this precise, we introduce scaling variables 
\begin{align}
\tau = \log (t+T), \quad \xi = \frac{1}{\sqrt{\alpha}} \frac{x+x_0}{\sqrt{t+T}}, 
\end{align}
where $x_0$ is a shift to be chosen later. We introduce $\Psi(\xi, \tau) = \psi^+(x,t)$, so that the equation $\NL [\psi^+] = R$ for $x \geq 1$ becomes $\cNL [\Psi] = e^{\tau} \mathcal{R}$ for $\xi \geq (1+x_0)/\sqrt{\alpha(t+T)}$ where $\mathcal{R} (\xi, \tau) = R(x,t)$ and
\begin{multline}
\cNL[\Psi] = \Psi_\tau - e^{\tau} \mathcal{S} \left( \frac{1}{\sqrt{\alpha}} e^{-\tau/2} \partial_\xi \right) \Psi - \frac{1}{2} \xi \Psi_\xi - \frac{3}{2} \Psi + \frac{3}{2 \eta_* \sqrt{\alpha}} e^{-\tau/2} \Psi_\xi \\ - \exp[\tau + \eta_* (\sqrt{\alpha} e^{\tau/2} \xi - x_0)] \tilde{f} \left( \exp[-\eta_* (\sqrt{\alpha} e^{\tau/2} \xi - x_0)] \Psi \right). \label{e: leading edge eqn in scaling variables}
\end{multline}
Note that by \eqref{e: leading edge weighted operator}
\begin{align}
e^\tau \mathcal{S} \left( \frac{1}{\sqrt{\alpha}} e^{-\tau/2} \partial_\xi \right) = \partial_{\xi \xi} + \sum_{k = 3}^{2m} \frac{\alpha_k}{\alpha^{k/2}} e^{(1- k/2)\tau} \partial_\xi^k = \partial_{\xi \xi} + \mathrm{O}(e^{-\tau/2}). 
\end{align}
In addition, the nonlinearity in \eqref{e: leading edge eqn in scaling variables} is irrelevant. To see this, note first that by \eqref{e: leading edge nonlinearity quadratic estimate}, we have 
\begin{align*}
\left| \exp[\tau + \eta_* (\sqrt{\alpha} e^{\tau/2} \xi - x_0)] \tilde{f} \left( \exp[-\eta_* (\sqrt{\alpha} e^{\tau/2} \xi - x_0)]\Psi \right) \right| \leq C e^{\tau} \exp[-\eta_* (\sqrt{\alpha} e^{\tau/2} \xi - x_0)] \Psi^2,
\end{align*}
provided $e^{-\eta_* x} \psi$ is bounded for $x \geq 1$. We will only use this equation on the length scale $x \geq (t+T)^{\mu}$, which corresponds to $\xi \geq \alpha^{-1/2} \left(  e^{(\mu - 1/2) \tau} + x_0 e^{-\tau/2} \right)$. On this scale, we have the estimate 
\begin{align}
\exp [-\eta_* (\sqrt{\alpha} e^{\tau/2} \xi - x_0)] \leq \exp[-\eta_* e^{\mu \tau}], \label{e: leading edge nonlinearity irrelevant estimate}
\end{align}
so that for $\mu > 0$ fixed and for $\tau$ large, this factor dominates $e^{\tau}$, and the nonlinearity is exponentially small in $\tau$. Therefore, to leading order in $e^{\tau/2}$, the equation $\cNL [\Psi] = 0$ is 
\begin{align}
\Psi_\tau = \Psi_{\xi \xi} + \frac{1}{2} \xi \Psi_\xi + \frac{3}{2} \Psi, \label{e: leading edge scaling heat eqn}
\end{align}
revealing in which sense the dynamics are essentially diffusive: this is precisely the heat equation in self-similar variables. The spectrum of the operator 
\begin{align}
L_\Delta := \partial_{\xi \xi} + \frac{1}{2} \xi \partial_\xi + 1
\end{align}
is well known (see for instance \cite[Appendix A]{GallayWayne}), and we make use of this in order to construct expansions for solutions of \eqref{e: leading edge eqn in scaling variables}. 
As in \cite{Comparison3}, we make an ansatz 
\begin{align}
\Psi(\xi, \tau) = e^{\tau/2} \Psi_0 (\xi) + \Psi_1 (\xi). \label{e: leading edge ansatz}
\end{align}
We only need a solution defined for $\xi \geq \alpha^{-1/2} \left( e^{(\mu - 1/2) \tau} + x_0 e^{-\tau/2} \right) \geq 0$ to match with the interior solution, so we will consider the resulting equations for $\Psi_0$ and $\Psi_1$ on the half-line $\xi \geq 0$. Inserting the ansatz \eqref{e: leading edge ansatz} into \eqref{e: leading edge eqn in scaling variables} and collecting terms in powers of $e^{\tau/2}$ gives 
\begin{align}
L_\Delta \Psi_0 = 0 \label{e: leading edge V0 eqn}
\end{align}
and 
\begin{align}
\left( L_\Delta + \frac{1}{2} \right) \Psi_1 = \frac{3}{2 \eta_* \sqrt{\alpha}} \partial_\xi \Psi_0 - \frac{\alpha_3}{\alpha^{3/2}} \partial_\xi^3 \Psi_0. \label{e: leading edge V1 eqn}
\end{align}
The interior of the front provides a strong absorption effect since the spectrum of $\mcl^-$ is strictly contained in the left half plane.  We reflect this fact in the choice of Dirichlet boundary conditions $\Psi_i (0) = 0, i = 0, 1$. The unique solution $\Psi_0 \in L^2 (\R_+)$ to \eqref{e: leading edge V0 eqn} then is \cite[Appendix A]{GallayWayne}
\begin{align}
\Psi_0 (\xi) = \beta_0 \xi e^{-\xi^2/4} \label{e: leading edge V0 formula}
\end{align}
for a constant $\beta_0 \in \R$. 
If we posed these equations on the whole real line, the next eigenvalue for $L_\Delta$ would be at $\lambda = -\frac{1}{2}$, which would present an obstacle to solving \eqref{e: leading edge V1 eqn}. However, the restriction to the half-line with a Dirichlet boundary condition, equivalent to considering the equation on the real line with odd data, removes this eigenvalue since the corresponding eigenfunction is even; see again \cite[Appendix A]{GallayWayne}. One can further obtain Gaussian estimates on the solution to \eqref{e: leading edge V1 eqn}:  conjugating with Gaussian weight $e^{\xi^2/8}$ transforms $L_\Delta$ into the quantum harmonic oscillator with well known spectral properties; see e.g. \cite{Helffer}. We collect the relevant results in the following lemma. 
\begin{lemma}\label{l: leading edge V1 estimates}
	With $V_0$ given by \eqref{e: leading edge V0 formula}, there exists a smooth solution $V_1$ to \eqref{e: leading edge V1 eqn} such that for each natural number $n$, we have 
	\begin{align}
	|\partial_{\xi}^n \Psi_1 (\xi)| \leq C_n e^{-\xi^2 /8} \label{e: leading edge V1 estimates}
	\end{align}
	for all $\xi > 0$. Furthermore, all derivatives of $\Psi_1$ extend continuously to the boundary $\xi = 0$. 
\end{lemma}
Reverting to the original variables, the approximate solution $\Psi(\xi, \tau)$ we have constructed has the form 
\begin{align}
\psi^+(x,t) := \beta_0 \frac{1}{\sqrt{\alpha}} (x+x_0) e^{-(x+x_0)^2/[4\alpha(t+T)]} + \Psi_1 \left( \frac{1}{\sqrt{\alpha}} \frac{x+x_0}{\sqrt{t+T}} \right), \label{e: psi plus formula}
\end{align}
where $\beta_0$ and $x_0$ are parameters to be chosen, and $\alpha$ is the effective diffusivity from \eqref{e: leading edge weighted operator}. In particular, we have $|\psi^+(x,t)| \leq C \langle x \rangle$ uniformly for $t > 0, x > 0$, so that $e^{-\eta_* x} \psi^+(x,t)$ is uniformly bounded in $t$, and we can use the estimate \eqref{e: leading edge nonlinearity quadratic estimate} on the nonlinearity to write 
\begin{align}
\left| \tilde{f} \left( \exp [-\eta_* (\sqrt{\alpha} e^{\tau/2} \xi - x_0)] \Psi \right) \right| \leq C \exp\left[ -2 \eta_* (\sqrt{\alpha} e^{\tau/2} \xi - x_0) \right] |\Psi|^2,
\end{align}
for $\Psi$ as in \eqref{e: leading edge ansatz}. Together with \eqref{e: leading edge nonlinearity irrelevant estimate}, this implies 
\begin{align}
\exp[ \tau + \eta_* (\sqrt{\alpha} e^{\tau/2} \xi - x_0)] \left| \tilde{f} \left( \exp[-\eta_* (\sqrt{\alpha} e^{\tau/2} \xi - x_0) ] \Psi \right) \right| &\leq C \exp[\tau - \eta_* e^{\mu \tau}] |\Psi|^2 \nonumber \\ &\leq C e^{-3 \tau/2} |\Psi|^2 \label{e: leading edge nonlinearity irrelevant estimate 2}
\end{align}
for $\tau$ sufficiently large, which we can guarantee by choosing $T$ large. 
Inserting our ansatz \eqref{e: leading edge ansatz} into \eqref{e: leading edge eqn in scaling variables}, we obtain 
\begin{multline}
\cNL[\Psi] = - \sum_{k=4}^{2m} \frac{\alpha_k}{\alpha^{k/2}} e^{\left(\frac{3}{2} - \frac{k}{2} \right)\tau} \partial_\xi^k \Psi_0 - \sum_{k=3}^{2m} \frac{\alpha_k}{\alpha^{k/2}} e^{\left(1 - \frac{k}{2} \right) \tau} \partial_\xi^k \Psi_1 + \frac{3}{2 \eta_* \sqrt{\alpha}} e^{-\tau/2} \partial_\xi \Psi_1 \\
- \exp[ \tau + \eta_* (\sqrt{\alpha} e^{\tau/2} \xi - x_0)] \tilde{f} \left( \exp[-\eta_* (\sqrt{\alpha} e^{\tau/2} \xi - x_0) ] \Psi \right).
\end{multline}
The important observation is that every term carries a factor of at least $e^{-\tau/2}$ and has Gaussian localization in $\xi$. More precisely, by Lemma \ref{l: leading edge V1 estimates}, the formula \eqref{e: leading edge V0 formula}, and the estimate \eqref{e: leading edge nonlinearity irrelevant estimate 2}, there exists a constant $C > 0$ such that 
\begin{align}
| \cNL [\Psi] (\xi, \tau) | \leq C e^{-\tau/2} e^{-\xi^2/8}
\end{align}
for $T$ large and for $\xi \geq \alpha^{-1/2} (e^{(\mu - 1/2) \tau} + x_0 e^{-\tau/2})$. In the original variables, one finds $\NL [\psi^+] = e^{-\tau} \cNL [\Psi]$, so that for $x \geq (t+T)^\mu$ and for $T$ large,
\begin{align}
| R^+(x,t) | \leq \frac{C}{(t+T)^{3/2}} e^{-(x+x_0)^2/[8\alpha (t+T)]}
\end{align} 
where $\NL [\psi^+] = R^+$. This pointwise estimate on the residual implies the following estimates in norm. 
\begin{lemma}\label{l: leading edge residual estimate}
	Fix $0 < \mu < \frac{1}{8}$ and set $r = 2 + \mu$. There exists a constant $C > 0$ so that 
	\begin{align}
	\| \NL [\psi^+] 1_{\{ x \geq (t+T)^\mu\}} \|_{L^\infty_r} \leq \frac{C}{(t+T)^{1/2-\mu/2}},
	\end{align}
	and \begin{align}
	\| \NL [\psi^+] 1_{\{ x \geq (t+T)^\mu\}} \|_{L^1_1} \leq \frac{C}{(t+T)^{1/2}}. 
	\end{align}
	for all $t > 0$ and for $T$ sufficiently large. 
\end{lemma}

\subsection{Matching the interior to the leading edge}

We construct the full approximate solution $\psi$ by matching $\psi^-$ and $\psi^+$ at the length scale $x \sim (t+T)^\mu$. Simply matching values  at $x = (t+T)^{\mu}$ through choosing the shift $\zeta(t+T)$ would leave us with mismatched derivatives and distributions in the residual.  To avoid this, we smoothly blend solutions over the region $x \in [(t+T)^\mu, (t+T)^\mu + 1]$. Therefore, consider the smooth cutoff $\chi_0$ with $0 \leq \chi_0 \leq 1$ and 
\begin{align}
\chi_0 (x) = \begin{cases}
1, &x \leq 0, \\
0, &x \geq 1. 
\end{cases}
\end{align}
We then define a time-varying smoothed cutoff by 
\begin{align}
\chi(x,t) = \chi_0 (x - (t+T)^\mu),
\end{align}
and define our approximate solution 
\begin{align}
\psi(x,t) = \chi(x,t) \psi^-(x,t) + (1-\chi(x,t)) \psi^+(x,t). 
\end{align}
The remainder of this section is dedicated to proving that $\psi$ satisfies $\NL [\psi] = R$ for a small residual $R$, provided we choose the constants $\beta_0$ and $x_0$ appropriately. 

\begin{prop}\label{p: residual estimate}
	Let $\beta_0 = \sqrt{\alpha}$ and let $x_0 = a$ from \eqref{e: front asymptotics}. Fix $0 < \mu < \frac{1}{8}$ and let $r = 2 + \mu$. There exists a shift $\zeta(t+T)$ and a constant $C > 0$ so that 
	\begin{align}
	\| R(\cdot, t; T) \|_{L^\infty_r} \leq \frac{C}{(t+T)^{1/2 - 4 \mu}}, \label{e: residual L inf r estimate}
	\end{align}
	for all $t> 0$ and $T$ sufficiently large, where $R(\cdot, t; T) = \NL[\psi](\cdot, t)$. 
\end{prop}
To prove this proposition, we write 
\begin{align}
\NL [\psi] = \chi \NL [\psi^-] + (1-\chi) \NL [\psi^+] + [\NL, \chi] (\psi^- - \psi^+),
\end{align}
where $[\NL, \chi]$ is the commutator defined by $[\NL, \chi] \psi = \NL[\chi \psi] - \chi \NL [\psi]$. The terms $\chi \NL[\psi^-]$ and $(1-\chi) \NL [\psi^+]$ satisfy \eqref{e: residual L inf r estimate} by Lemmas \ref{l: interior residual estimate} and \ref{l: leading edge residual estimate}, so we only need to prove \eqref{e: residual L inf r estimate} for the term involving the commutator. Since $f(0) = f(1) = 0$, the support of this commutator is contained in $\{ (x,t) : x \in [(t+T)^\mu, (t+T)^\mu + 1]\}$ by construction of $\chi$. We first match the values of $\psi^-$ and $\psi^+$ precisely at $x = (t+T)^\mu$, as in \cite{Comparison3}. 

\begin{lemma} \label{l: pointwise matching}
	Let $\beta_0 = \sqrt{\alpha}$ and let $x_0 = a$ from \eqref{e: front asymptotics}, fix $0 < \mu < \frac{1}{8}$. There exists a shift $\zeta(t+T)$ so that $\zeta(t+T) = \mathrm{O}((t+T)^{\mu - 1/2})$, $\dot{\zeta} (t+T) = \mathrm{O}((t+T)^{\mu - 3/2})$, and 
	\begin{align}
	\psi^-((t+T)^\mu, t) = \psi^+ ((t+T)^\mu, t), \label{e: matching}
	\end{align}
	for $t > 0$ and $T$ large. 
	Moreover, the derivatives satisfy for $t > 0$ and $T$ large
	\begin{align}
	|\psi^-_x ((t+T)^\mu, t) - \psi^+_x((t+T)^\mu, t)| = \mathrm{O}((t+T)^{-1/2}). \label{e: matching derivatives}
	\end{align}
	and 
	\begin{align}
	|\partial_x^k \psi^\pm (x, t) | \leq \frac{C}{(t+T)^{1/2}} 
	\end{align}
	for any integer $2 \leq k \leq 2m$ and for any $x \in [(t+T)^\mu, (t+T)^\mu + 1]$. 
\end{lemma}
\begin{proof}
	Setting $x = (t+T)^\mu$ in \eqref{e: psi plus formula} and Taylor expanding the exponential, we find 
	\begin{align}
	\psi^+ ((t+T)^\mu, t) = \frac{\beta_0}{\sqrt{\alpha}} [(t+T)^\mu + x_0] [ 1 + \mathrm{O}(t+T)^{-1 + 2 \mu}] + \Psi_1 \left( \frac{1}{\sqrt{\alpha}} \frac{(t+T)^\mu + x_0}{\sqrt{t+T}} \right). \label{e: matching psi plus expansion}
	\end{align}
	Since $\Psi_1(\xi)$ is smooth for $\xi \geq 0$ by Lemma \ref{l: leading edge V1 estimates} and $\Psi_1 (0) = 0$, we have $|\Psi_1 (\xi)| \leq C \xi$ for $\xi> 0$ small, and hence
	\begin{align}
	\psi^+ ((t+T)^\mu, t) = \frac{\beta_0}{\sqrt{\alpha}} [x_0 + (t+T)^\mu] + \mathrm{O}((t+T)^{\mu - 1/2}). 
	\end{align}
	Using the front asymptotics \eqref{e: front asymptotics} in the definition \eqref{e: psi minus def} of $\psi^-$, we obtain 
	\begin{align}
	\psi^-((t+T)^\mu, t) = a + (t+T)^\mu + \zeta(t+T) + \mathrm{O} \left( e^{-\eta_0 (t+T)^\mu} \right). \label{e: matching psi minus expansion} 
	\end{align}
	Choosing $\beta_0 = \sqrt{\alpha}$ and $x_0 = a$ therefore matches the leading order terms in this expansion, so that 
	\begin{align}
	\psi^- ((t+T)^\mu, t) - \psi^+ ((t+T)^\mu, t) = \zeta(t+T) + \mathrm{O} ((t+T)^{\mu - 1/2}). 
	\end{align}
	\hl{Therefore, we can choose $\zeta(t+T),$ smooth, to cancel the remainder while ensuring $\zeta(t+T) = \mathrm{O}((t+T)^{\mu - 1/2})$. More precisely, we see from a short calculation following the above expansions that the map
	\begin{align*}
	    \zeta(\cdot + T) \mapsto \psi^+((\cdot+T)^\mu, \cdot) - \psi^- ((\cdot+T)^\mu, \cdot) + \zeta(\cdot+T) : M \to M
	\end{align*}
	is a contraction, where 
	\begin{align*}
	    M = \{ \zeta(\cdot + T) \in C^1 [0, \infty) : |\zeta(t+T)| \leq C(t+T)^{\mu - 1/2}, \, |\dot{\zeta} (t+T)| \leq C (t+T)^{\mu - 3/2} \text{ for } t \geq 0 \}
	\end{align*}
	for $T$ sufficiently large and some constant $C > 0$ (notice that $\psi^-$ depends on $\zeta$, so that the dependence on $\zeta$ in this map is nontrivial). We then define $\zeta(t+T)$ to be the unique fixed point of this map, and find that under this choice we have $\psi^+ ((t+T)^\mu, t)=\psi^-((t+T)^\mu, t)$, as desired. }

	Evaluating the derivatives of $\psi^\pm$ at $x = (t+T)^\mu$, one finds 
	\begin{align*}
	\psi^+_x ((t+T)^\mu, t) = 1 + \mathrm{O}((t+T)^{-1/2}).\qquad 
% 	\end{align*}
% 	and 
% 	\begin{align*}
	\psi^-_x ((t+T)^\mu, t) = 1 + \mathrm{O} \left( e^{-\eta_0 (t+T)^\mu} \right),
	\end{align*}
	and hence \eqref{e: matching derivatives} holds. The estimates on higher derivatives follow from similar considerations --- we point out that the smoothness of $\Psi_1$ up to the boundary from Lemma \ref{l: leading edge V1 estimates} is essential here. 
\end{proof}

\begin{remark}
    This matching procedure forces the choice of coefficient of the logarithmic shift. Indeed, if instead of $-\frac{3}{2 \eta_*} (\log (t+T) - \log(T)$, we consider an arbitrary shift $-\frac{c_0}{\eta_*} (\log (t+T) - \log (T)$ with $c_0 \in \R$, then the diffusive equation governing dynamics in the leading edge \eqref{e: leading edge scaling heat eqn} is replaced by 
    \begin{align*}
    W_\tau = W_{\xi \xi} + \frac{1}{2} \xi W_\xi + c_0 W. 
    \end{align*}
    Solutions to this equation are related to solutions of \eqref{e: leading edge scaling heat eqn} through 
    \begin{align}
        W(\xi, \tau) = e^{(\frac{3}{2} - c_0) \tau} V(\xi, \tau), 
    \end{align}
    so that we cannot match a diffusive tail with the interior of the front solution  when $c_0 \neq \frac{3}{2}$:  the analogue of \eqref{e: matching psi plus expansion} would no longer be on the same order as the interior solution, \eqref{e: matching psi minus expansion}. 
\end{remark}

\begin{proof}[Proof of Proposition \ref{p: residual estimate}]
	It remains to show that  \eqref{e: residual L inf r estimate} holds for $[\NL, \chi] (\psi^- - \psi^+)$. While there are many terms in this expression, and writing all of them would be unwieldy, every term is either: 
    \begin{itemize}
	    \setlength\itemsep{0em}
	    \item the term $(\partial_t \chi) (\psi^- - \psi^+)$, 
		\item the term $\omega f(\omega^{-1} [\chi(\psi^- - \psi^+)]) - \chi \omega f(\omega^{-1} (\psi^- - \psi^+))$,
		\item or some smooth bounded $x$-dependent coefficient multiplied by $(\partial_x^k \chi) \partial_x^\ell (\psi^- - \psi^+)$, with $k + \ell \leq 2m$, possibly also with a factor of $(t+T)^{-1}$. 
	\end{itemize}
	For any term of the third type, we note by construction of $\chi$, all derivatives of $\chi$ up to order $2m$ are uniformly bounded and are supported on $x \in [(t+T)^\mu, (t+T)^\mu +1]$. To estimate these terms, it therefore suffices to get good estimates on $\partial_x^\ell (\psi^- - \psi^+)$ for $x \in [(t+T)^\mu, (t+T)^\mu +1]$. 
	
	For $\ell = 0$, we Taylor expand, using the control of derivatives from Lemma \ref{l: pointwise matching} to write 
	\begin{align*}
	\psi^- (x,t) &= \psi^- ((t+T)^\mu, t) + \psi_x^- ((t+T)^\mu, t) (x-(t+T)^\mu) + \mathrm{O} \left( (t+T)^{-1/2} \right), \\
	 \psi^+(x,t) &= \psi^+((t+T)^\mu, t) + \psi_x^+ ((t+T)^\mu, t) (x- (t+T)^\mu) + \mathrm{O} \left( (t+T)^{-1/2} \right)
	\end{align*}
	for $(t+T)^\mu \leq x \leq (t+T)^\mu +1$. Hence by Lemma \ref{l: pointwise matching}, we have 
	\begin{align}
	| \psi^- (x,t) - \psi^+ (x,t) | \leq \frac{C}{(t+T)^{1/2}} \label{e: matching estimate blending region}
	\end{align}
	for $t > 0$, $T$ large, and $(t+T)^\mu \leq x \leq (t+T)^\mu +1$. We similarly obtain by Taylor expanding
	\begin{align*}
	|\psi^-_x (x,t) - \psi^+_x (x,t) | \leq \frac{C}{(t+T)^{1/2}}
	\end{align*}
	for $t > 0$, $T$ large, and $(t+T)^\mu \leq x \leq (t+T)^\mu +1$. Higher derivatives in this region are bounded by $C(t+T)^{-1/2}$ by Lemma \ref{l: pointwise matching}, so if $b(x)$ is any smooth bounded function and $k$, $\ell$ are non-negative integers such that $k + \ell \leq 2m$, we obtain 
	\begin{align*}
	\| b(\cdot) (\partial_x^k \chi) \partial_x^\ell (\psi^- - \psi^+) \|_{L^\infty_r} &\leq \frac{C}{(t+T)^{1/2}} \sup_{(t+T)^\mu \leq x \leq (t+T)^\mu +1} | \langle x \rangle^{2 + \mu} | \\
	&\leq \frac{C}{(t+T)^{1/2}} \langle (t+T) \rangle^{\mu (2+\mu)} \\
	&\leq \frac {C}{(t+T)^{1/2 - 4 \mu}},
	\end{align*}
	as desired. The estimates on the term involving $\partial_t \chi$ are similar --- in fact, they are improved since we gain a factor of $(t+T)^{\mu - 1}$ when differentiating $\chi$ in time. 
	
	For the term involving the nonlinearity, since $\omega^{-1} (\psi^- - \psi^+)$ is uniformly bounded and $f(0) = 0$, we can Taylor expand the nonlinearity to obtain an estimate
	\begin{align*}
	| \omega f(\omega^{-1} [\chi (\psi^- - \psi^+)]) - \chi \omega f(\omega^{-1} (\psi^- - \psi^+)) | \leq C 1_{\{(t+T)^\mu \leq x \leq (t+T)^\mu +1\}} |\psi^- - \psi^+| 
	\end{align*}
	Hence by estimate \eqref{e: matching estimate blending region} on $\psi^- - \psi^+$ in the region of interest, we obtain 
	\begin{align*}
	\| \omega f(\omega^{-1} [\chi (\psi^- - \psi^+)]) - \chi \omega f(\omega^{-1} (\psi^- - \psi^+)) \|_{L^\infty_r} &\leq \frac{C}{(t+T)^{1/2}} \sup_{(t+T)^\mu \leq x \leq (t+T)^\mu + 1} \langle x \rangle^{2 + \mu} \\
	&\leq \frac{C}{(t+T)^{1/2 - 4 \mu}},
	\end{align*}
	which completes the proof of the proposition. 
\end{proof}

Finally, we collect for later use the following result, which says that, for large times, $\psi$ is well-approximated by $\omega q_*$. The bulk of the work is in proving Lemma \ref{l: interior shifted front vs front} --- the estimates in the leading edge are simple in comparison, so we omit the proof. 

\begin{lemma}\label{l: front versus psi}
	Let $r = 2 + \mu$ with $0 < \mu < \frac{1}{8}$. There exists a constant $C > 0$ such that 
	\begin{align}
	\| q_* - \omega^{-1} \psi (\cdot, t) \|_{L^\infty_{r+1}} \leq \frac{C}{(t+T)^{1/2-4\mu}}. 
	\end{align}	
	for all $t > 0$. 
\end{lemma}

\subsection{Outlook: approximating the true solution} \label{s: outlook}
The goal of the remainder of the paper is to prove that $\psi$ is a good approximation to the true solution $u$, in an appropriate sense. We let $v$ solve $\NL [v] = 0$, where $\NL$ is given by \eqref{e: NL def}. This is simply the original equation \eqref{e: eqn}, in the co-moving frame with the logarithmic delay, and with the exponential weight $\omega$. We then let $w = v - \psi$, i.e. we view $v$ as a perturbation of $\psi$. We must then control the perturbation $w$, which solves
\begin{align}
	w_t = \mcl w - f'(q_*) w - \frac{3}{2 \eta_* (t+T)} \left[ \omega (\omega^{-1})' w + w_x \right] + \omega f(\omega^{-1} w + \psi) - \omega f(\omega^{-1} \psi) - R,
\end{align}
where $R = \NL [\psi]$. Note that we have introduced a term $f'(q_*) w$ so that the principal linear part of this equation has the form $\mcl w$. We then define 
\begin{align}
	N( \omega^{-1} w) = f (\omega^{-1} (w + \psi)) - f(\omega^{-1} \psi) - f'(\omega^{-1} \psi) \omega^{-1} w, 
\end{align}
so that the equation for $w$ becomes 
\begin{align} 
	w_t = \mcl w - \frac{3}{2 \eta_* (t+T)} \left[ \omega (\omega^{-1})' w + w_x \right] + (f'(\omega^{-1} \psi) - f'(q_*)) w - R + \omega N(\omega^{-1} w). \label{e: w eqn}
\end{align}
We would like to view $w_t = \mcl w$ as the principal part of this equation, so that we can control $w$ by estimating nonlinear terms in the variation of constants formula with sharp bounds on the linear semigroup $e^{\mcl t}$. This is not yet possible at this level, however: as suggested in the discussion of the reduced model problem in Section \ref{s: overview}, the term $-\frac{3}{2 \eta_*} (t+T)^{-1} \omega (\omega^{-1})' w$ is \emph{critical}, and prevents the solution from decaying in time. To unravel this, we introduce $z = (t+T)^{-3/2} w$ and find that $z$ solves
\begin{multline}
	z_t = \mcl z - \frac{3}{2 \eta_* (t+T)} [\omega (\omega^{-1})' + \eta_*] z - \frac{3}{2 \eta_* (t+T)} z_x + [f'(\omega^{-1} \psi) - f'(q_*)] z  - (t+T)^{-3/2} R \\  + (t+T)^{-3/2} \omega N (\omega^{-1} (t+T)^{3/2} z). \label{e: z eqn1}
\end{multline}
Note that $\omega(\omega^{-1})' (x) + \eta_* \equiv 0 $ for $x > 1$, so that this term is essentially removed, which ultimately allows us to view this equation as a perturbation $z_t = \mcl z + H(t) [z]$ of the linear equation $z_t = \mcl z$, with $H(t) [z]$ encoding the nonlinear and non-autonomous terms in \eqref{e: z eqn1}. Our goal is then to obtain sharp temporal decay of $z$, equivalent to boundedness of $w$, by controlling all terms in the variation of constants formula
\begin{align*}
	z(t) = e^{\mcl t} z_0 + \int_0^t e^{\mcl (t-s)} H(s) [z(s)] \, ds. 
\end{align*}
Closing this argument requires sharp estimates on the decay of the semigroup $e^{\mcl t}$, which we obtain in the following two sections by a detailed analysis of the resolvent $(\mcl - \lambda)^{-1}$. 

\section{Resolvent estimates}\label{s: resolvent estimates}
We obtain the linear estimates necessary for our analysis through sharp estimates on the resolvent of the weighted linearization $\mcl$ near its essential spectrum. We start with some preliminary spectral theory. We say the essential spectrum of an operator $A$ is the set of $\lambda \in \C$ such that $A - \lambda$ is not a Fredholm operator of index 0. The following discussion applies in particular on $L^p(\R)$, $1 \leq p \leq \infty$ or on any algebraically weighted $L^p$ space. Fredholm properties of the linearization $\mathcal{A} = \mathcal{P}(\partial_x) + c_* \partial_x + f'(q_*)$ are determined by the asymptotic dispersion curves \eqref{e: right dispersion curve} and \eqref{e: Sigma minus}; see \cite{FiedlerScheel, KapitulaPromislow} for background. Specifically, $\mathcal{A} - \lambda$ is Fredholm if and only if $\lambda \notin \Sigma^+ \cup \Sigma^-$ and $\mathcal{A} - \lambda$ has index zero if $\lambda$ is to the right of these curves. 

Exponential weights change the asymptotic dispersion relations and thereby move the Fredholm borders. As a result, the Fredholm borders of the weighted linearization $\mcl$ are given by
\begin{align*}
\Sigma^+_{\eta_*} = \{ d^+ (\lambda, ik - \eta_*) = 0 \text{ for some } ik \in \R \} \label{e: Sigma plus eta}, \qquad \Sigma^-_{\eta_*}=\Sigma^-.
\end{align*}
Hypothesis \ref{hyp: spreading speed} and Hypothesis \ref{hyp: stable on left} then imply that the essential spectrum of $\mcl$ is marginally stable, as depicted in Figure \ref{f: spectrum and initial data}. 

We start by proving estimates for the resolvent $(\mcl^+-\lambda)^{-1}$ of the limiting operator at $x = \infty$ near its essential spectrum $\Sigma^+_{\eta_*}$. Hypothesis \ref{hyp: spreading speed} implies that the dispersion relation has a branch point at $\lambda = 0$, which we unfold through $\gamma := \sqrt{\lambda}$ with branch cut along the negative real axis.
\subsection{Estimates on the asymptotic resolvent}
We analyze the asymptotic resolvent through its integral kernel $G^+_\gamma$, which solves 
\begin{align}
(\mcl^+ - \gamma^2) G^+_\gamma = - \delta_0. 
\end{align}
Since $\mcl^+$ is a constant coefficient differential operator (defined by \eqref{e: Lplus def}), the solution to $(\mcl^+ - \gamma^2) u = g$ is given through convolution with $G^+_\gamma$. In \cite[Section 2]{AveryScheel}, we give a detailed description of this resolvent kernel, and use this description to prove that the asymptotic resolvent is Lipschitz in $\gamma$ near the origin in an appropriate sense. 

\begin{prop}\label{p: asymptotic resolvent Lipschitz estimate}
	Let $r > 2$. There exist positive constants $C$ and $\delta$ and a limiting operator $R_0^+$ such that for any odd function $g \in L^1_{1,1} (\R)$, we have
	\begin{align}
	\| (\mcl^+ - \gamma^2)^{-1} g - R_0^+ g \|_{W^{2m-1, 1}_{-r, -r}} \leq C |\gamma| \| g \|_{L^1_{1,1}},
	\end{align}
	and 
	\begin{align}
	\| (\mcl^+ - \gamma^2)^{-1} g - R_0^+ g \|_{W^{2m-1, \infty}_{-1,-1}} \leq C |\gamma| \| g ||_{L^1_{1,1}}
	\end{align}
	for all $\gamma \in B(0,\delta)$ such that $\gamma^2$ is to the right of $\Sigma_{\eta_*}^+$. 
\end{prop}
This is essentially Proposition 2.1 of \cite{AveryScheel}, although there we state the result for $L^2$-based spaces only. The proof carries over with minor modifications, so we do not give the full details here. We will use the same general approach to prove further estimates which translate to improved time decay for derivatives, and therefore we outline the overall strategy in the following. \hl{We note that since proposition only involves the asymptotic operator $\mcl^+$, which can be defined by \eqref{e: Lplus def} independently of the existence of a critical front, Proposition \ref{p: asymptotic resolvent Lipschitz estimate} relies only on Hypothesis \ref{hyp: spreading speed}.} Our approach is based on a decomposition of the resolvent kernel in which the principal piece resembles the resolvent kernel for the heat equation. To arrive at this decomposition, we write $(\mcl^+ - \gamma^2)u = g$ as a first order system in $U = (u, \partial_x u, ..., \partial_x^{2m-1} u)$, which has the form
\begin{align}
\partial_x U = M (\gamma) U + F, \label{e: asymptotic resolvent system}
\end{align}
where $F = (0, 0, ..., g)^T$ and $M(\gamma)$ is a $2m$-by-$2m$ matrix which is analytic (in fact a polynomial) in $\gamma$. The structure of this matrix implies 
\begin{align}
\det ((M(\gamma) - \nu I)) = d^+(\gamma^2, \nu - \eta_*), 
\end{align}
so that eigenvalues of $M(\gamma)$, which we call \textit{spatial eigenvalues}, correspond with roots of the dispersion relation. Using Fourier transform and asymptotics for large $\gamma>0$, we see that, for  $\gamma^2$ to the right of the essential spectrum, $M(\gamma)$ is a hyperbolic matrix with stable and unstable subspaces $E^\mathrm{s}(\gamma)$ and $E^\mathrm{u}(\gamma)$ satisfying $\dim E^\mathrm{s}(\gamma) = \dim E^\mathrm{u}(\gamma)$. We let $P^\mathrm{s}(\gamma)$ and $P^\mathrm{u}(\gamma) = 1 - P^\mathrm{s}(\gamma)$ denote the corresponding spectral projections, which are analytic for $\gamma^2$ to the right of the essential spectrum. We can use these projections to write the matrix Green's function for the system \eqref{e: asymptotic resolvent system} as in \cite{HolzerScheelPointwiseGrowth}
\begin{align}
T_\gamma (x) = \begin{cases}
- e^{M(\gamma) x} P^\mathrm{s}(\gamma), & x \geq 0, \\
e^{M(\gamma) x} P^\mathrm{u}(\gamma), & x < 0. 
\end{cases}
\end{align}
We can then recover the scalar resolvent kernel $G_\gamma^+$ through the formula
\begin{align}
G^+_\gamma = P_1 T_\gamma Q_1 p_{2m}^{-1},
\end{align} 
where $P_1$ is the projection onto the first component and $Q_1$ is the embedding into the last component, i.e. $P(u_1, ..., u_{2m}) = u_1$ and $Q_1 g = (0, ..., 0, g)^T$; see \cite{HolzerScheelPointwiseGrowth, AveryScheel} for details. 

Following \cite{HolzerScheelPointwiseGrowth}, we conclude that singularities of $G_\gamma^+$ are determined precisely by singularities of the stable projection $P^\mathrm{s}(\gamma)$. Hypothesis \ref{hyp: spreading speed} implies that the dispersion relation $d^+(\gamma^2, \nu^\pm - \eta_*)$ has two roots of the form
\begin{align}
\nu^\pm (\gamma) = \pm \nu_0 \gamma+ \mathrm{O}(\gamma^2) \label{e: nu pm expansion}
\end{align}
for $\gamma$ close to zero, and that all other roots are bounded away from zero for $\gamma$ small. In particular, $\nu^\pm (\gamma)$ collide as $\gamma \to 0$ and form a Jordan block for $M(0)$ to the eigenvalue 0 \cite{HolzerScheelPointwiseGrowth}. Hence, $P^\mathrm{s}(\gamma)$ necessarily has a singularity at $\gamma = 0$ \cite{Kato}. We isolate this singularity by splitting 
\begin{align}
P^\mathrm{s/u} (\gamma) = P^\mathrm{cs/cu} (\gamma) + P^\mathrm{ss/uu}(\gamma),
\end{align}
where $P^\mathrm{cs}(\gamma)$ is the spectral projection associated to the eigenvalue $\nu^-(\gamma)$, and $P^\mathrm{ss}(\gamma)$ is the strong stable projection associated to the rest of the stable eigenvalues. Since the other stable eigenvalues are bounded away from the imaginary axis for $\gamma$ small, $P^\mathrm{ss}(\gamma)$ is analytic in $\gamma$ near the origin, and only $P^\mathrm{cs}(\gamma)$ is singular. Similarly, $P^\mathrm{cu} (\gamma)$ is the spectral projection associated to $\nu^+(\gamma)$, and the strong unstable projection $P^\mathrm{uu} (\gamma)$ is analytic near $\gamma = 0$. 

In \cite[Lemma 2.2]{AveryScheel}, we characterize the singularity of $P^\mathrm{cs}(\gamma)$ at $\gamma = 0$ using Lagrange interpolation to write $P^\mathrm{cs}(\gamma)$ as a polynomial in $M(\gamma)$, and we find that the singularity is a simple pole such that $\gamma P^\mathrm{cs} (\gamma) \big|_{\gamma = 0} = - \gamma P^\mathrm{cu} (\gamma) \big|_{\gamma =0}$. We let $\tilde{P}^{\mathrm{cs/cu}} (\gamma)$ denote the analytic part of $P^\mathrm{cs/cu}(\gamma)$ which remains after subtracting  this pole. We arrive at the following decomposition of the resolvent kernel: 
\begin{align}
G^+_\gamma = G^\mathrm{c}_\gamma + \tilde{G}^\mathrm{c}_\gamma + G^\mathrm{h}_\gamma, 
\end{align}
where $G^\mathrm{c}_\gamma$ is the principal part associated to the pole at $\gamma = 0$, 
\begin{align}
G^\mathrm{c}_\gamma (x) = \begin{cases}
\frac{\beta_c}{\gamma} e^{\nu^-(\gamma) x}, & x \geq 0, \\
\frac{\beta_c}{\gamma} e^{\nu^+(\gamma) x}, & x < 0,
\end{cases} \label{e: G c def}
\end{align}
for some constant $\beta_c \neq 0$, $\tilde{G}^\mathrm{c}_\gamma$ is a remainder term associated to $\tilde{P}^{\mathrm{cs/cu}} (\gamma)$, 
\begin{align}
\tilde{G}^\mathrm{c}_\gamma (x) = \begin{cases}
- p_{2m}^{-1} e^{\nu^- (\gamma) x} P_1 \tilde{P}^{\mathrm{cs}} (\gamma) Q_1, & x \geq 0, \\
p_{2m}^{-1} e^{\nu^+ (\gamma) x} P_1 \tilde{P}^\mathrm{cu} (\gamma) Q_1, & x < 0, 
\end{cases} \label{e: G tilde c def}
\end{align}
and $G^\mathrm{h}_\gamma$ is a remainder term associated to the strongly hyperbolic projections $P^\mathrm{ss/uu}(\gamma)$,
\begin{align}
G^\mathrm{h}_\gamma (x) = \begin{cases}
- p_{2m}^{-1} P_1 e^{M(\gamma) x} P^\mathrm{ss} (\gamma) Q_1, & x \geq 0, \\
p_{2m}^{-1} P_1 e^{M(\gamma) x} P^\mathrm{uu} (\gamma) Q_1, & x < 0. 
\end{cases} \label{e: G h def}
\end{align}

Our goal in the remainder of this section is to use this decomposition of the resolvent kernel to prove the following estimates on derivatives of solutions to $(\mcl^+ - \gamma^2) u = g$. 

\begin{prop}\label{p: asymptotic resolvent derivative estimate}
	Let $r > 2$. There exist positive constants $C$ and $\delta$ such that for any odd function $g \in L^\infty_{r,r} (\R)$, we have 
	\begin{align}
	\| \partial_x (\mcl^+ - \gamma^2)^{-1} g \|_{L^1_{1,1}} \leq \frac{C}{|\gamma|} \| g \|_{L^\infty_{r,r}},
	\end{align}
	and 
	\begin{align}
	\| \partial_x (\mcl^+ - \gamma^2)^{-1} g \|_{L^\infty_{r,r}} \leq \frac{C}{|\gamma|^{r-1}} \| g \|_{L^\infty_{r,r}},
	\end{align}
	for all $\gamma \in B(0, \delta)$ such that $\Re \gamma \geq \frac{1}{2}|\Im \gamma|$.  
\end{prop}
\hl{We note again that Proposition \ref{p: asymptotic resolvent derivative estimate} relies only on Hypothesis \ref{hyp: spreading speed}, since it only involves the linearization about $u \equiv 0$}. The assumption that $g$ is odd in Propositions \ref{p: asymptotic resolvent Lipschitz estimate} and \ref{p: asymptotic resolvent derivative estimate} models the absorption effect in the wake of the front due to the fact that the spectrum of the linearization at $u \equiv 1$ is strictly contained in the left half plane. Indeed, for a second order equation, $\mcl^+ = \partial_{xx}$, and so oddness is equivalent to imposing a Dirichlet boundary condition. Although the full operator does not necessarily commute with reflections, we can exploit oddness in Section \ref{s: full resolvent} when extending estimates to the full resolvent $(\mcl - \gamma^2)^{-1}$. 

The main work in proving Proposition \ref{p: asymptotic resolvent derivative estimate} is to prove the estimate for the piece of the resolvent given by convolution with $G^\mathrm{c}_\gamma$, since this is the worst behaved term from the perspective of $\gamma$ dependence. Indeed, it is clear that $|\tilde{G}^\mathrm{c}_\gamma| \leq C |\gamma| G^\mathrm{c}_\gamma$ for $\gamma$ small, and $G^\mathrm{h}$ is even better behaved, since it is uniformly exponentially localized in space for $\gamma$ small. Hence we only give the proof for the term involving $G^\mathrm{c}_\gamma$. We first need a basic preliminary result on localization of antiderivatives of odd functions, a proof of which can be found for instance in \cite[Appendix A]{JaramilloScheelWu}. 

\begin{lemma}\label{l: antiderivative localization}
	Let $r > 1$, and let $g \in L^\infty_{r,r}(\R)$ be odd. Define 
	\begin{align}
	G(x) = \int_{-\infty}^x g(y) \, dy. 
	\end{align}
	Then $G \in L^\infty_{r-1, r-1} (\R)$, and $\|G\|_{L^\infty_{r-1, r-1}} \leq C \| g \|_{L^\infty_{r,r}}$ for some constant $C > 0$ independent of $g$. 
\end{lemma}

\begin{lemma}\label{l: asymptotic resolvent heat derivative estimate}
	Let $r > 2$. There exist positive constants $C$ and $\delta$ such that for any odd function $g \in L^\infty_{r,r}(\R)$, we have 
	\begin{align}
	\| \partial_x (G_\gamma^\mathrm{c} \ast g) \|_{L^1_{1,1}} \leq \frac{C}{ |\gamma|} \| g \|_{L^\infty_{r,r}}, \label{e: asymptotic resolvent G c L11 estimate}
	\end{align}
	and 
	\begin{align}
	\| \partial_x (G_\gamma^\mathrm{c} \ast g) \|_{L^\infty_{r,r}} \leq \frac{C}{ |\gamma|^{r-1}} \| g \|_{L^\infty_{r,r}} \label{e: asymptotic resolvent G c Linf r estimate}
	\end{align}
	for all $\gamma \in B(0,\delta)$ such that $\Re \gamma \geq \frac{1}{2}|\Im \gamma|$. 
\end{lemma}
\begin{proof}
	We adopt a similar approach to \cite[proof of Lemma 5.1]{HowardCahnHilliard}, although we carry out our estimates in the Laplace domain rather than the time domain. We prove the $L^1_{1,1}$ control in estimate \eqref{e: asymptotic resolvent G c L11 estimate}, and the estimate \eqref{e: asymptotic resolvent G c Linf r estimate} follows with minor modifications. We first split the integral in the convolution as 
	\begin{align}
	(\partial_x G_\gamma^\mathrm{c} \ast g) (x) = \int_{ |y| \leq \frac{|x|}{2}} \partial_x G_\gamma^\mathrm{c}(x-y) g(y) \, dy + \int_{|y| > \frac{|x|}{2}} \partial_x G_\gamma^\mathrm{c}(x-y) g(y) \, dy. \label{e: asymptotic resolvent integral splitting}
	\end{align}
	Integrating by parts in the first term, we obtain 
	\begin{multline}
	\int_{|y| \leq \frac{|x|}{2}} \partial_x G_\gamma^\mathrm{c}(x-y) g(y) \, dy = - \int_{|y| \leq \frac{|x|}{2}} \partial_y \partial_x G_\gamma^\mathrm{c} (x-y) G(y) \, dy \\ + \left[ G \left( \frac{|x|}{2} \right) \partial_x G_\gamma^{\mathrm{c}} \left( x - \frac{|x|}{2} \right) - G \left( - \frac{|x|}{2} \right) \partial_x G_\gamma^\mathrm{c} \left( x + \frac{|x|}{2} \right) \right], \label{e: asymptotic resolvent integration by parts}
	\end{multline}
	where $G(x) = \int_{-\infty}^x g(y) \, dy$. Assume $x > 0$, so that $x \mp \frac{|x|}{2} \geq \frac{x}{2} \geq 0$. Differentiating the formula \eqref{e: G c def} for $G_\gamma^\mathrm{c} (\xi)$ for $\xi > 0$, we then obtain 
	\begin{align}
	\partial_x G_\gamma^\mathrm{c} \left( x \mp \frac{|x|}{2} \right) = \beta_c \frac{\nu^-(\gamma)}{\gamma} e^{\nu^-(\gamma) (x \mp \frac{x}{2})}.
	\end{align}
	for $x > 0$. By \eqref{e: nu pm expansion}, $\nu^- (\gamma)/\gamma$ is bounded for $\gamma$ small. Furthermore, if $\gamma$ is small and  $\Re \gamma \geq \frac{1}{2} |\Im \gamma|$, we have 
	\begin{align}
	\Re \nu^- (\gamma) = (\Re \nu^- (\gamma) +  \Re \nu_0 \gamma) - \Re \nu_0 \gamma \leq C |\gamma|^2 - c|\gamma| \leq -c |\gamma| 
	\end{align}
	for some constant $c > 0$, again by \eqref{e: nu pm expansion}, together with the fact that for $\Re \gamma \geq \frac{1}{2} |\Im \gamma|$, one has $|\gamma|^2 = (\Re \gamma)^2 + (\Im \gamma)^2 \leq C (\Re \gamma)^2$. Hence we have for $x > 0$
	\begin{align*}
	\langle x \rangle \left| G \left( \pm \frac{|x|}{2} \right) \partial_x G^\mathrm{c}_\gamma \left(x \mp \frac{|x|}{2} \right) \right| \leq C \| G \|_{L^\infty_{1,1}}  e^{\Re \nu^- (\gamma) (x \mp \frac{x}{2})} &\leq C \| g \|_{L^\infty_{r,r}} e^{\Re \nu^-(\gamma) x/2} \\
	&\leq C \|g \|_{L^\infty_{r,r}} e^{- c |\gamma| x} 
	\end{align*}
	by Lemma \ref{l: antiderivative localization}. Using the same consideration for $x < 0$ with $\nu^+(\gamma)$ replacing $\nu^-(\gamma)$, we obtain 
	\begin{align}
	\langle x \rangle \left| G \left( \pm \frac{|x|}{2} \right) \partial_x G^\mathrm{c}_\gamma \left(x \mp \frac{|x|}{2} \right) \right| \leq C \| g \|_{L^\infty_{r,r}} e^{-c |\gamma| |x|}
	\end{align}
	for all $x \neq 0$. By the change of variables $z = |\gamma| x$, we have 
	\begin{align}
	\int_\R e^{-c |\gamma| |x|} \, dx \leq \frac{C}{|\gamma|}
	\end{align}
	for some constant $C > 0$ independent of $\gamma$, and so the boundary terms in \eqref{e: asymptotic resolvent integration by parts} are controlled in $L^1_{1,1} (\R)$ by $\frac{C}{|\gamma|}\| g \|_{L^\infty_{r,r}}$. For the remaining integral in \eqref{e: asymptotic resolvent integration by parts}, we note that $G_\gamma^\mathrm{c} (x-y)$ is smooth on the region of integration since $x \neq y$ there. Splitting into cases $x - y > 0$ and $x - y < 0$ corresponding to the two different formulas in \eqref{e: G c def} and arguing as above, we obtain an estimate
	\begin{align*}
	\left| \int_{|y| \leq \frac{|x|}{2}} \partial_y \partial_x G_\gamma^\mathrm{c} (x-y) G(y) \, dy \right| \leq C |\gamma| \int_{|y| \leq \frac{|x|}{2}} e^{- c |\gamma| |x-y|} |G(y) | \, dy.
	\end{align*}
	By Lemma \ref{l: antiderivative localization}, we then have 
	\begin{align*}
	|\gamma| \int_{|y| \leq \frac{|x|}{2} } e^{- c |\gamma| |x-y|} |G(y) | \, dy &\leq C |\gamma| \int_{|y| \leq \frac{|x|}{2}} e^{- c |\gamma| |x-y|} \langle y \rangle^{-r+1} \, dy \| g \|_{L^\infty_{r,r}} \\
	&\leq C |\gamma| e^{-c |\gamma| |x|/2} \int_{|y| \leq \frac{|x|}{2}} \langle y \rangle^{-r+1} \, dy \| g \|_{L^\infty_{r,r}} \\
	&\leq C |\gamma| e^{-c |\gamma| |x|/2} \|g \|_{L^\infty_{r,r}},
	\end{align*}
	where we have used the fact that $|x-y| \geq \frac{|x|}{2}$ for $|y| \leq \frac{|x|}{2}$, and that $\langle y \rangle^{-r+1}$ is integrable on $\R$ for $r > 2$. Hence we have 
	\begin{align*}
	\int_\R \langle x \rangle \left| \int_{|y| \leq \frac{|x|}{2}} \partial_y \partial_x G_\gamma^\mathrm{c} (x-y) G(y) \, dy \right| \, dx &\leq C \| g \|_{L^\infty_{r,r}} \int_\R |\gamma| \langle x \rangle e^{- |\gamma| |x|/2} \, dx 
% 	\\&
	\leq \frac{C}{|\gamma|} \| g \|_{L^\infty_{r,r}},
	\end{align*}
	where the last estimate follows from again using the change of variables $z = |\gamma| x$ to estimate the remaining integral. Hence the first term in the decomposition of $\partial_x (G^\mathrm{c}_\gamma \ast g)$ in \eqref{e: asymptotic resolvent integral splitting} satisfies \eqref{e: asymptotic resolvent G c L11 estimate}. 
	
	For the second term in \eqref{e: asymptotic resolvent integral splitting}, for any fixed $x$ we have for almost every $y$ with $|y| > \frac{|x|}{2}$ 
	\begin{align*}
	|\partial_x G_\gamma^\mathrm{c} (x-y)| \leq e^{-c |\gamma| |x-y|} 
	\end{align*}
	for all $\gamma$ sufficiently small with $\Re \gamma \geq \frac{1}{2} |\Im \gamma|$. We therefore have 
	\begin{align*}
	\int_\R \langle x \rangle \left | \int_{|y| > \frac{|x|}{2} } \partial_x G_\gamma^\mathrm{c} (x-y) g(y) \, dy \right| \, dx &\leq C \| g \|_{L^\infty_{r,r}} \int_\R \int_{|y| > \frac{|x|}{2}} \langle x \rangle e^{-c |\gamma| |x-y|} \langle y \rangle^{-r} \, dy \, dx \\
	&\leq C \| g \|_{L^\infty_{r,r}} \int_\R \langle x \rangle^{-r+1} \int_{|y| \geq \frac{|x|}{2}} e^{- c |\gamma| |x-y|} \, dy \, dx \\
	&\leq C \| g \|_{L^\infty_{r,r}} \left( \int_\R \langle x \rangle^{-r+1} \, dx \right) \left(\int_\R e^{-c |\gamma| |z|} \, dz \right) \\
	&\leq \frac{C}{|\gamma|} \| g \|_{L^\infty_{r,r}},
	\end{align*}
	which proves \eqref{e: asymptotic resolvent G c L11 estimate}. The proof of \eqref{e: asymptotic resolvent G c Linf r estimate} is the same, simply replacing $L^1_{1,1}$ by $L^\infty_{r,r}$ norms. 
\end{proof}

\subsection{Estimates on the full resolvent} \label{s: full resolvent}
As in \cite{AveryScheel}, we use a far-field/core decomposition to transfer estimates from the asymptotic resolvent to the full resolvent $(\mcl - \gamma^2)^{-1}$. The main difference from \cite{AveryScheel} is that here we need to work in both $L^1$- and $L^\infty$-based spaces rather than simply in $L^2$-based spaces, as an interplay of these spaces is needed in order to handle some borderline cases in our estimates. Such a borderline case can be seen already in Proposition \ref{p: asymptotic resolvent derivative estimate}, where only in the space $L^1_{1,1} (\R)$ do we get the optimal $|\gamma|^{-1}$ estimate --- we cannot close the argument in $L^\infty_{2,2} (\R)$ due to the fact that $\langle x \rangle^{-1}$ is not integrable on $\R$, and for $r > 2$ we find a slightly worse estimate \eqref{e: asymptotic resolvent G c Linf r estimate} in $L^\infty_{r,r} (\R)$. 

In this section we turn to obtaining estimates on the resolvent $(\mcl - \gamma^2)^{-1}$ that translate into optimal decay estimates for the semigroup $e^{\mcl t}$. \hl{We note that since these estimates concern the full linearization about the front $q_*$, the estimates in this section rely on all of Hypotheses \ref{hyp: spreading speed} through \ref{hyp: resonance}}. The first estimate corresponds to the $t^{-3/2}$ decay of the semigroup when one gives up sufficient algebraic localization. 

\begin{prop}\label{p: full resolvent lipschitz estimate}
	There exist positive constants $C$ and $\delta$ and a bounded limiting operator $R_0 : L^1_1 (\R) \to W^{2m-1, \infty}_{-1} (\R)$ such that for any $g \in L^1_1(\R)$, we have 
	\begin{align}
	\| (\mcl - \gamma^2)^{-1} g - R_0 g \|_{W^{2m-1, \infty}_{-1}} \leq C |\gamma| \| g \|_{L^1_1}
	\end{align}
	for all $\gamma \in B(0, \delta)$ such that $\gamma^2$ is to the right of the essential spectrum of $\mcl$. 
\end{prop}

The next result contains estimates on derivatives which imply sharp decay estimates on $\partial_x e^{\mcl t}$ when giving up (essentially) no spatial localization. 

\begin{prop}\label{p: full resolvent derivative estimates}
	Let $r > 2$. There exist positive constants $C$ and $\delta$ such that for any $g \in L^\infty_r(\R)$, we have 
	\begin{align}
	\| \partial_x (\mcl - \gamma^2)^{-1} g \|_{L^1_1} \leq \frac{C}{|\gamma|} \| g \|_{L^\infty_r} \label{e: full resolvent L11 derivative estimate}
	\end{align}
	and 
	\begin{align}
	\| \partial_x (\mcl - \gamma^2)^{-1} g \|_{L^\infty_r} \leq \frac{C}{|\gamma|^{r-1}} \| g \|_{L^\infty_r} \label{e: full resolvent L inf r derivative estimate}
	\end{align}
	for all $\gamma \in B(0,\delta)$ such that $\Re \gamma \geq \frac{1}{2}|\Im \gamma|$. 
\end{prop}

In the remainder of this section, we prove Proposition \ref{p: full resolvent derivative estimates}. The essential ingredients of this proof are used to prove the $L^2$ analogue of Proposition \ref{p: full resolvent lipschitz estimate} in \cite{AveryScheel}, so we do not give the full details of the proof of Proposition \ref{p: full resolvent lipschitz estimate} here. As in \cite{AveryScheel}, we decompose our data $g$ into a left piece, a center piece, and a right piece. We use the left and right pieces as data for the asymptotic operators $\mcl_\pm$, and then solve the resulting equation on the center piece in exponentially weighted function spaces in which we recover Fredholm properties of $\mcl$. 

To this end, we let $(\chi_-, \chi_c, \chi_+)$ be a partition of unity on $\R$ such that 
\begin{align}
\chi_+ (x) = \begin{cases}
0, & x \leq 2, \\
1, & x \geq 3,
\end{cases}
\end{align}
$\chi_- (x) = \chi_+(-x)$, and $\chi_c=1-\chi_+-\chi_-$ compactly supported. Given $g \in L^1_1 (\R)$, we then write 
\begin{align}
g = \chi_- g + \chi_c g + \chi_+g =: g_- + g_c + g_+. 
\end{align}
We would like to decompose the solution $u$ to $(\mcl - \gamma^2) u = g$ in a similar way, but we need to refine this approach to take advantage of the fact that the spectrum of $\mcl_-$ is strictly in the left half plane creating a strong absorption effect on the left. For this, let $g^\mathrm{odd}_+ (x) = g_+ (x) - g_+ (-x)$ be the odd extension of $g_+$ and let $u^+$ solve 
\begin{align}
(\mcl_+ - \gamma^2)u^+ = g_+^\mathrm{odd}. 
\end{align}
We let $u^-$ solve 
\begin{align}
(\mcl_- - \gamma^2) u^- = g_-,
\end{align}
decompose the solution $u$ to $(\mcl - \gamma^2) u = g$ as $u = u^- + u^c + \chi_+ u^+$, such that $u^c$ solves 
\begin{align}
(\mcl - \gamma^2) u^c = \tilde{g} (\gamma), \label{e: full resolvent u c eqn}
\end{align}
with
\begin{align}
\tilde{g}(\gamma) = g_c + (\chi_+- \chi_+^2) g - [\mcl_+, \chi_+] u^+ + (\mcl_+ - \mcl) (\chi_+u^+) + (\mcl_- - \mcl) u^-, \label{e: full resolvent tilde f def}
\end{align}
where $[\mcl_+, \chi_+] = \mcl_+ (\chi_+ \cdot) - \chi_+ \mcl \cdot$ is the commutator. Note that $\mcl$ attains its limits exponentially quickly as $x \to \infty$, and the commutator $[\mcl_+, \chi_+]$ is a differential operator of order $2m-1$ with compactly supported coefficients since $\chi_+$ is constant outside the interval $[2, 3]$, and so $\tilde{g}(\gamma)$ is exponentially localized with a rate that is uniform in $\gamma$ for $\gamma$ small. In fact, the exponential localization of the coefficients of $u^+$ and $u^-$ in \eqref{e: full resolvent tilde f def} together with Proposition \ref{p: asymptotic resolvent Lipschitz estimate} and Hypothesis \ref{hyp: stable on left} (which implies that $0$ is not in the spectrum of $\mcl_-$)  allow us to conclude that $\tilde{g}(\gamma)$ is Lipschitz in $\gamma$ in exponentially localized spaces with small exponents. Recall the notation for exponentially weighted spaces $L^p_{\mathrm{exp}, \eta_-, \eta_+} (\R)$ introduced in Section \ref{s: overview}. 

\begin{lemma}\label{l: full resolvent tilde f localization}
	Let $r > 2$ and let $\eta > 0$ be small. There exist positive constants $C$ and $\delta$ such that for $\gamma \in B(0,\delta)$ with $\gamma^2$ to the right of the essential spectrum of $\mcl$, we have 
	\begin{align}
	\| \tilde{g} (\gamma) - \tilde{g}(0) \|_{L^1_{\mathrm{exp}, -\eta, \eta}} \leq C | \gamma| \| g \|_{L^1_1} \label{e: tilde f L1 estimate}
	\end{align}
	for any $g \in L^1_1 (\R)$, and 
	\begin{align}
	\| \tilde{g} (\gamma) - \tilde{g}(0) \|_{L^\infty_{\mathrm{exp}, -\eta, \eta}} \leq C |\gamma| \| g \|_{L^\infty_r} \label{e: tilde f Linfty estimate}
	\end{align}
	for any $g \in L^\infty_r (\R)$. 
\end{lemma}

With exponential localization of $\tilde{f}(\gamma)$ in hand, we solve \eqref{e: full resolvent u c eqn} by making the far-field/core ansatz 
\begin{align*}
u^c (x) = w(x) + a \chi_+ (x) e^{\nu^- (\gamma) x},
\end{align*}
where $a \in \C$ and we will require $w$ (the \textit{core} piece of the solution) to be exponentially localized. Inserting this ansatz into \eqref{e: full resolvent u c eqn} results in an equation
\begin{align}
F (w, a; \gamma) = \tilde{g}(\gamma) \label{e: full resolvent ff core eqn}
\end{align}
where 
\begin{align}
F(w, a; \gamma) = \mcl w + a \mcl (\chi_+ e^{\nu^- (\gamma) \cdot}) - \gamma^2 (w + a \chi_+ e^{\nu^- (\gamma) \cdot}). \label{e: full resolvent F def}
\end{align}
We will solve \eqref{e: full resolvent ff core eqn} by taking advantage of Fredholm properties of $F$ on exponentially weighted function spaces. First we state relevant Fredholm properties of $\mcl$ which follow readily from Morse index calculations \cite{FiedlerScheel, KapitulaPromislow}. Throughout, for $\eta > 0$ we let $(X_\eta, Y_\eta)$ denote either pair of spaces \begin{align}
X_\eta = L^1_{\mathrm{exp}, -\eta, \eta} (\R), \, Y_\eta = W^{2m, 1}_{\mathrm{exp}, -\eta, \eta} (\R) \quad \text{ or } \quad X_\eta = L^\infty_{\mathrm{exp}, -\eta, \eta} (\R), \, Y_\eta = W^{2m, \infty}_{\mathrm{exp}, -\eta, \eta} (\R). \label{e: full resolvent X eta Y eta}
\end{align}
\begin{lemma}\label{l: full resolvent L fredholm properties}
	For $\eta > 0$ sufficiently small, the operator $\mcl : Y_\eta \to X_\eta$ is Fredholm with index $-1$ for either pair of spaces $(X_\eta, Y_\eta)$ defined in \eqref{e: full resolvent X eta Y eta}.  
\end{lemma}

\begin{lemma}
	For $\eta > 0$ sufficiently small and for either pair of spaces $(X_\eta, Y_\eta)$ defined in \eqref{e: full resolvent X eta Y eta}, there exists a $\delta > 0$ such that $F : Y_\eta \times \C \times B(0,\delta) \to X_\eta$ is well defined and the mapping 
	\begin{align}
	\gamma \mapsto F(\cdot, \cdot; \gamma) : B(0, \delta) \to \mathcal{B} (Y_\eta \times \C, X_\eta)
	\end{align}
	is analytic in $\gamma$. 
\end{lemma}
\begin{proof}
	Exploiting the fact that $(\mcl - \gamma^2) e^{\nu^-(\gamma) \cdot} = 0$ since $\nu^-(\gamma)$ is a root of the dispersion relation, we can rewrite $F$ as 
	\begin{align*}
	F(w, a; \gamma) = (\mcl - \gamma^2) w + a \left[ \chi_+ (\mcl - \mcl_+) e^{\nu^- (\gamma) \cdot} + [\mcl, \chi_+] e^{\nu^-(\gamma) \cdot} \right] 
	\end{align*}
	The coefficients of the operators $(\mcl - \mcl_+)$ and $[\mcl, \chi_+]$ are exponentially localized in space with a rate that is uniform in $\gamma$ for $\gamma$ small, which implies that $F$ is well-defined in these spaces for $\eta$ sufficiently small, since this exponential localization is enough to absorb any small exponential growth of $e^{\nu^-(\gamma) \cdot}$. Since $\nu^-(\gamma)$ solves 
	\begin{align*}
	d^+(\gamma^2, \nu^-(\gamma) - \eta_*) = 0,
	\end{align*}
	which is a polynomial in $\nu^-(\gamma)$, one can use the Newton polygon to show that $\gamma \mapsto \nu^-(\gamma)$ is analytic in a neighborhood of $\gamma = 0$. Analyticity of $F$ in $\gamma$ then follows from the uniform localization of the coefficients of $(\mcl - \mcl_+)$ and $[\mcl, \chi_+]$ together with the analyticity of 
	\begin{align}
	\gamma \mapsto e^{\nu^- (\gamma \cdot)} : B(0,\delta) \to Y_{-\eta}
	\end{align}
	for $Y_{-\eta} = W^{2m, \infty}_{\mathrm{exp}, \eta, -\eta} (\R)$ or $W^{2m, 1}_{\mathrm{exp}, \eta, -\eta} (\R)$. The fact that this map is analytic readily follows from the fact that $\nu^-(\gamma)$ is analytic and has the expansion \eqref{e: nu pm expansion}, along with pointwise analyticity of the exponential function. 
\end{proof}

\begin{corollary}\label{c: full resolvent ff core invertability}
	For $\eta > 0$ sufficiently small and for either pair of spaces $(X_\eta, Y_\eta)$ in \eqref{e: full resolvent X eta Y eta}, there exists a $\delta> 0$ such that for each $\gamma \in B(0,\delta)$, the map 
	\begin{align}
	(w,a) \mapsto F(w, a; \gamma) : Y_\eta \times \C \to X_\eta 
	\end{align}
	is invertible. We denote the solution to $F(w, a; \gamma) = \tilde{g}$ by 
	\begin{align}
	(w,a) = (T(\gamma) \tilde{g}, A(\gamma) \tilde{g}), 
	\end{align}
    with analytic maps
    \begin{align}
	\gamma \mapsto (T(\gamma),A(\gamma))  : B(0, \delta) \to \mathcal{B} (X_\eta, Y_\eta) \times  \mathcal{B} (X_\eta, \C).
	\end{align}
\end{corollary}
\begin{proof}
	Observe that $F$ is linear in both $w$ and $a$. By Lemma \ref{l: full resolvent L fredholm properties} and continuity of the Fredholm index, $D_w F = \mcl - \gamma^2$ is Fredholm with index $-1$ for $\gamma$ sufficiently small. Therefore, the joint linearization $D_{(w,a)} F$ has Fredholm index 0 by the Fredholm bordering lemma. Moreover, the kernel of $F(\cdot, \cdot; 0)$ is trivial, since a nontrivial kernel would imply there exists $(w,a) \in Y_\eta \times \C$ such that $\mcl (w + a \chi_+) = 0$, i.e. there would be a bounded solution to $\mcl u = 0$, which contradicts Hypothesis \ref{hyp: resonance}. The result then follows from the analytic Fredholm theorem.    
\end{proof}

The proof of Proposition \ref{p: full resolvent lipschitz estimate} is similar to the proof of \cite[Proposition 3.1]{AveryScheel} --- the only additional ingredient is the use of the Sobolev embedding $\| g \|_{L^\infty} \leq C \| g \|_{W^1_1}$ to account for the interplay of $L^1$- and $L^\infty$-based spaces, so we only give the details of the proof of Proposition \ref{p: full resolvent derivative estimates}.

\begin{proof}[Proof of Proposition \ref{p: full resolvent derivative estimates}]
	The estimates for $\chi_+ u^+$ follow from Proposition \ref{p: asymptotic resolvent derivative estimate}, and the estimates for $u^-$ follow from the fact that $0$ is in the resolvent set of $\mcl^-$, so we only need to prove the estimate for $u^c$. We use Corollary \ref{c: full resolvent ff core invertability} to write $u^c$ as 
	\begin{align}
 	u^c (\gamma) = T(\gamma) \tilde{g}(\gamma) + A(\gamma) \tilde{g}(\gamma) \chi_+ e^{\nu^-(\gamma) \cdot}, \label{e: full resolvent uc ff core formula}
 	\end{align}
 	where $T(\gamma) : L^1_{\mathrm{exp}, -\eta, \eta} (\R) \to W^{2m,1}_{\mathrm{exp}, -\eta, \eta} (\R)$ and $A(\gamma) : L^1_{\mathrm{exp}, -\eta, \eta} (\R) \to \C$ are analytic in $\gamma$ in a neighborhood of the origin for $\eta > 0$ small, fixed. We use Corollary \ref{c: full resolvent ff core invertability} to estimate
	\begin{align*}
	\| \partial_x T(\gamma) \tilde{g}(\gamma) \|_{L^1_1} \leq C \| T(\gamma) \tilde{g}(\gamma) \|_{W^{1, 1}_{\mathrm{exp}, -\eta, \eta}} \leq C \| \tilde{g}(\gamma) \|_{L^1_{\mathrm{exp}, -\eta, \eta}} \leq C \| g \|_{L^1_1} \leq C \| g \|_{L^\infty_r}.
	\end{align*}
	For $\gamma$ small with $\Re \gamma \geq  \frac{1}{2}| \Im \gamma|$, $\gamma^2$ is in particular to the right of the essential spectrum of $\mcl$, so that by Lemma \ref{l: full resolvent tilde f localization} we have 
	\begin{align}
	\| \tilde{g}(\gamma) \|_{L^1_{\mathrm{exp}, -\eta, \eta}} \leq C \| g \|_{L^1_1} \leq C \| g \|_{L^\infty_r},
	\end{align}
	using also the continuous embedding $L^\infty_r (\R) \subset L^1_1(\R)$ for $r > 2$. Hence the core term satisfies an even better estimate than \eqref{e: full resolvent L11 derivative estimate}, namely, 
	\begin{align}
		\| \partial_x T(\gamma) \tilde{g}(\gamma) \|_{L^1_1} \leq C \| g \|_{L^\infty_r}. 
	\end{align}
	For the far-field terms, we again use Corollary \ref{c: full resolvent ff core invertability} and Lemma \ref{l: full resolvent tilde f localization} to estimate
	\begin{align*}
	\| \partial_x [ A(\gamma) \tilde{g}(\gamma) \chi_+ e^{\nu^-(\gamma) \cdot}] \|_{L^1_1} &\leq C | A(\gamma) \tilde{g}(\gamma) | \| \partial_x (\chi_+ e^{\nu^-(\gamma) \cdot}) \|_{L^1_1} \\
	&\leq C \| g \|_{L^1_1} \| \partial_x (\chi_+ e^{\nu^-(\gamma) \cdot}) \|_{L^1_1} \\
	&\leq C \| g \|_{L^\infty_r} \left( \| \chi_+' e^{\nu^-(\gamma) \cdot}) \|_{L^1_1} + |\nu^-(\gamma)| \| \chi_+ e^{\nu^-(\gamma) \cdot} \|_{L^1_1} \right).
	\end{align*}
	We recall that $\chi_+'$ is supported on the interval $[2,3]$, and that for $\gamma $ small with $\Re \gamma \geq \frac{1}{2} |\Im \gamma|$, we have $\Re \nu^-(\gamma) \leq -c |\gamma|$ for some constant $c > 0$. Hence we have 
	\begin{align*}
	\| \chi_+' e^{\nu^-(\gamma) \cdot} \|_{L^1_1} \leq C \int_2^3 \langle x \rangle e^{-c |\gamma| x} \, dx \leq C \int_2^3 x e^{- c |\gamma| x} \, dx = \frac{C}{|\gamma|^2} \int_{2 |\gamma|}^{3 |\gamma|} z e^{-cz} \, dz 
	\end{align*}
	using the change of variables $z = |\gamma| x$. By the fundamental theorem of calculus, the remaining integral is $\mathrm{O}(\gamma)$ for $\gamma$ small, so that 
	\begin{align}
	\| \chi_+' e^{\nu^-(\gamma) \cdot} \|_{L^1_1} \leq \frac{C}{|\gamma|}. 
	\end{align}
	For the other term, we similarly obtain
	\begin{align*}
	| \nu^-(\gamma) | \| \chi_+ e^{\nu^-(\gamma) \cdot} \|_{L^1_1} \leq C |\gamma| \int_2^\infty e^{- c |\gamma| x} \, dx \leq C \leq \frac{C}{|\gamma|},
	\end{align*}
	so that 
	\begin{align*}
	\| \partial_x [ A(\gamma) \tilde{g}(\gamma) \chi_+ e^{\nu^-(\gamma) \cdot}] \|_{L^1_1} \leq \frac{C}{|\gamma|} \| g \|_{L^\infty_r} 
	\end{align*}
	for $\gamma$ small with $\Re \gamma \geq \frac{1}{2} |\Im \gamma|$, which completes the proof of \eqref{e: full resolvent L11 derivative estimate}. The proof of \eqref{e: full resolvent L inf r derivative estimate} is similar. 
\end{proof}

\section{Linear decay estimates}\label{s: linear estimates}

We translate the resolvent estimates of Section \ref{s: resolvent estimates} into decay estimates on the semigroup $e^{\mcl t}$. We only sketch proofs, which follow very closely the analogous results in \cite{AveryScheel}, and point out modifications. \hl{We note that the estimates in this section rely on all of Hypotheses \ref{hyp: spreading speed} through \ref{hyp: resonance}, as they depend on the full resolvent estimates of Section \ref{s: full resolvent}.}

\begin{lemma}
   For each $1 \leq p \leq \infty$, the operator $\mcl : W^{2m, p} (\R) \subseteq L^p(\R) \to L^p(\R)$ is sectorial. 
\end{lemma}
\begin{proof}
    For $1 < p < \infty$, this result is implied by \cite[Theorem 3.2.2]{Lunardi}, even for $x\in\R^n$. In the boundary cases $p = 1, \infty$, the result holds for $x\in\R$, as can readily be seen as follows. By Fourier transform and scaling,   the integral kernel for $((-1)^{m+1} \partial_x^{2m} - \lambda)^{-1}$ is bounded above by $C|\lambda|^{1-1/2m} e^{- c |\lambda|^{1/2m} |x|}$ in  $\{\lambda = |\lambda| e^{i \theta}, \theta \in (-\frac{\pi}{2}, \frac{\pi}{2})\}$, for some constants $C, c > 0$, which yields 
    \begin{align}
    \| [(-1)^{m+1} \partial_x^{2m} - \lambda]^{-1} \|_{L^p \to L^p} \leq \frac{C_p}{|\lambda|}
    \end{align}
    for each $1 \leq p \leq \infty$ by Young's convolution inequality. Hence the highest order part $(-1)^{m+1} \partial_x^{2m}$ of $\mcl$ is sectorial, and, by the Gagliardo-Nirenberg-Sobolev inequality, $\mcl$ is sectorial as well \cite{Henry}.
\end{proof}

% As an elliptic operator with smooth bounded coefficients, $\mcl$ is sectorial on $L^p (\R)$ with domain $W^{2m, p} (\R)$ for $1 \leq p \leq \infty$. Here we allow an operator to be called sectorial even if it is not densely defined, so that $\mcl$ is still sectorial on $L^\infty(\R)$. For $1 < p < \infty$, this is implied by \cite[Theorem 3.2.2]{Lunardi}. It is difficult to find this precise statement in the literature for $p = 1$ or $\infty$, since many results are stated for operators in higher space dimensions, and there are additional technical difficulties with maximal regularity for $p = 1$ or $\infty$ in higher dimensions. 

% Nevertheless, one can use the Fourier transform and some complex analysis to show that for $\lambda = |\lambda| e^{i \theta}$ with $\theta \in (-\frac{\pi}{2}, \frac{\pi}{2})$, the integral kernel for $((-1)^{m+1} \partial_x^{2m} - \lambda)^{-1}$ is bounded above by $|\lambda|^{1-1/2m} e^{- c |\lambda|^{1/2m} |x|}$ for some constant $c > 0$, from which it follows by Young's convolution inequality that 
% \begin{align}
% \| [(-1)^{m+1} \partial_x^{2m} - \lambda]^{-1} \|_{L^p \to L^p} \leq \frac{C_p}{|\lambda|}
% \end{align}
% for each $1 \leq p \leq \infty$. Hence the highest order part of $\mcl$ is sectorial, and by the Gagliardo-Nirenberg-Sobolev inequality, $\mcl$ is relatively bounded perturbation of $(-1)^{m+1} \partial_x^{2m}$ such that $\mcl$ is also sectorial \cite{Henry}. 

Therefore, $\mcl$ generates an analytic semigroup on $L^p (\R)$, given by the inverse Laplace transform 
\begin{align}
e^{\mcl t} = -\frac{1}{2 \pi i} \int_\Gamma e^{\lambda t} (\mcl - \lambda)^{-1} \, d \lambda, \label{e: semigroup formula}
\end{align}
for a suitably chosen contour $\Gamma$. Note that $\mcl$ is not densely defined on $L^\infty(\R)$ and we rely on the construction of analytic semigroups in \cite{Lunardi} for not necessarily densely defined sectorial operators; in particular,  strong continuity at time $t = 0$ holds only after regularizing with $(\mcl-\lambda)^{-1}$. We now begin stating  decay estimates on this semigroup. 

\begin{prop}\label{p: linear decay L11} 
	There exists a constant $C > 0$ such that for any $z_0 \in L^1_1(\R)$, we have for all $t>0$,
	\begin{align} 
	\| e^{\mcl t} z_0 \|_{L^\infty_{-1}} \leq \frac{C}{t^{3/2}} \| z_0 \|_{L^1_1}.
	\end{align}
% 	for all $t > 0$. 
\end{prop}
\begin{proof}
    The proof is the analogous to the proof of Proposition 4.1 in \cite{AveryScheel}, exploiting the previously established estimates on the resolvent in function spaces slightly different from those in  \cite{AveryScheel}.  The key insight is to use an  integration contour  tangent to the essential spectrum of $\mcl$ in the $\gamma$-plane at the origin. By Hypothesis \ref{hyp: resonance}, there are no unstable eigenvalues to obstruct shifting of the contours since eigenfunctions are smooth and exponentially localized and thus independent of algebraic weights or choice of $L^p$ space; see Figure \ref{f: contours} for a schematic of the integration contours. 
\end{proof}

\begin{corollary}\label{c: linear decay L inf r}
	Let $r > 2$. There exists a constant $C > 0$ such that for any $z_0 \in L^\infty_r (\R)$, we have 
	\begin{align}
	\| e^{\mcl t} z_0 \|_{L^\infty_{-1}} \leq \frac{C}{(1+t)^{3/2}} \| z_0 \|_{L^\infty_r} 
	\end{align}
	for all $t > 0$. 
\end{corollary} 
\begin{proof}
	The estimate holds for $t > 1$ by Proposition \ref{p: linear decay L11} and the fact that  $L^\infty_r(\R)$ is continuously embedded in $L^1_1(\R)$ for $r > 2$. For $0 < t < 1$, observe that conjugating $\mcl$ with the weight $\langle x \rangle^{-1}$ results in an elliptic operator with smooth bounded coefficients, so that 
	\begin{align}
	\| e^{\mcl t} z_0 \|_{L^\infty_{-1}} \leq C \| g \|_{L^\infty_{-1}} \leq C \| g \|_{L^\infty_r}
	\end{align}
	for $0 < t < 1$ by standard semigroup theory, and the result follows. 
\end{proof}

\begin{figure}
	\centering
	\includegraphics[width=1\textwidth]{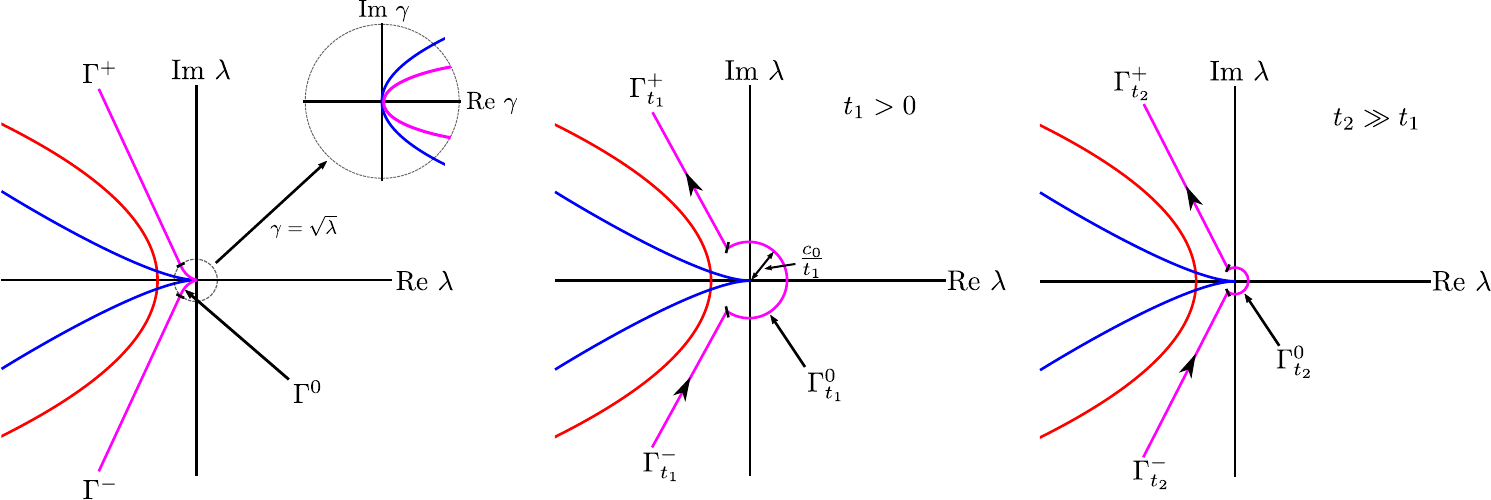}
	\caption{Left: Fredholm borders of $\mcl$ (blue, red) together with integration contour $\Gamma = \Gamma^+ \cup \Gamma^0 \cup \Gamma^-$ (magenta) used in the proof of Proposition \ref{p: linear decay L11}. Center and right: integration contours (magenta) used in the proof of Proposition \ref{p: linear derivative estimates} at times $t_1$ (center) and $t_2 > t_1$ (right).}
	\label{f: contours}
\end{figure}
The previous results trade spatial  localization for temporal decay. The next result obtains decay of derivatives without loss of localization.

\begin{prop}\label{p: linear derivative estimates}
	Let $r > 2$. There exists a constant $C > 0$ such that for any $z_0 \in L^\infty_r (\R)$, we have 
	\begin{align}
	\| \partial_x e^{\mcl t} z_0 \|_{L^\infty_r} \leq \frac{C}{t^{3/2 - r/2}} \| z_0 \|_{L^\infty_r},
	\end{align}
	and 
	\begin{align}
	\| \partial_x e^{\mcl t} z_0 \|_{L^1_1} \leq \frac{C}{t^{1/2}} \| z_0 \|_{L^\infty_r} 
	\end{align}
	for all $t > 1$. 
\end{prop}
\begin{proof}
	We differentiate \eqref{e: semigroup formula} and proceed as in \cite[Proposition 7.4]{AveryScheel}, choosing $\Gamma$ to be a circular arc centered at the origin whose radius scales as $t^{-1}$, connected to two rays extending out to infinity in the left half plane; see Figure \ref{f: contours}. The estimates in Proposition \ref{p: full resolvent derivative estimates} on the blowup of derivatives of the resolvent near the origin then translate into the claimed decay rates. 
\end{proof}

Finally, we record useful small time regularity estimates for $e^{\mcl t}$. 

\begin{lemma}\label{l: regularity L1 Linf}
	There exists a constant $C > 0$ such that 
	\begin{align}
	\| e^{\mcl t} z_0 \|_{L^\infty} \leq \frac{C}{t^{1/2m}} \| z_0 \|_{L^1}
	\end{align}
	for all $0 < t < 2$. 
\end{lemma}
\begin{proof}
	Set $\mcl_0 = (-1)^{m+1} \partial_x^{2m}$, and write $\mcl = \mcl_0 + (\mcl - \mcl_0)$. Using the Fourier transform yields
	\begin{align}
	\| e^{\mcl_0 t} z_0 \|_{L^\infty} \leq \frac{C}{t^{1/2m}} \| z_0 \|_{L^1}. 
	\end{align}
	We then write $e^{\mcl t} z_0$ in mild form, viewing $\mcl$ as a perturbation of $\mcl_0$, so that
	\begin{align}
	z(t) := e^{\mcl t} z_0 = e^{\mcl_0 t} z_0 + \int_0^t e^{\mcl_0 (t-s)} [(\mcl-\mcl_0) z(s)] \, ds. 
	\end{align}
	Since $(\mcl - \mcl_0)$ is a differential operator of order $2m-1$ with smooth bounded coefficients, the Gagliardo-Nirenberg-Sobolev inequality allows us to control the integrand and obtain the desired estimate through a contraction argument in temporally weighted spaces.  
\end{proof}

Similarly, we obtain the following small time bounds on derivatives of solutions. 

\begin{lemma}\label{l: small time derivative estimates}
	Let $r > 2$. There exists a constant $C > 0$ such that 
	\begin{align}
	\| \partial_x e^{\mcl t} z_0 \|_{L^\infty_r} \leq \frac{C}{t^{1/2m}} \| z_0 \|_{L^\infty_r},
	\end{align}
	and 
	\begin{align}
	\| \partial_x e^{\mcl t} z_0 \|_{L^1_1} \leq \frac{C}{t^{1/2m}} \| z_0 \|_{L^1_1}
	\end{align}
	for all $0 < t < 2$. 
\end{lemma}

\section{Stability argument}\label{s: stability argument}

Recall from Section \ref{s: outlook} that our goal is to use the sharp linear estimates obtained in Section \ref{s: linear estimates} to control perturbations to the approximate solution. We start by letting $v$ solve solve $\NL [v] = 0$, where $\NL$ is given by \eqref{e: NL def}, which is the original equation \eqref{e: eqn} in the co-moving frame with the logarithmic delay, and with the exponential weight $\omega$. We then define the perturbation $w = v -\psi$, which solves 
\begin{align}
	w_t = \mcl w - f'(q_*) w - \frac{3}{2 \eta_* (t+T)} \left[ \omega (\omega^{-1})' w + w_x \right] + \omega f(\omega^{-1} w + \psi) - \omega f(\omega^{-1} \psi) - R,
\end{align}
where $R = \NL[\psi]$. 
Note that $\omega^{-1} \psi$ is uniformly bounded, so by Taylor's theorem,
\begin{align}
	\omega | N (\omega^{-1} w) | \leq C \omega^{-1} w^2, \label{e: nonlinearity quadratic estimate}
\end{align}
and 
\begin{align}
	|(f'(\omega^{-1} \psi)- f'(q_*)) w| \leq C | \omega^{-1} \psi - q_*| |w| 
\end{align}
for some constant $C > 0$. 

Recall that, in order to deal with the presence of the critical term $-\frac{3}{2 \eta_*} (t+T)^{-1} \omega(\omega^{-1} w)$ in the $w$-equation, we introduce the new variable $z = (t+T)^{-3/2} w$, which solves 
\begin{multline}
z_t = \mcl z - \frac{3}{2 \eta_* (t+T)} [\omega (\omega^{-1})' + \eta_*] z - \frac{3}{2 \eta_* (t+T)} z_x + [f'(\omega^{-1} \psi) - f'(q_*)] z  - (t+T)^{-3/2} R \\  + (t+T)^{-3/2} \omega N (\omega^{-1} (t+T)^{3/2} z). \label{e: z eqn}
\end{multline}
Note that $\omega(\omega^{-1})' (x) + \eta_* \equiv 0 $ for $x > 1$, so that this term is essentially removed, which ultimately allows us to regain the $t^{-3/2}$ decay for $z$, equivalent to  boundedness for $w$. This argument requires some care, however: we must track dependence on $T$, and compensate for the extra factor of $(t+T)^{3/2}$ now appearing with the nonlinearity. \hl{Tracking the $T$ dependence is necessary since we will treat linear terms such as $-\frac{3}{\eta_*} (t+T)^{-1} z_x$ perturbatively, and so we use largeness of $T$ to guarantee that these terms remain small, and because we need to track the $T$ dependence of $z$ in order to conclude that $w = (t+T)^{3/2} z$ is bounded.} To handle the nonlinearity, note that, by Taylor's theorem and exponential decay of $\omega^{-1}$, there exists a non-decreasing function $K : \R_+ \to \R_+$ such that 
\begin{align}
\| (t+T)^{-3/2} \omega N (\omega^{-1} (t+T)^{3/2} z) \|_{L^\infty_r} \leq K(B) (t+T)^{3/2} \| z \|_{L^\infty_{-1}}^2, \label{e: nonlinearity quadratic norm estimate}
\end{align}
provided $(t+T)^{3/2} \| \omega^{-1} z \|_{L^\infty} \leq B$. In summary,  the nonlinearity gains spatial localization but carries a factor of $(t+T)^{3/2}$.  

We rewrite \eqref{e: z eqn} in mild form via the variation of constants formula
\begin{align}
z(t) = e^{\mcl t} z_0 + \mathcal{I}_1 (t) + \mathcal{I}_2 (t) + \mathcal{I}_3 (t) + \mathcal{I}_R (t) + \mathcal{I}_N (t), \label{e: voc}
\end{align} 
where 
\begin{align}
\mathcal{I}_1 (t) &= - \frac{3}{2 \eta_*} \int_0^t e^{\mcl (t-s)} \left[(\omega (\omega^{-1})' + \eta_*) \frac{z(s)}{s+T} \right] \, ds, \\
\mathcal{I}_2 (t) &= - \frac{3}{2 \eta_*} \int_0^t e^{\mcl (t-s)} \frac{z_x (s)}{s+T} \, ds, \label{e: I2 def} \\
\mathcal{I}_3 (t) &= \int_0^t e^{\mcl (t-s)} \left[ (f'(\omega^{-1} \psi (s)) - f'(q_*)) z(s) \right] \, ds, \\
\mathcal{I}_R (t) &= - \int_0^t e^{\mcl (t-s)} (s+T)^{-3/2} R(s) \, ds, \\
% \end{align}
% and 
% \begin{align}
\mathcal{I}_N (t) &= \int_0^t e^{\mcl (t-s)} [ (s+T)^{-3/2} \omega N(\omega^{-1} (s+T)^{3/2} z(s)) ] \, ds. 
\end{align}

By standard parabolic regularity  \cite{Lunardi, Henry}, the equation for $z$ is locally well-posed in $L^\infty_r (\R)$ for any $r \in \R$, in the sense that given any small initial data $z_0 \in L^\infty_r (\R)$, the variation of constants formula \eqref{e: voc} defines a unique solution $z(t)$ for $t \in (0, t_*)$ to \eqref{e: z eqn} with
\begin{align}
\lim_{t \to 0^+} (\mcl - \lambda)^{-1} z(t) = (\mcl - \lambda)^{-1} z_0 \text{ in } L^\infty_r (\R) \label{e: z initial data}
\end{align}
for any $\lambda$ in the resolvent set of $\mcl$. Furthermore, the maximal existence time $t_*$ depends only on $\| z_0 \|_{L^\infty_r}$, and there is a constant $C > 0$ such that 
for $t$ sufficiently small,
\begin{align}
\| z_x (t)\|_{L^\infty_r} \leq \frac{C}{t^{1/2m}} \| z_0 \|_{L^\infty_r}. \label{e: z small time regularity}
\end{align}

\begin{thmlocal}\label{t: nonlinear decay}
	Let $0 < \mu < \frac{1}{8}$ and $r = 2 + \mu$. Choose $\zeta(t+T)$ as in Proposition \ref{p: residual estimate}, and let $R(t;T)$ be the associated nonlinear residual defined in \eqref{e: residual L inf r estimate}. Then define
	\begin{align}
	R_0 = \sup_{T \geq T_*} \sup_{s > 0} \, (s+T)^{1/2-4\mu} \| R(s; T) \|_{L^\infty_r}, \label{e: R0 def}
	\end{align}
	which is finite for some $T_*$ sufficiently large by Proposition \ref{p: residual estimate}. There exist positive constants $C$ and $\eps$ such that if $z_0 \in L^\infty_r (\R)$ with 
	\begin{align}
	T^{3/2} \| z_0 \|_{L^\infty_r} + T^{-1/2+4\mu} R_0  < \eps, \label{e: initial data smallness}
	\end{align}
	then the solution $z(t)$ to \eqref{e: z eqn} with initial data $z_0$ exists globally in time in $L^\infty_{-1} (\R)$ and satisfies
	\begin{align}
	\| z(t) \|_{L^\infty_{-1}} \leq \frac{C}{(t+T)^{3/2}} \left( T^{3/2} \| z_0 \|_{L^\infty_r} + T^{-1/2+4\mu} R_0 \right) 
	\end{align}
	for all $t > 0$.
\end{thmlocal}

Note that $z_0 = T^{-3/2} w_0$, where $w_0$ is the initial data for \eqref{e: w eqn}, such that \eqref{e: initial data smallness} only enforces smallness of $\| w_0 \|_{L^\infty_r}$, independent of $T$. 

The remainder of this section is dedicated to proving Theorem  \ref{t: nonlinear decay} by estimating terms in the variation of constants formula. A first attempt would aim to control $(t+T)^{3/2} \| z (t) \|_{L^\infty_{-1}}$, but $z_x$ also enters via the term $\mathcal{I}_2 (t)$. Handling this term requires the most care: in order to obtain $t^{-3/2}$ decay there, one needs $z_x (s)$ to decay in $L^1_1(\R)$, a norm that enforces localization.  To close the argument, we then use the weaker decay of $z_x (s)$ in $L^\infty_r (\R)$ according to the linear estimates in Proposition \ref{p: linear derivative estimates}. We capture this bootstrapping procedure, as well as the small-time regularity of the solution in the norm template
\begin{multline}
\Theta(t) = \sup_{0 < s < t} \bigg[ (s+T)^{3/2} \| z(s) \|_{L^\infty_{-1}} + 1_{\{0 < s < 1\}} T^{3/2} s^{1/2m} \left( \| z_x (s) \|_{L^1_1} + \| z_x (s) \|_{L^\infty_r} \right) \\ + 1_{\{s \geq 1\}} T^{1/2} \left( (s+T)^{1/2}  \| z_x (s) \|_{L^1_1} +  (s+T)^\beta \| z_x (s) \|_{L^\infty_r} \right) \bigg], \label{e: theta def}
\end{multline}
where $\beta = \frac{3}{2} - \frac{r}{2} = \frac{1}{2} - \frac{\mu}{2}$. With the spatio-temporal decay, uniformly in $T$, encoded in $\Theta(t)$, we eventually obtain global-in-time control of $\Theta(t)$ from the following local-in-time estimates. 
\begin{prop}\label{p: theta control}
	Let $r = 2 + \mu$ with $0 < \mu < \frac{1}{8}$. There exist positive constants $C_0, C_1,$ and $C_2$ independent of $z_0 \in L^\infty_r (\R)$ such that 
	\begin{align}
	\Theta(t) \leq C_0 \left( T^{3/2} \| z_0 \|_{L^\infty_r} + T^{-1/2 + 4 \mu} R_0 \right) + \frac{C_1}{T^{1/2 - 4 \mu}} \Theta(t) + C_2 K (B \Theta(t)) \Theta(t)^2,
	\end{align}
    for all $t \in (0, t_*)$, where $B = \| \rho_1 \omega^{-1} \|_{L^\infty}$, with $K$ as in \eqref{e: nonlinearity quadratic norm estimate} and $\rho_1$ given by \eqref{e: rho one sided}. 
\end{prop}

We break the proof of this proposition into several parts according to the the estimates on $z(t)$ versus $z_x(t)$, and to the different terms in the variation of constants formula \eqref{e: voc}. For the remainder of this section, we fix $0 < \mu < \frac{1}{8}$, $z_0 \in L^\infty_r (\R)$, and $T$ large, and let $t_*$ be the maximal existence time of $z(t)$ in $L^\infty_{-1} (\R)$. Unless otherwise noted, constants in this section are independent of $z_0 \in L^\infty_r(\R)$ and $T \geq T_*$. 

\subsection{Heuristics for $\Theta(t)$}\label{s: theta heuristics}
For $x$ large, the equation for $w$ at least formally resembles the equation $\NL[v] = 0$ considered in Section \ref{s: approximate solution}. Revisiting the scaling variables analysis therein (compare \eqref{e: psi plus formula}), we expect $w$ to develop a diffusive tail as well, so that to leading order, at large $x$, 
\begin{align}
    w(x,t) \sim c_0 (x+a) e^{-(x+a)^2/[4(t+T)]} 
\end{align}
for some constant $c_0$ depending on the initial data. Since $z(x,t) = (t+T)^{-3/2} w(x,t)$, this yields
\begin{align}
    z(x,t) \sim c_0 \frac{x+a}{(t+T)^{3/2}} e^{-(x+a)^2/[4 (t+T)]} \label{e: theta heuristics}
\end{align}
for $x$ large. We expect dynamics to be driven by this diffusive tail and infer the decay rates of $z(x,t)$ from this approximation. Indeed, we see decay in $L^\infty_{-1} (\R_+)$ with rate $(t+T)^{-3/2}$ in the right hand
side of \eqref{e: theta heuristics}, and decay of the derivative in $L^1_1(\R)$ with rate $(t+T)^{-1/2}$, as captured in \eqref{e: theta def}.

However, we have $\| z_0 \|_{L^\infty_r} \lesssim T^{-3/2}$, for the initial data, which we wish to encode in the constant $c_0$. For small times, we expect to retain this $T^{-3/2}$ estimates at the price of some blowup in $t$ according to small-time regularity estimates in Section \ref{s: linear estimates}, which we capture in the terms in \eqref{e: theta def} involving $0 < s <1$. We then track how this smallness of the initial data propagates to large times by incorporating $1_{s \geq 1} (s+T)^{1/2} T^{\tilde{\beta}} \| z_x (s)\|_{L^1_1}$ into the definition of $\Theta(t)$, leaving $\tilde{\beta}$ free at first. Throughout the course of the proof (see in particular Proposition \ref{p: z x large time L11}), we find that $\tilde{\beta} = \frac{1}{2}$ is the optimal choice, and this control just suffices to close our argument. 

At this point, we can also identify why we needed sharp $L^1$-$L^\infty$ estimates to replace the $L^2$-based linear estimates of $\cite{AveryScheel}$. In order to control $\mathcal{I}_2(t)$ and obtain $t^{-3/2}$ decay of $z(t)$, we need the integrand in \eqref{e: I2 def} to be bounded by $(t-s)^{-3/2} (s+T)^{-3/2}$. In order to extract $(t-s)^{-3/2}$ decay from $e^{\mcl (t-s)}$, we have to control $z_x (s)$ in a space $X$ for which an estimate $\| e^{\mcl t} \|_{X \to L^\infty_{-1}} \leq C t^{-3/2}$ holds. From Proposition \ref{p: asymptotic resolvent Lipschitz estimate}, we see that the weakest such norm on the scale of algebraically weighted $L^p$ spaces is $L^1_1(\R)$. We need this weakest norm, since we also have to extract decay from $z_x (s)$ to obtain the $(s+T)^{-3/2}$ estimate. Indeed, for $r > 2$ we have by Proposition \ref{p: linear derivative estimates}
\begin{align*}
    \| \partial_x e^{\mcl t} \|_{L^\infty_r \to L^\infty_r} \leq \frac{C}{t^{3/2 - r/2}},
\end{align*}
a decay is strictly slower than $t^{-1/2}$, which would not enable us to obtain the necessary $(s+T)^{-3/2}$ decay in the integrand and close the argument. This obstruction remains if we replace $L^\infty_r (\R), r > 2$, with the corresponding $L^2$-based localization, $L^2_{\tilde{r}} (\R), \tilde{r} > \frac{3}{2}$. Measuring the derivative instead in $L^1_1(\R)$ gives the sharp $t^{-1/2}$ estimate which suffices to close the estimates on $z(t)$. 

The proof now proceeds by estimating norms of $z$ and $z_x$ through the variation-of-constant formula invoking $\Theta$ to control the right-hand side.
% \end{remark}

\subsection{Control of $z(t)$}
We start with estimates on $z(t)$.
\begin{prop}[Estimates on $z(t)$]\label{p: z control}
	There exist constants $C_0, C_1$ and $C_2 > 0$ such that 
	\begin{align}
	(t+T)^{3/2} \| z(t) \|_{L^\infty_{-1}} \leq C_0 \left( T^{3/2} \| z_0 \|_{L^\infty_r} + T^{-1/2 + 4\mu} R_0 \right) + \frac{C_1}{T^{1/2 - 4 \mu}} \Theta(t) + C_2 K(B \Theta(t)) \Theta(t)^2
	\end{align}
	for all $t \in (0, t_*)$. 
\end{prop}

In the following, we estimate each term in the variation of constants formula. We prepare the proof with several lemmas. Throughout, we repeatedly use the following elementary inequality. 

\begin{lemma}\label{l: t minus s integral estimate}
	Let $\alpha \geq \frac{3}{2}$. There exists a constant $C >0$ such that for all $t > 0$, 
	\begin{align}
	(t+T)^{3/2} \int_0^t \frac{1}{(1+t-s)^{3/2}} \frac{1}{(s+T)^\alpha} \, ds \leq \frac{C}{T^{\alpha - 3/2}}.
	\end{align}
\end{lemma}
\begin{proof}
	We split the integral into two pieces, 
	\begin{equation*}
	(t+T)^{3/2} \int_0^t  \frac{1}{(1+t-s)^{3/2}} \frac{1}{(s+T)^{\alpha}} \, ds = (t+T)^{3/2} \left(\int_0^{t/2} + \int_{t/2}^t  \right)\frac{1}{(1+t-s)^{3/2}} \frac{1}{(s+T)^{\alpha}} \, ds %\\ + (t+T)^{3/2}  \frac{1}{(1+t-s)^{3/2}} \frac{1}{(s+T)^{\alpha}} \, ds. 
	\end{equation*}
	In the second integral, $s \geq t/2$, and so 
	\begin{align*}
	(t+T)^{3/2} \int_{t/2}^t  \frac{1}{(1+t-s)^{3/2}} \frac{1}{(s+T)^{\alpha}} \, ds  &\leq C \frac{(t+T)^{3/2}}{(t+T)^{\alpha}} \int_{t/2}^t \frac{1}{(1+t-s)^{3/2}} \, ds \\
	&= \frac{C}{(t+T)^{\alpha-3/2}} \int_0^{t/2} \frac{1}{(1+\tau)^{3/2}} \, d \tau \\
	&\leq \frac{C}{T^{\alpha-3/2}}
	\end{align*}
	since $\tau \mapsto (1+\tau)^{-3/2}$ is integrable on $[0, \infty]$. In the first integral, on $[0, t/2]$, we use $t-s \geq t/2$, 
	\begin{align*}
	(t+T)^{3/2} \int_0^{t/2} \frac{1}{(1+t-s)^{3/2}} \frac{1}{(s+T)^{\alpha}} \, ds &\leq C \frac{(t+T)^{3/2}}{(1+t)^{3/2}} \int_0^{t/2} \frac{1}{(s+T)^{\alpha}} \, ds \\
	&= C \frac{(t+T)^{3/2}}{(1+t)^{3/2}} T^{-\alpha} \int_0^{t/2} \frac{1}{(1+\frac{s}{T})^{\alpha}} \, ds \\
	&= C \frac{(t+T)^{3/2}}{(1+t)^{3/2}} T^{-\alpha+1} \int_0^{t/(2T)} \frac{1}{(1+\tau)^{\alpha}} \, d \tau, 
	\end{align*}
	where we substituted $\tau = \frac{s}{T}$. By Taylor's theorem, there exists $C > 0$ such that for $t \leq T$, we have 
	\begin{align*}
	\int_0^{t/(2T)} \frac{1}{(1+\tau)^{\alpha}} \, d \tau \leq C \frac{t}{T},
	\end{align*}
	since the left-hand side is a smooth function that vanishes at $t = 0$. Hence for $t \leq T$, we have 
	\begin{align*}
	(t+T)^{3/2} \int_0^{t/2} \frac{1}{(1+t-s)^{3/2}} \frac{1}{(s+T)^{\alpha}} \, ds \leq C \frac{(t+T)^{3/2}}{(1+t)^{3/2}} T^{-\alpha} t \leq \frac{C}{T^{\alpha-3/2}} \frac{t}{(1+t)^{3/2}} \leq \frac{C}{T^{\alpha-3/2}}.  
	\end{align*}
	For $t \geq T$, we write 
	\begin{align*}
	\left( \frac{t+T}{1+t} \right)^{3/2} = \left( 1 + \frac{T-1}{1+t} \right)^{3/2} \leq C, 
	\end{align*}
	and so in this case 
	\begin{align*}
	(t+T)^{3/2} \int_0^{t/2} \frac{1}{(1+t-s)^{3/2}} \frac{1}{(s+T)^{\alpha}} \, ds &\leq  C \frac{(t+T)^{3/2}}{(1+t)^{3/2}} T^{-\alpha+1} \int_0^{t/(2T)} \frac{1}{(1+\tau)^{\alpha}} \, d \tau \\
	&\leq C T^{-\alpha+1} \int_0^\infty \frac{1}{(1 + \tau)^{5/2}} \, d \tau \\ 
	&\leq C T^{-\alpha+1} \\
	&\leq\frac{C}{T^{\alpha-3/2}},
	\end{align*}
	which completes the proof of the lemma. 
\end{proof}

\begin{lemma}[Estimates on $\mathcal{I}_1(t)$ and $\mathcal{I}_3 (t)$] \label{l: I1 I3 estimates}
	There exists a constant $C > 0$ such that 
	\begin{align}
	(t+T)^{3/2} \| \mathcal{I}_1 (t) + \mathcal{I}_3 (t) \|_{L^\infty_{-1}} \leq \frac{C}{T^{1/2 - 4 \mu}} \Theta(t) 
	\end{align}
	for all $t \in (0, t_*)$. 
\end{lemma}
\begin{proof}
	Using the linear decay from Corollary \ref{c: linear decay L inf r}, we have 
	\begin{align*}
	(t+T)^{3/2} \| \mathcal{I}_1 (t) \|_{L^\infty_{-1}} \leq C (t+T)^{3/2} \int_0^t \frac{1}{(1+t-s)^{3/2}} \frac{1}{s+T} \| (\omega (\omega^{-1})' + \eta_*) z(s) \|_{L^\infty_r} \, ds. 
	\end{align*}
	Since $\omega (\omega^{-1})' + \eta_*$ is identically zero on the right, we can absorb any algebraic weights into this factor, so that in particular
	\begin{align*}
	\| (\omega (\omega^{-1})' + \eta_*) z(s) \|_{L^\infty_r} \leq C \| z (s) \|_{L^\infty_{-1}}.
	\end{align*}
	By the definition of $\Theta(t)$, we then have 
	\begin{align*}
	\| z(s) \|_{L^\infty_{-1}} \leq (s+T)^{-3/2} \Theta(s) \leq (s+T)^{-3/2} \Theta(t) 
	\end{align*}
	for $0 < s \leq t$, and so 
	\begin{align*}
	(t+T)^{3/2} \| \mathcal{I}_1(t) \|_{L^\infty_{-1}} &\leq C (t+T)^{3/2} \int_0^t \frac{1}{(1+t-s)^{3/2}} \frac{1}{s+T} \| z(s) \|_{L^\infty_{-1}} \, ds \\ 
	&\leq C \Theta(t) (t+T)^{3/2} \int_0^t \frac{1}{(1+t-s)^{3/2}} \frac{1}{(s+T)^{5/2}} \, ds.
	\end{align*}
	By Lemma \ref{l: t minus s integral estimate}, we have 
	\begin{align*}
	(t+T)^{3/2} \int_0^t \frac{1}{(1+t-s)^{3/2}} \frac{1}{(s+T)^{5/2}} \, ds \leq \frac{C}{T}
	\end{align*}
	for all $t > 0$, so that
	\begin{align*}
	(t+T)^{3/2} \| \mathcal{I}_1 (t) \|_{L^\infty_{-1}} \leq \frac{C}{T} \Theta(t). 
	\end{align*}
	
	We next turn to $\mathcal{I}_3(t)$. With  the linear decay estimate in Corollary \ref{c: linear decay L inf r} we have 
	\begin{align*}
	(t+T)^{3/2} \| \mathcal{I}_3 (t) \|_{L^\infty_{-1}} &\leq C (t+T)^{3/2} \int_0^t \frac{1}{(1+t-s)^{3/2}} \| (f'(\omega^{-1} \psi(s)) - f'(q_*)) z(s) \|_{L^\infty_r} \, ds \\
	&\leq C (t+T)^{3/2} \int_0^t \frac{1}{(1+t-s)^{3/2}} \| f'(\omega^{-1} \psi(s)) - f'(q_*) \|_{L^\infty_{r+1}} \| z(s) \|_{L^\infty_{-1}} \, ds.  
	\end{align*}
	By Lemma \ref{l: front versus psi} and the fact that $f$ is smooth, 
	\begin{align*}
	\| f'(\omega^{-1} \psi(s)) - f'(q_*) \|_{L^\infty_{r+1}} \leq C \| \omega^{-1} \psi(s) - q_* \|_{L^\infty_{r+1}} \leq \frac{C}{(s+T)^{1/2-4\mu}},
	\end{align*}
	so that 
	\begin{align*}
	(t+T)^{3/2} \| \mathcal{I}_3 (t) \|_{L^\infty_{-1}} \leq C (t+T)^{3/2} \Theta(t) \int_0^t \frac{1}{(1+t-s)^{3/2}} \frac{1}{(s+T)^{2 - 4 \mu}} \, ds,
	\end{align*}
	using also the definition of $\Theta(t)$ to extract a factor of $(s+T)^{-3/2}$ from $\| z(s) \|_{L^\infty_{-1}}$. By Lemma \ref{l: t minus s integral estimate}, we have 
	\begin{align*}
	(t+T)^{3/2} \int_0^t \frac{1}{(1+t-s)^{3/2}} \frac{1}{(s+T)^{2-4\mu}} \, ds \leq \frac{C}{T^{1/2 - 4 \mu}},
	\end{align*}
	and hence 
	\begin{align*}
	(t+T)^{3/2} \| \mathcal{I}_3 (t) \|_{L^\infty_{-1}} \leq \frac{C}{T^{1/2-4\mu}} \Theta(t),
	\end{align*}
	which completes the proof of the lemma. 
\end{proof}

\begin{lemma}[Estimates on $\mathcal{I}_R (t)$]
	Let $r = 2 + \mu$. There exists a constant $C > 0$ such that 
	\begin{align}
	(t+T)^{3/2} \| \mathcal{I}_R (t) \|_{L^\infty_{-1}} \leq \frac{C}{T^{1/2-4\mu}} R_0
	\end{align}
	for all $t \in (0, t_*)$. 
\end{lemma}
\begin{proof}
	Using linear decay from Corollary \ref{c: linear decay L inf r} and control of the residual from Proposition \ref{p: residual estimate} yields
	\begin{align*}
	(t+T)^{3/2} \| \mathcal{I}_R (t) \|_{L^\infty_{-1}} &\leq C (t+T)^{3/2} \int_0^t \frac{1}{(1+t-s)^{3/2}} \frac{1}{(s+T)^{3/2}} \| R(s; T) \|_{L^\infty_r} \, ds \\
	&\leq C R_0 (t+T)^{3/2} \int_0^t \frac{1}{(1+t-s)^{3/2}} \frac{1}{(s+T)^{2 - 4 \mu}} \, ds \\
	&\leq \frac{C}{T^{1/2-4\mu}} R_0,
	\end{align*}
	where the final estimate follows from Lemma \ref{l: t minus s integral estimate}. 
\end{proof}

We next estimate the nonlinearity. 

\begin{lemma}[Estimates on $\mathcal{I}_N(t)$]
	There exists a constant $C > 0$ such that 
	\begin{align}
	(t+T)^{3/2} \| \mathcal{I}_N (t) \|_{L^\infty_{-1}} \leq C K(B \Theta(t)) \Theta(t)^2
	\end{align}
	for all $t \in (0, t_*)$. 
\end{lemma}
\begin{proof}
	By the decay estimate in Corollary \ref{c: linear decay L inf r} together with the quadratic estimate \eqref{e: nonlinearity quadratic norm estimate} on the nonlinearity, we have 
	\begin{align*}
	(t+T)^{3/2} &\| \mathcal{I}_N (t) \|_{L^\infty_{-1}} \\
	&\leq C (t+T)^{3/2} \int_0^t \frac{1}{(1+t-s)^{3/2}} \| (s+T)^{-3/2} \omega N(\omega^{-1} (s+T)^{3/2} z(s)) \|_{L^\infty_r} \, ds \\
	&\leq C (t+T)^{3/2} \int_0^t \frac{1}{(1+t-s)^{3/2}} K \left( \| \omega^{-1} (s+T)^{3/2} z(s) \|_{L^\infty} \right) (s+T)^{3/2} \| z(s) \|_{L^\infty_{-1}}^2 \, ds\\
	&\leq C (t+T)^{3/2} \int_0^t \frac{1}{(1+t-s)^{3/2}} K(B \Theta(s)) (s+T)^{3/2} \| z(s) \|_{L^\infty_{-1}}^2 \, ds \\
	&\leq C K(B \Theta(t)) \Theta(t)^2 \int_0^t \frac{1}{(1+t-s)^{3/2}} \frac{1}{(s+T)^{3/2}} \, ds,
	\end{align*}
	where we have also used that $K$ and $\Theta$ are non-decreasing. By Lemma \ref{l: t minus s integral estimate}, there exists a constant $C > 0$ such that 
	\begin{align*}
	\int_0^t \frac{1}{(1+t-s)^{3/2}} \frac{1}{(s+T)^{3/2}} \, ds \leq C. 
	\end{align*}
	Hence we obtain 
	\begin{align*}
	(t+T)^{3/2} \| \mathcal{I}_N (t) \|_{L^\infty_{-1}} \leq C K(B \Theta(t)) \Theta(t)^2, 
	\end{align*}
	as desired. 
\end{proof}

We finally estimate $\mathcal{I}_2(t)$, using the decay of $z_x(t)$ in $L^1_1(\R)$ which we have encoded into the definition of $\Theta(t)$. 
\begin{lemma}[Estimates on $\mathcal{I}_2(t)$]\label{l: I2 estimates}
	There exists a constant $C > 0$ such that 
	\begin{align}
	(t+T)^{3/2} \| \mathcal{I}_2 (t) \|_{L^\infty_{-1}} \leq \frac{C}{T^{1/2}} \Theta(t)
	\end{align}
	for all $t \in (0, t_*)$. 
\end{lemma}
\begin{proof}
	First assume $0 < t < 1$. Using the small time regularity estimate from Lemma \ref{l: regularity L1 Linf}, we have 
	\begin{align*}
	(t+T)^{3/2} \| \mathcal{I}_2 (t) \|_{L^\infty_{-1}} &\leq C (t+T)^{3/2} \int_0^t \frac{1}{(t-s)^{1/2m}} \frac{1}{s+T} \| z_x (s) \|_{L^1_1} \, ds \\
	&\leq C (t+T)^{3/2} T^{-3/2} \Theta(t) \int_0^t \frac{1}{(t-s)^{1/2m}} \frac{1}{s+T} s^{-1/2m}  \, ds,
	\end{align*}
	where the last estimate follows from the definition \eqref{e: theta def} of $\Theta(t)$. Since for $t < 1$, $(t+T)^{3/2} \leq C T^{3/2}$, we have 
	\begin{align*}
	(t+T)^{3/2} \| \mathcal{I}_2 (t) \|_{L^\infty_{-1}} \leq C \Theta(t) \int_0^t \frac{1}{(t-s)^{1/2m}} \frac{s^{-1/2m}}{s+T} \, ds &\leq \frac{C}{T} \Theta(t) \int_0^t \frac{1}{(t-s)^{1/2m}} s^{-1/2m} \, ds 
	%\\&
	\leq \frac{C}{T} \Theta(t).
	\end{align*} 
	Next, we consider $1 \leq t < 2$. Here we split the integral into two pieces, since the estimates on $z_x(s)$ encoded in the definition of $\Theta(t)$ differ for $s < 1$ and $s> 1$. We write 
	\begin{align*}
	(t+T)^{3/2} \| \mathcal{I}_2 (t) \|_{L^\infty_{-1}} \leq C(t+T)^{3/2} \left[ \int_0^1 \left\| e^{\mcl (t-s)} \frac{z_x(s)}{s+T} \right\|_{L^\infty_{-1}} \, ds + \int_1^t \left\| e^{\mcl(t-s)} \frac{z_x (s)}{s+T} \right\|_{L^\infty_{-1}} \, ds \right]. 
	\end{align*}
	We estimate the first integral as in the $0 < t < 1$ case above and thereby obtain
	\begin{align*}
	(t+T)^{3/2} \int_0^1 \left\| e^{\mcl(t-s)} \frac{z_x (s)}{s+T} \right\|_{L^\infty_{-1}} \leq \frac{C}{T} \Theta(t).
	\end{align*}
	For the integral from $1$ to $t$, we again use Lemma \ref{l: regularity L1 Linf} to estimate $\| e^{\mcl (t-s)} \|_{L^1_1 \to L^\infty_{-1}}$, but the estimate on $\|z_x (s)\|_{L^1_1}$ differs slightly for $s > 1$, so that we obtain
	\begin{align*}
	(t+T)^{3/2} \int_1^t \left\| e^{\mcl (t-s)} \frac{z_x(s)}{s+T} \right\|_{L^\infty_{-1}} \, ds &\leq C T^{3/2} \int_1^t \frac{1}{(t-s)^{1/2m}} \frac{\| z_x (s) \|_{L^1_1}}{s+T} \, ds \\
	&\leq C T^{3/2} \Theta(t) \int_1^t \frac{1}{(t-s)^{1/2m}} \frac{1}{(s+T)^{3/2}} T^{-1/2} \, ds \\
	&\leq \frac{C}{T^{1/2}} \Theta(t) \int_1^t \frac{1}{(t-s)^{1/2m}} \, ds \\
	&\leq \frac{C}{T^{1/2}} \Theta(t).
	\end{align*}
	It remains only to obtain the estimate for $t \geq 2$. For this, we split the integral into four pieces,  
	\begin{align*}
	(t+T)^{3/2} \| \mathcal{I}_2 (t) \|_{L^\infty_{-1}} \leq C (t+T)^{3/2} 
	\bigg[ \int_0^1+ \int_1^{t/2} + \int_{t/2}^{t-1} + \int_{t-1}^t \bigg]\left\| e^{\mcl(t-s)} \frac{z_x (s)}{s+T} \right\|_{L^\infty_{-1}} \, ds. 
% 	\bigg[ \int_0^1 \left\| e^{\mcl(t-s)} \frac{z_x (s)}{s+T} \right\|_{L^\infty_{-1}} \, ds + \int_1^{t/2} \left\| e^{\mcl(t-s)} \frac{z_x (s)}{s+T} \right\|_{L^\infty_{-1}} \, ds \\ + \int_{t/2}^{t-1} \left\| e^{\mcl(t-s)} \frac{z_x (s)}{s+T} \right\|_{L^\infty_{-1}} \, ds + \int_{t-1}^t \left\| e^{\mcl(t-s)} \frac{z_x (s)}{s+T} \right\|_{L^\infty_{-1}} \, ds \bigg]. 
	\end{align*}
	We split the integral in this manner because we have to handle separately: the blowup of $z_x (s)$ when $s \sim 0$; the decay of $e^{\mcl (t-s)}$ when $(t-s) \sim t$; the decay of $z_x (s)$ when $s \sim t$; and the blowup of $e^{\mcl(t-s)}$ when $(t-s) \sim 0$. 
	
	For the first integral, we have $(t-s) \geq t/2 \geq 1$, and by Proposition \ref{p: linear decay L11} and the definition of $\Theta(t)$,
	\begin{align*}
	(t+T)^{3/2} \int_0^1 \left\| e^{\mcl (t-s)} \frac{z_x (s)}{s+T} \right\|_{L^\infty_{-1}} \, ds &\leq C(t+T)^{3/2} \int_0^1 \frac{1}{(t-s)^{3/2}} \frac{\| z_x (s) \|_{L^1_1}}{s+T} \, ds \\
	&\leq C \frac{(t+T)^{3/2}}{t^{3/2}} T^{-3/2} \Theta(t) \int_0^1 \frac{s^{-1/2m}}{s+T} \, ds \\
	&\leq \frac{C}{T} \Theta(t) \int_0^1 s^{-1/2m} \, ds \\
	&\leq \frac{C}{T} \Theta(t). 
	\end{align*}
	For the integral from $1$ to $t/2$, we still have $(t-s) \geq t/2 \geq 1$, but the estimate on $z_x(s)$ differs, so that
	\begin{align}
	(t+T)^{3/2} \int_1^{t/2} \left\| e^{\mcl(t-s)} \frac{z_x (s)}{s+T} \right\|_{L^\infty_{-1}} \, ds &\leq C \frac{(t+T)^{3/2}}{t^{3/2}} \Theta(t) T^{-1/2} \int_1^{t/2} \frac{1}{(s+T)^{3/2}} \, ds. \label{e: I2 estimate 1 to t over 2}
	\end{align}
	For the remaining integral, arguing as in Lemma \ref{l: t minus s integral estimate}, we obtain
	\begin{align*}
	\frac{(t+T)^{3/2}}{t^{3/2}} T^{-1/2} \int_1^{t/2} \frac{1}{(s+T)^{3/2}} \, ds \leq \frac{C}{T^{1/2}}, 
	\end{align*}
	so that
	\begin{align*}
	(t+T)^{3/2} \int_1^{t/2} \left\| e^{\mcl(t-s)} \frac{z_x (s)}{s+T} \right\|_{L^\infty_{-1}} \, ds \leq \frac{C}{T^{1/2}} \Theta(t). 
	\end{align*}
	
% 	If $t \leq T$, then by Taylor's theorem we have 
% 	\begin{align}
% 	T^{-1/2} \int_{1/T}^{t/2T} \frac{1}{(1+\tau)^{3/2}} \, d \tau \leq C \frac{t}{T^{3/2}},
% 	\end{align}
% 	so that in this case 
% 	\begin{align*}
% 	(t+T)^{3/2} \int_1^{t/2} \left\| e^{\mcl(t-s)} \frac{z_x (s)}{s+T} \right\|_{L^\infty_{-1}} \, ds \leq \frac{C}{T^{1/2}} \frac{(t+T)^{3/2}}{t^{3/2} T^{3/2}} \Theta(t) \leq \frac{C}{T^{1/2}} \Theta(t). 
% 	\end{align*}
% 	If on the other hand $t \geq T$, we return to \eqref{e: I2 estimate 1 to t over 2}, noting that in this case $(t+T)^{3/2}/t^{3/2} \leq C$ for some constant $C > 0$, so that 
% 	\begin{align*}
% 	(t+T)^{3/2} \int_1^{t/2} \left\| e^{\mcl(t-s)} \frac{z_x (s)}{s+T} \right\|_{L^\infty_{-1}} \, ds \leq \frac{C}{T^{1/2}} \Theta(t) \int_1^{t/2} \frac{1}{(s+T)^{3/2}} \, ds \leq \frac{C}{T^{1/2}} \Theta(t). 
%\end{align*}
	
	For the integral from $t/2$ to $t-1$, we have $(t-s) \geq 1$ and $s \geq t/2$, so we use the linear decay from Proposition \ref{p: linear decay L11} to estimate 
	\begin{align*}
	(t+T)^{3/2} \int_{t/2}^{t-1} \left\| e^{\mcl (t-s)} \frac{z_x (s)}{s+T} \right\|_{L^\infty_{-1}} \, ds &\leq C (t+T)^{3/2} \int_{t/2}^{t-1} \frac{1}{(t-s)^{3/2}} \frac{\| z_x (s) \|_{L^1_1}}{s+T} \, ds \\
	&\leq C (t+T)^{3/2} \Theta(t) T^{-1/2} \int_{t/2}^{t-1} \frac{1}{(t-s)^{3/2}} \frac{1}{(s+T)^{3/2}} \, ds \\
	&\leq C \frac{(t+T)^{3/2}}{(t+T)^{3/2}}\Theta(t) T^{-1/2} \int_{t/2}^{t-1} \frac{1}{(t-s)^{3/2}} \, ds \\
	&= \frac{C}{T^{1/2}} \Theta(t) \int_1^{t/2} \frac{1}{\tau^{3/2}} \, d \tau \\
	&\leq \frac{C}{T^{1/2}} \Theta(t). 
	\end{align*}
	
	Finally, for the integral from $t-1$ to $t$, we use Lemma \ref{l: regularity L1 Linf} to estimate 
	\begin{align*}
	(t+T)^{3/2} \int_{t-1}^t \left\| e^{\mcl(t-s)} \frac{z_x (s)}{s+T} \right\|_{L^\infty_{-1}} \, ds  &\leq C (t+T)^{3/2} \int_{t-1}^t \frac{1}{(t-s)^{1/2m}} \frac{\| z_x (s) \|_{L^1_1}}{s+T} \, ds \\
	&\leq C (t+T)^{3/2} \Theta(t) T^{-1/2} \int_{t-1}^t \frac{1}{(t-s)^{1/2m}} \frac{1}{(s+T)^{3/2}} \, ds \\
	&\leq \frac{C}{T^{1/2}} \frac{(t+T)^{3/2}}{(t+T)^{3/2}} \Theta(t) \int_{t-1}^t \frac{1}{(t-s)^{1/2m}} \, ds \\
	&\leq \frac{C}{T^{1/2}} \Theta(t),
	\end{align*}
	since $s \sim t$ in this region. This completes the proof of the lemma. 
\end{proof}
We emphasize that the proof of Lemma \ref{l: I2 estimates} shows  that  control of $z(t)$ necessitates good estimates on $\| z_x (t) \|_{L^1_1}$. We now gather the individual estimates on terms in $z(t)$ for the proof of Proposition \ref{p: z control}, and thereby establish control of $\| z(t) \|_{L^\infty_{-1}}$ provided appropriate  control of $z_x(t)$.
\begin{proof}[Proof of Proposition \ref{p: z control}]
	With Lemmas \ref{l: I1 I3 estimates} through \ref{l: I2 estimates} at hand, it only remains to estimate the term in the variation of constants formula \eqref{e: voc} involving the initial data, for which we have, by the linear decay estimate in Corollary \ref{c: linear decay L inf r},
	\begin{align*}
	(t+T)^{3/2} \| e^{\mcl t} z_0 \|_{L^\infty_{-1}} \leq C \frac{(t+T)^{3/2}}{(1+t)^{3/2}} \| z_0 \|_{L^\infty_r} \leq C T^{3/2} \| z_0 \|_{L^\infty_r},
	\end{align*}
	which completes the proof of the proposition. 
\end{proof}

\subsection{Control of derivatives}

To complete the proof of Proposition \ref{p: theta control}, we now estimate $z_x (t)$. We therefore differentiate the variation of constants formula, obtaining 
\begin{align}
z_x (t) = \partial_x (e^{\mcl t} z_0) + \partial_x \mathcal{I}_1 (t) + \partial_x \mathcal{I}_2 (t) + \partial_x \mathcal{I}_3 (t) + \partial_x \mathcal{I}_R (t) + \partial_x \mathcal{I}_N (t). \label{e: voc derivative}
\end{align}
We need estimates on $\| z_x (t) \|_{L^1_1}$ to close the argument for $\| z(t) \|_{L^\infty_{-1}}$, and in turn we need estimates on $\| z_x (t)\|_{L^\infty_r}$ to close the argument for $\| z_x (t) \|_{L^1_1}$. We thereby rely on the sharp linear estimates on derivatives from Section \ref{s: linear estimates}. 

\begin{prop}[Small time  estimates on $z_x(t)$]\label{p: z x small time}
	Let $r = 2 + \mu$ with $0 < \mu < \frac{1}{8}$. There exist constants $C_0, C_1,$ and $C_2 > 0$ such that 
	\begin{align}
	T^{3/2} t^{1/2m} \| z_x (t) \|_X \leq C_0 \left( T^{3/2} \| z_0 \|_{L^\infty_r} + T^{-1/2 + 4 \mu} R_0 \right) + \frac{C_1}{T^{1/2-4\mu}} \Theta(t) + C_2 K(B
	\Theta(t)) \Theta(t)^2
	\end{align}
	for all $0 \leq t < \min(1, t_*)$, where $X = L^1_1 (\R)$ or $L^\infty_r (\R)$. 
\end{prop}
\begin{proof}
	For the term involving the initial data in \eqref{e: voc derivative}, we have by Lemma \ref{l: small time derivative estimates}, 
	\begin{align*}
	T^{3/2} t^{1/2m} \| \partial_x (e^{\mcl t} z_0) \|_X \leq C T^{3/2} \| z_0 \|_{L^1_1} \leq C T^{3/2} \| z_0 \|_{L^\infty_r}.
	\end{align*}
	We now focus on the term involving $\mathcal{I}_N (t)$, as this term is the closest to critical, in the sense that if its $T$ dependence were any worse, we would not be able to close the argument. By Lemma \ref{l: small time derivative estimates} and the estimate \eqref{e: nonlinearity quadratic norm estimate} on the nonlinearity, we have 
	\begin{align*}
	T^{3/2} t^{1/2m} \| \partial_x \mathcal{I}_N (t) \|_X &\leq C T^{3/2} t^{1/2m} \int_0^t \frac{1}{(t-s)^{1/2m}} \| (s+T)^{-3/2} \omega N(\omega^{-1} (s+T)^{3/2} z(s)) \|_{L^\infty_r} \, ds \\
	&\leq C K(B \Theta(t)) T^{3/2} t^{1/2m} \int_0^t \frac{1}{(t-s)^{1/2m}} (s+T)^{3/2} \| z(s) \|_{L^\infty_{-1}}^2 \, ds \\
	&\leq C K(B \Theta(t)) \Theta(t)^2 T^{3/2} t^{1/2m} \int_0^t \frac{1}{(t-s)^{1/2m}} \frac{1}{(s+T)^{3/2}} \, ds  \\
	&\leq C K(B \Theta(t)) \Theta(t)^2 t^{1/2m} \int_0^t \frac{1}{(t-s)^{1/2m}} \, ds \\
	&\leq C K(B \Theta(t)) \Theta(t)^2
	\end{align*}
	for $0 < t < 1$, as desired. We readily obtain the estimates on the other terms by similar arguments. 
\end{proof}

\begin{prop}[Large time $L^1_1$ estimates on $z_x (t)$]\label{p: z x large time L11}
	There exist constants $C_0, C_1$, and $C_2 > 0$ such that 
	for all $t \in [1, t_*)$,
	\begin{align}
	(t+T)^{1/2} T^{1/2} \| z_x (t) \|_{L^1_1} \leq C_0 \left( T^{3/2} \| z_0 \|_{L^\infty_r} + T^{-1/2 + 4 \mu} R_0 \right) + \frac{C_1}{T^{1/2-4\mu}} \Theta(t) + C_2 K(B \Theta(t)) \Theta(t)^2.
	\end{align}
\end{prop}
\begin{proof}
    The two most important terms in \eqref{e: voc derivative} are $\partial_x \mathcal{I}_N(t)$ and $\partial_x \mathcal{I}_2(t)$. The first one is the closest to critical and thereby determines the optimal choice of $\tilde{\beta}$ in the definition of $\Theta(t)$ as discussed in Section \ref{s: theta heuristics}. The estimates on the term $\partial_x \mathcal{I}_2 (t)$ reveal how we can close our bootstrapping argument by relying on the lack of criticality in this term. 
    
    We focus first on $\partial_x \mathcal{I}_N (t)$, for which we have 
    \begin{align*}
    (t+T)^{1/2} T^{1/2} \| \partial_x \mathcal{I}_N (t) \|_{L^1_1} \leq C (t+T)^{1/2} T^{1/2} \int_0^t \left\| \partial_x e^{\mcl (t-s)} \tilde{N}(z(s), s; T) \right\|_{L^1_1} \, ds,
    \end{align*}
    where 
    \begin{align}
    \tilde{N}(z(s), s; T) = (s+t)^{-3/2} \omega N(\omega^{-1} (s+T)^{3/2} z(s)). 
    \end{align}
    We split the integral into an integral from $0$ to $t/2$ and another from $t/2$ to $t$. For the integral from $0$ to $t/2$, we use Proposition \ref{p: linear derivative estimates} together with the estimate \eqref{e: nonlinearity quadratic norm estimate} to obtain
    \begin{align}
    \int_0^{t/2} \left\| \partial_x e^{\mcl (t-s)} \tilde{N}(z(s), s; T) \right\|_{L^1_1} \, ds \leq C K(B\Theta(t)) \Theta(t)^2 \int_0^{t/2} \frac{1}{(t-s)^{1/2}} \frac{1}{(s+T)^{3/2}} \, ds.
    \end{align}
    Inside the integral, $t-s \sim t$, so that 
    \begin{align*}
    (t+T)^{1/2} T^{1/2} \int_0^{t/2} \frac{1}{(t-s)^{1/2}} \frac{1}{(s+T)^{3/2}} \, ds &\leq C t^{-1/2} (t+T)^{1/2} T^{1/2} \int_0^{t/2} \frac{1}{(s+T)^{3/2}} \, ds \\
    &= C \left( 1 + \frac{T}{t} \right)^{1/2} \int_0^{t/2T} \frac{1}{(1+\tau)^{3/2}} \, d \tau \\
    &\leq C.
    \end{align*}
    Notably, we only get a constant bound here: there is no extra $T$ decay to extract, and so this estimate determines our choice for the form of $\Theta(t)$. Hence
    \begin{align*}
        (t+T)^{1/2} T^{1/2} \int_0^{t/2} \left\| \partial_x e^{\mcl (t-s)} \tilde{N}(z(s), s; T) \right\|_{L^1_1} \, ds \leq C K(B \Theta(t)) \Theta(t)^2. 
    \end{align*}
    For the integral from $t/2$ to $t$, we have by Proposition \ref{p: linear derivative estimates}, Lemma \ref{l: small time derivative estimates}, and the estimate \eqref{e: nonlinearity quadratic norm estimate} on the nonlinearity, 
    \begin{align*}
    \int_{t/2}^t \left\| \partial_x e^{\mcl (t-s)} \tilde{N}(z(s), s; T) \right\|_{L^1_1} \, ds \leq C K(B \Theta(t)) \Theta(t)^2 \int_{t/2}^t \frac{1}{(t-s)^{1/2}} \frac{1}{(s+T)^{3/2}} \, ds.
    \end{align*}
    For the remaining integral, we have the elementary estimate 
    \begin{align*}
    (t+T)^{1/2} T^{1/2} \int_{t/2}^t \frac{1}{(t-s)^{1/2}} \frac{1}{(s+T)^{3/2}} \, ds \leq \frac{C}{t+T} T^{1/2} t^{1/2} \leq C, 
    \end{align*}
    so that 
    \begin{align*}
    (t+T)^{1/2} T^{1/2} \int_{t/2}^t \left\| \partial_x e^{\mcl (t-s)} \tilde{N}(z(s), s; T) \right\|_{L^1_1} \, ds \leq C K(B \Theta(t)) \Theta(t)^2
    \end{align*}
    as well, which completes the estimates for $\partial_x \mathcal{I}_N (t)$. 
    
    For the term involving $\partial_x \mathcal{I}_2 (t)$, we focus on only the integral from $0$ to $t/2$, since the modifications to handle the other integral are similar to the above. For this term, we have by Proposition \ref{p: linear derivative estimates}
    \begin{align*}
    (t+T)^{1/2} T^{1/2} \int_0^{t/2} \left\| \partial_x e^{\mcl(t-s)} \frac{z_x (s)}{s+T} \right\|_{L^1_1} \, ds \leq C (t+T)^{1/2} T^{1/2} \int_0^{t/2} \frac{1}{(t-s)^{1/2}} \frac{\| z_x (s) \|_{L^\infty_r}}{s+T} \, ds. 
    \end{align*}
    We then use the weaker decay of $\| z_x (s) \|_{L^\infty_r}$ built into $\Theta(t)$ to estimate
    \begin{align*}
    (t+T)^{1/2} T^{1/2} \int_0^{t/2} \frac{1}{(t-s)^{1/2}} \frac{\| z_x (s) \|_{L^\infty_r}}{s+T} \, ds \leq C \Theta(t) (t+T)^{1/2} \int_0^{t/2} \frac{1}{(t-s)^{1/2}} \frac{1}{(s+T)^{1 + \beta}} \, ds.
    \end{align*}
    Note that we have also absorbed the factor of $T^{1/2}$ with $\| z_x (s) \|_{L^\infty_r}$, using the construction of $\Theta(t)$. For the remaining integral, we have the elementary estimate
    \begin{align*}
    (t+T)^{1/2} \int_0^{t/2} \frac{1}{(t-s)^{1/2}} \frac{1}{(s+T)^{1 + \beta}} \, ds &\leq \frac{C}{t^{1/2}} (t+T)^{1/2} \int_0^{t/2} \frac{1}{(s+T)^{1+\beta}} \, ds \\
    &= C T^{-\beta} \left( 1 + \frac{T}{t} \right)^{1/2} \int_0^{t/2T} \frac{1}{(1+\tau)^{1+\beta}} \, d \tau \\
    &\leq C T^{-\beta} = \frac{C}{T^{1/2 - \mu/2}} \leq \frac{C}{T^{1/2 - 4\mu}}, 
    \end{align*}
    recalling that $\beta = \frac{1}{2} - \frac{\mu}{2}$. Estimating also the integral from $t/2$ to $t$ with similar arguments to the above, we obtain 
    \begin{align*}
        (t+T)^{1/2} T^{1/2} \| \partial_x \mathcal{I}_2 (t) \|_{L^1_1} \leq \frac{C}{T^{1/2 - 4 \mu}} \Theta(t),
    \end{align*}
    as desired. The estimates on $\partial_x \mathcal{I}_1 (t), \partial_x \mathcal{I}_3 (t), \partial_x \mathcal{I}_R (t)$, and $\partial_x (e^{\mcl t} z_0)$ are similar. 
\end{proof}

\begin{prop}[Large time $L^\infty_r$ estimates on $z_x (t)$] \label{p: z x large time Linf r}
	Let $r = 2 + \mu$ with $0 < \mu < \frac{1}{8}$. There exist constants $C_0, C_1$, and $C_2 > 0$ such that 
	\begin{align}
	(t+T)^{\beta} T^{1/2} \| z_x (t) \|_{L^\infty_r} \leq C_0 \left( T^{3/2} \| z_0 \|_{L^\infty_r} + T^{-1/2 + 4 \mu} R_0 \right) + \frac{C_1}{T^{1/2-4\mu}} \Theta(t) + C_2 K(B \Theta(t)) \Theta(t)^2
	\end{align}
	for all $t \in [1, t_*)$, where $\beta = \frac{3}{2} - \frac{r}{2} = \frac{1}{2} - \frac{\mu}{2}$. 
\end{prop}
\begin{proof}
	Since the proof is similar to that of Proposition \ref{p: z x large time L11}, we only point out the crucial feature that allows us to terminate the bootstrapping procedure. By Proposition \ref{p: linear derivative estimates}, we have 
	\begin{align*}
	(t+T)^\beta T^{1/2} \| \partial_x \mathcal{I}_2 (t) \|_{L^\infty_r} \leq C (t+T)^\beta T^{1/2} \int_0^t \frac{1}{(t-s)^\beta} \frac{ \| z_x (s) \|_{L^\infty_r}}{s+T} \, ds. 
	\end{align*}
	Owing to the extra factor of $(s+T)^{-1}$ in the integrand, we can carry out exactly the same argument used to estimate $\| \partial_x \mathcal{I}_2 (t) \|_{L^1_1}$ in the previous proposition here, and thereby obtain 
	\begin{align*}
	    (t+T)^\beta T^{1/2} \| \partial_x \mathcal{I}_2 (t) \|_{L^\infty_r} \leq \frac{C}{T^{1/2 - 4 \mu}} \Theta(t). 
	\end{align*}
\end{proof}

\subsection{Proof of Theorem \ref{t: nonlinear decay}}

The full control of $\Theta(t)$, Proposition \ref{p: theta control}, follows directly from the control of $z(t)$ and $z_x (t)$ in Propositions \ref{p: z control}, \ref{p: z x small time}, \ref{p: z x large time L11}, and \ref{p: z x large time Linf r}.

\begin{proof}[Proof of Theorem \ref{t: nonlinear decay}]
	By Proposition \ref{p: theta control}, for $T$ sufficiently large (so that $C_1/T^{1/2-4\mu} < 1$) there exist constants $\tilde{C}_0$ and $\tilde{C}_2$ so that 
	\begin{align}
	\Theta(t) \leq \tilde{C}_0 (T^{3/2} \| z_0 \|_{L^\infty_r} + T^{-1/2 + 4 \mu} R_0) + \tilde{C}_2 K(B \Theta(t)) \Theta(t)^2. \label{e: Theta control}
	\end{align}
	From the local well-posedness theory, it follows that there exists a constant $C > 0$ such that 
	\begin{align}
	\Theta(t) \leq C T^{3/2} \| z_0 \|_{L^\infty_r} \label{e: Theta 0 estimate}
	\end{align}
	for small times provided the initial data is sufficiently small. 
	Suppose $\Omega_0 := T^{3/2} \| z_0 \|_{L^\infty_r} + T^{-1/2 + 4 \mu} R_0$ is small enough so that
	\begin{align}
	2 C_0 \Omega_0 < 1 \quad \text{ and } \quad 4 C_0 C_2 K(B) \Omega_0 < 1. 
	\end{align}
	We show that 
	\begin{align}
	\Theta(t) \leq 2 C_0 \Omega_0 < 1 \label{e: Theta smallness}
	\end{align}
	for all $t \in (0, t_*)$. By \eqref{e: Theta 0 estimate}, $\Theta(t) \leq \Omega_0 \leq 2 C_0 \Omega_0 < 1$ for $t$ sufficiently small (redefining $C_0$ so that $C_0 > \frac{1}{2}$ if necessary). By construction, $t \mapsto \Theta(t)$ is continuous on $(0, t_*)$. Hence if \eqref{e: Theta smallness} does not hold, then there must be some time $t_1 > 0$ at which $\Theta(t_1) = 2 C_0 \Omega_0$. Considering \eqref{e: Theta control} at time $t_1$, we obtain
	\begin{align*}
	2 C_0 \Omega_0 \leq C_0 \Omega_0 + 4 K(B 2 C_0 \Omega_0 ) C_0^2 C_2 \Omega_0^2 \leq C_0 \Omega_0 (1 + 4  C_0 C_2 K (B) \Omega_0)
	\end{align*}
	since $K$ is non-decreasing and $2 C_0 \Omega_0 < 1$. Since also $4 C_0 C_2 K(B) \Omega_0 < 1$, we conclude
	\begin{align*}
	2 < 1 + 4 C_0 C_2 K(B) \Omega_0 < 2, 
	\end{align*}
	a contradiction. Therefore, $\Theta(t) \leq 2 C_0 \Omega_0$ for all $t \in (0, t_*)$, which implies that $t_* = \infty$ by the local well-posedness theory. This global control on $\Theta(t)$ implies in particular 
	\begin{align*}
	\| z (t) \|_{L^\infty_{-1}} \leq \frac{C}{(t+T)^{3/2}} \Omega_0, 
	\end{align*}
	for some constant $C > 0$, as desired. 
\end{proof} 

\section{Consequences for front propagation --- proof of Theorem \ref{t: main}}\label{s: propagation}
Let $u$ solve the original equation 
\begin{align}
u_t = \mathcal{P}(\partial_y) u + f(u)
\end{align}
with initial data $u_0$. Here we use $y$ for the original stationary variable, since we will consider both the stationary and moving frames in this section. Indeed, we let 
\begin{align}
x = y - c_* t + \frac{3}{2\eta_*} \log(t+T) - \frac{3}{2 \eta_*} \log (T), 
\end{align}
and define $U(x,t) = u(y,t)$, so that $U$ solves the equation in the moving frame with initial data $U_0 (x) := U(x,0) = u_0 (x)$. We then let $v = \omega U$, so that $v$ solves $\NL [v] = 0$, and let $w(x,t) = v(x,t) - \psi(x, t; T)$ be a perturbation to $\psi$, so that $w$ solves \eqref{e: w eqn}. The decay of $z(x,t)$ in Theorem \ref{t: nonlinear decay} translates into the following stability result for $w$. 
\begin{corollary}
	Let $r = 2 + \mu$ with $0 < \mu < \frac{1}{8}$, and let $R_0$ be given by \eqref{e: R0 def}. There exist positive constants $C$ and $\eps$ such that if $w_0 \in L^\infty_r (\R)$ with 
	\begin{align}
	\| w_0  \|_{L^\infty_r} + T^{-1/2 + 4 \mu} R_0 < \eps,
	\end{align}
	then the solution $w(x,t)$ exists for all positive time, and 
	\begin{align}
	\| w(\cdot, t) \|_{L^\infty_{-1}} \leq C \left( \| w_0  \|_{L^\infty_r} + T^{-1/2 + 4 \mu} R_0 \right)
	\end{align}
	for all $t > 0$. 
\end{corollary}
\begin{proof}
	With $z(x,t) = (t+T)^{-3/2} w(x,t)$, smallness of $w_0$ implies smallness of $T^{3/2} z_0$, so that the assumptions of Theorem \ref{t: nonlinear decay} are satisfied, and hence 
	\begin{align*}
	\| w(\cdot, t) \|_{L^\infty_{-1}} = (t+T)^{3/2} \| z(\cdot, t) \|_{L^\infty_{-1}} \leq C \left( \| w_0  \|_{L^\infty_r} + T^{-1/2 + 4 \mu} R_0 \right), 
	\end{align*}
	as desired. 
\end{proof}

In other words, provided the initial data is close to $\psi(\cdot, 0; T)$, the solution $U$ in the co-moving frame with the logarithmic shift is well-approximated by $\psi(\cdot, t; T)$ for all times. 
\begin{corollary}\label{c: U versus psi}
	Let $r = 2 + \mu$ with $0 < \mu < \frac{1}{8}$. There exist positive constants $C$ and $\eps$ such that if $U_0$ satisfies 
	\begin{align}
	\| \omega U_0 - \psi(\cdot, 0; T) \|_{L^\infty_r} + T^{-1/2+4\mu} R_0 < \eps, \label{e: U0 smallness condition}
	\end{align}
	then 
	\begin{align}
	\| \omega U(\cdot, t) - \psi(\cdot, t; T) \|_{L^\infty_{-1}} \leq C  \left( \| \omega U_0 - \psi(\cdot, 0; T) \|_{L^\infty_r} + T^{-1/2+4 \mu} R_0 \right)
	\end{align}
	for all $t > 0$. 
\end{corollary}

\begin{proof}[Proof of Theorem 1]
	Given $\eps > 0$ small enough so that Corollary \ref{c: U versus psi} holds with this choice of $\eps$, fix $T$ large enough so that $T^{-1/2+4\mu} R_0 < \frac{\eps}{4}$. We then define 
	\begin{align}
	\mathcal{U}_\eps = \left\{ U_0 : \omega U_0 \in L^\infty_r (\R) \text{ and } \| \omega U_0 - \psi(\cdot, 0; T) \|_{L^\infty_r} + T^{-1/2 + 4\mu} R_0 < \frac{\eps}{2} \right\}. 
	\end{align}
	For any fixed $T$, $\mathcal{U}_\eps$ is clearly open in the norm $\| U_0 \|_\rho = \| \rho_r \omega U_0 \|_{L^\infty}$. If we define 
	\begin{align}
	 \tilde{U}_0 (x) = \begin{cases}
	 \omega^{-1}(x) \psi(x, 0; T), & x < T^{1/2 + \mu} - x_0, \\
	 0, & x \geq T^{1/2 + \mu} - x_0, 	
	\end{cases}
	\end{align}
	then from the expression \eqref{e: psi plus formula} for $\psi^+$, one sees that 
	\begin{align}
	   \| \omega \tilde{U}_0 - \psi(\cdot, 0; T) \|_{L^\infty_r} \leq C T^{5/2+2\mu} \exp(-c T^{2 \mu})
	\end{align}
	for some constants $C, c > 0$. In particular, for $T$ sufficiently large depending on $\eps$, we have $\tilde{U}_0 \in \mathcal{U}_\eps$, so that Definition \ref{d: selection}, (ii), is satisfied by $\mathcal{U}_\eps$. 
	Recalling that $u(x+\tilde{\sigma}_T (t),t) = U(x, t)$ with 
	\begin{align}
	\tilde{\sigma}_T (t) = c_* t - \frac{3}{2 \eta_*} \log (t+T) + \frac{3}{2 \eta_*} \log (T),
	\end{align}
	we conclude from Corollary \ref{c: U versus psi} that
	\begin{align}
	    \sup_{x \in \R} | \rho_{-1} (x) \omega(x) [u(x + \tilde{\sigma}_T (t), t) - \omega(x)^{-1} \psi(x, t; T) ]| < \frac{\eps}{2}. 
	\end{align}
	for all $u_0 \in \mathcal{U}_\eps$. By Lemma \ref{l: front versus psi}, $\omega^{-1} \psi$ is a good approximation to the critical front, so that by the triangle inequality
	\begin{align}
	\sup_{x \in \R} | \rho_{-1}(x) \omega(x) [u(x + \tilde{\sigma}_T (t), t) - q_* (x) ]| < \frac{3 \eps}{4} \label{e: sigma tilde T estimate}
	\end{align}
	provided $T$ is sufficiently large. Then, if we define 
	\begin{align}
	\sigma(t) = c_* t - \frac{3}{2 \eta_*} \log (t) -  \frac{3}{2 \eta_*}\log \left( \frac{1}{T} \right),
	\end{align}
	we see that for $T$ fixed
	\begin{align}
	   | \tilde{\sigma}_T (t) - \sigma(t) | = \left| \log \left( \frac{1}{1 + t/T} \right) \right| \to 0 \text{ as } t \to \infty. 
	\end{align}
	Hence, since $u$ is smooth by parabolic regularity, for $t$ sufficiently large we can replace $\tilde{\sigma}_T (t)$ by $\sigma(t)$ in \eqref{e: sigma tilde T estimate} to obtain 
	\begin{align}
	 \sup_{x \in \R} | \rho_{-1} (x) \omega(x) [u(x + \sigma (t), t) - q_* (x) ]| < \eps
	\end{align}
	for $t$ sufficiently large, depending on $u_0$ via $T$. This is the refined estimate of Theorem \ref{t: main}, with 
	\begin{align}
	x_\infty (u_0) = - \frac{3}{2 \eta_*} \log \left( \frac{1}{T} \right). 
	\end{align}
	In particular, $q_*$ is a selected front. 

\end{proof}

\section{Robustness of assumptions --- proof of Theorem \ref{t: robustness}}\label{s: robustness}
We now consider an equation
\begin{align}
u_t = \mathcal{P}(\partial_x; \delta) u + f(u; \delta)
\end{align}
where $\mathcal{P} (\partial_x; \delta)$ is an elliptic operator of order $2m$ whose coefficients are smooth in $\delta$, and where $f$ is smooth in both its arguments, with $f(0; \delta) = f(1; \delta) = 0$ for $\delta$ small and $f'(0; \delta) > 0, f'(1; \delta) < 0$ for $\delta$ small. We assume that at $\delta = 0$, Hypotheses \ref{hyp: spreading speed} through \ref{hyp: resonance} are satisfied, and we conclude here that these assumptions hold for $\delta$ small. Our argument is essentially that of \cite{AveryGarenaux} adapted to a general setting, without the additional technical difficulty of regularizing the singular perturbation. 

We write the asymptotic dispersion relation on the right as 
\begin{align}
d(\lambda, \nu; c, \delta) = \mathcal{P}(\nu; \delta) + c \nu + f'(0; \delta) - \lambda. 
\end{align}
Since by assumption Hypothesis \ref{hyp: spreading speed} holds at $\delta = 0$, we have 
\begin{align}
d(0, -\eta_0; c_0, 0) = \partial_\nu d(0, -\eta_0; c_0, 0) = 0 
\end{align}
for some $c_0, \eta_0 > 0$, and there is $\alpha_0 > 0$ such that 
\begin{align}
d(\lambda, \nu - \eta_0; c_0, 0) = \alpha_0 \nu^2 - \lambda + \mathrm{O}(\nu^3). 
\end{align}

\begin{lemma}[Robustness of simple pinched double roots] \label{l: robustness of simple PDR}
	There exists $\delta_0 > 0$ and smooth functions $\eta_* : (-\delta_0, \delta_0) \to \C$ and $c_* : (-\delta_0, \delta_0) \to \R$ such that $\eta_*(0) = \eta_0, c_*(0) = c_0$, and
	\begin{align}
	d(0, -\eta_* (\delta); c_*(\delta), 0) = \partial_\nu d(0, -\eta_*(\delta); c_*(\delta), 0) = 0 \label{e: robustness double root}
	\end{align}
	for all $\delta \in (-\delta_0, \delta_0)$. Furthermore, there is a smooth function $\alpha : (-\delta_0, \delta_0) \to \R_+$ such that $\alpha(0) = \alpha_0$ and for $\lambda$ and $\nu$ small, we have 
	\begin{align}
	d(\lambda, \nu - \eta_*(\delta); c_*(\delta), \delta) = \alpha(\delta) \nu^2 - \lambda + \mathrm{O}(\nu^3), \label{e: robustness dispersion expansion}
	\end{align}
	where the $\mathrm{O}(\nu^3)$ terms depend on $\delta$ as well. As a result, there are two roots $\nu^\pm$ of the dispersion relation $\nu \mapsto d(\gamma^2, \nu - \eta_*(\delta); c_*(\delta), \delta)$ which satisfy 
	\begin{align*}
	\nu^\pm (\gamma; \delta) = \pm\nu_0 (\delta) \gamma + \mathrm{O}(\gamma^2),
	\end{align*}
	for $\gamma$ small $\mcl(\delta)$, where $\nu_0 (\delta) > 0$ for $\delta$ small. 
\end{lemma}
\begin{proof}
	Define $F : \C \times \R^2 \to \C^2$ by 
	\begin{align}
	F(\nu, c; \delta) = \begin{pmatrix}
	d (0, \nu; c, \delta) \\
	\partial_\nu d (0, \nu; c, \delta)
	\end{pmatrix}.
	\end{align}
	The result \eqref{e: robustness double root} is equivalent to $F(-\eta_*(\delta), c_*(\delta); \delta) = 0$, and by assumption, we have $F(-\eta_0, c_0; 0) = 0$. The derivative of $F$ with respect to its first two arguments,
	\begin{align}
	D_{\nu, c} F(-\eta_0, c_0; 0) = \begin{pmatrix}
	0 & - \eta_0 \\
	2 \alpha_0 & 0
	\end{pmatrix},
	\end{align}
    is invertible. Therefore, by the implicit function theorem, for $\delta$ small there exist $\eta_*(\delta)$ and $c_*(\delta)$ smooth in $\delta$ so that $F(-\eta_*(\delta), c_*(\delta); \delta) = 0$, as desired. The expansion \eqref{e: robustness dispersion expansion} follows from this smoothness and the smoothness of the original dispersion relation. For more details on continuity of linear spreading speeds, see \cite{HolzerScheelPointwiseGrowth}. 
\end{proof}
We note that this lemma together with standard spectral perturbation theory implies that the essential spectrum of $\mcl(\delta)$ is marginally stable, touching the imaginary axis only at the origin, i.e. Hypothesis \ref{hyp: spreading speed} is satisfied. 

We now turn to the existence of the critical front. Let $q_0$ denote the critical front at $\delta = 0$ and $\mcl(0)$ the linearization about that front, in the weighted space with weight $\omega_0$, where
\begin{align}
\omega_\delta (x) = \begin{cases}
e^{\eta_*(\delta) x}, & x \geq 1, \\
1, & x \leq -1. 
\end{cases}
\end{align}
Recall from Section \ref{s: resolvent estimates} that our assumptions imply that $\mcl(0)$ is Fredholm with index -1 when considered as an operator from $H^{2m}_{\mathrm{exp}, \eta} (\R) \to L^2_{\mathrm{exp}, \eta}( \R)$ for $\eta> 0$ small. By Hypothesis \ref{hyp: resonance}, the kernel of $\mcl(0)$ is trivial on this space, so it has a one-dimensional cokernel. 

\begin{lemma}\label{l: robustness phi}
	Let $\mcl(0)^* : H^{2m}_{\mathrm{exp}, -\eta} (\R) \subset L^2_{\mathrm{exp}, -\eta} (\R) \to L^2_{\mathrm{exp}, -\eta} (\R)$ be the $L^2$-adjoint of $\mcl(0)$, and let $\ker \mcl(0)^* = \mathrm{span} (\phi)$. Then 
	\begin{align}
	\langle \mcl(0) \chi_+, \phi \rangle \neq 0, 
	\end{align}
	where $\chi_+$ is as defined in Section \ref{s: resolvent estimates}. 
\end{lemma}
\begin{proof}
	As in Section \ref{s: resolvent estimates}, we make the ansatz $u = w + \beta \chi_+$ for the resonance equation $\mcl u = 0$, where $w \in H^{2m}_{\mathrm{exp}, \eta} (\R)$ and $\beta \in \R$. We then let $P$ be the orthogonal projection onto the range of $\mcl(0)$ in $L^2_{\mathrm{exp}, \eta}(\R)$, and decompose the resulting equation as 
	\begin{align}
	\begin{cases}
	P \mcl(0) (w + \beta \chi_+) &= 0, \\
	\langle \mcl(0)(w + \beta \chi_+), \phi \rangle &= 0. 
	\end{cases} \label{e: robustness resonance equation}
	\end{align}
	By construction, $P \mcl(0) : H^{2m}_{\mathrm{exp}, \eta} (\R) \to \text{Range}(\mcl(0))$ is invertible, and so this equation has a solution $(w, \beta)$ if and only if $\langle \mcl(0) (w + \beta \chi_+), \phi \rangle = 0$. By definition, 
	\begin{align*}
	\langle \mcl(0) w, \phi \rangle = \langle w, \mcl(0)^* \phi \rangle = 0. 
	\end{align*} 
	However, $\chi_+ \notin H^{2m}_{\mathrm{exp}, \eta} (\R)$ is not localized and so we cannot just pass $\mcl(0)^*$ onto $\phi$ in this term. Hence, \eqref{e: robustness resonance equation} has a solution $(w, \beta) \in H^{2m}_{\mathrm{exp}, \eta} (\R) \times \R$ if and only if $\langle \mcl(0) \chi_+, \phi \rangle = 0$. A solution to this equation would give a bounded solution to $\mcl u = 0$, contradicting Hypothesis \ref{hyp: resonance}, so in particular $\langle \mcl(0) \chi_+, \phi \rangle \neq 0$. 
\end{proof}

We now state and prove the existence of the critical front, which solves 
\begin{align}
\mathcal{P}(\partial_x; \delta) q + c_*(\delta) \partial_x q + f(q; \delta) = 0, \quad q(-\infty) = 1, \quad q(\infty) = 0. \label{e: robustness TW eqn}
\end{align}
By assumption, at $\delta = 0$, there exists a solution $q_0$ to this equation such that 
\begin{align}
q_0 (x) = (a_0 + x)e^{-\eta_0 x} + \mathrm{O}(e^{-(\eta_0 + \eta) x})
\end{align}
for some $a_0 \in \R$ and $\eta > 0$. 
\begin{prop}
	For $\delta$ sufficiently small, there exists a smooth solution $q_*(\cdot; \delta)$ to \eqref{e: robustness TW eqn} such that 
	\begin{align}
	q_* (x; \delta) = (a(\delta) + x) e^{-\eta_*(\delta) x} + \mathrm{O}(e^{-(\eta_0 + \eta) x})
	\end{align}
	for some $\eta > 0$ and $a(\delta)$ depending smoothly on $\delta$. 
\end{prop}
\begin{proof}
	As in \cite{AveryGarenaux}, we make an ansatz 
	\begin{align}
	q (x) = \chi_- (x) + w(x) + \chi_+(x) (a+x) e^{-\eta_*(\delta) x},
	\end{align}
	where we will require $w$ to be exponentially localized. This ansatz captures convergence to $1$ and $0$ at $-\infty$ and $+\infty$ respectively, and in particular the weak exponential decay near $+\infty$ associated to the simple pinched double root. To enforce exponential localization of $w$, we set $v = \omega w$, and require $v \in H^{2m}_{\mathrm{exp}, \eta} (\R)$ for some $\eta > 0$ small. Inserting this ansatz into \eqref{e: robustness TW eqn} leads to an equation $F(v, a; \delta) = 0$, where 
	\begin{align}
	F : H^{2m}_{\mathrm{exp}, \eta} (\R) \times \R \times (-\delta_0, \delta_0) \to L^2_{\mathrm{exp}, \eta} (\R). 
	\end{align}
	By our assumption, there exist $(v_0, a_0) \in H^{2m}_{\mathrm{exp}, \eta} (\R) \times \R$ such that $F(v_0, a_0; 0) = 0$. 
	
	Computing the linearizations, one finds $D_v F(v_0, a_0; 0) = \mcl(0)$ and $D_a F(v_0, a_0; 0) = \mcl(0) \chi_+$. As noted, $\mcl(0)$ is Fredholm  with index -1, so by the Fredholm bordering lemma, the joint linearization $D_{(v,a)} F(v_0, a_0; 0)$ is Fredholm with index 0. It is then invertible provided the range of $D_a F(v_0, a_0; 0)$ is linearly independent from the range of $\mcl(0)$, which is true by Lemma \ref{l: robustness phi}. With the implicit function theorem, we find $v(\cdot; \delta) \in H^{2m}_{\mathrm{exp}, \eta}(\R)$ and $a(\delta) \in \R$ smooth in $\delta$ such that $F(v(\cdot; \delta), a(\delta); \delta) = 0$, and hence, by the form of our ansatz, there is a smooth solution 
	\begin{align}
	q_* (x; \delta) = \chi_- (x) + \omega^{-1} (x) v(x; \delta) + \chi_+ (x) (a(\delta) + x) e^{-\eta_*(\delta) x}
	\end{align}
	to \eqref{e: robustness TW eqn}. Since $v \in H^{2m}_{\mathrm{exp}, \eta} (\R)$, $q_*$ has the desired asymptotics as $x \to \infty$. For more details, albeit in the specific case of the extended Fisher-KPP equation, see \cite{AveryGarenaux}. 
\end{proof}

With the existence of the critical front in hand, we let $\mcl(\delta)$ denote the linearization about the critical front in the weighted space with weight $\omega_\delta$. We now show that there is no resonance for $\mcl(\delta)$ at $\lambda = 0$ for $\delta$ small. As in Section \ref{s: resolvent estimates}, we substitute the ansatz 
\begin{align}
u(x) = w(x) + \beta \chi_+ (x) e^{\nu^-(\gamma; \delta) x}
\end{align}
to the equation $(\mcl(\delta) - \gamma^2) u = g$ for $g \in L^2_{\mathrm{exp}, \eta} (\R)$. Decomposing the resulting equation as in Lemma \ref{l: robustness phi}, we obtain the system 
\begin{align}
\begin{cases}
P (\mcl(\delta) - \gamma^2) (w + \beta \chi_+ e^{\nu^-(\gamma; \delta) \cdot}) &= Pg, \\
\langle (\mcl(\delta) - \gamma^2) (w + \beta \chi_+ e^{\nu^-(\gamma; \delta) \cdot}), \phi \rangle &= \langle g, \phi \rangle,
\end{cases} \label{e: robustness eval problem}
\end{align}
where $P$ is again the orthogonal projection onto the range of $\mcl(0)$, and $\ker \mcl(0)^* = \mathrm{span} (\phi)$. We define 
\begin{align}
\mathcal{F} (w, \beta; \gamma, \delta) = P (\mcl(\delta) - \gamma^2) (w + \beta \chi_+ e^{\nu^-(\gamma; \delta) \cdot}), 
\end{align}
so that the first equation in \eqref{e: robustness eval problem} reads $\mathcal{F} (w, \beta; \gamma, \delta) = P g$. For $g = 0$, we have a trivial solution $F(0, 0; 0, 0) = 0$. Since $D_w F (0, 0; 0, 0) = P\mcl(0)$ is invertible, the implicit function theorem gives a solution $w(\beta; \gamma, \delta) \in H^{2m}_{\mathrm{exp}, \eta} (\R)$ for $\beta, \gamma$, and $\delta$ small which is unique in a neighborhood of the origin. Using this uniqueness and the fact that $\mathcal{F}$ is linear in $w$ and $\beta$, one sees that $w (\beta; \gamma, \delta) = \beta \tilde{w}(\gamma, \delta)$ for some $\tilde{w}(\gamma, \nu) \in H^{2m}_{\mathrm{exp}, \eta} (\R)$. Inserting this solution into the second equation in \eqref{e: robustness eval problem} for $g = 0$, we find an equation
\begin{align}
E(\gamma, \delta) := \langle (\mcl(\delta) - \gamma^2) (\tilde{w}+\chi_+ e^{\nu^-(\gamma; \delta) \cdot}), \phi \rangle = 0. \label{e: robustness E def}
\end{align}
The function $E$ is analytic in $\gamma$ and smooth in $\delta$, and by Lemma \ref{l: robustness phi}, we have $E(0, 0) \neq 0$, so that $E(\gamma, \delta)$ is nonzero for $\gamma, \delta$ in a neighborhood of the origin as well. From the construction of $E$, we see that we have a solution $(w, \beta)$ to \eqref{e: robustness eval problem} if and only if $E(\gamma, \delta) = 0$. One can further show that $(\mcl(\delta) - \gamma^2)$ is invertible if $\gamma$ is to the right of the essential spectrum of $\mcl(\delta)$ and $E(0, 0) \neq 0$, so that the zeros of $E$ precisely detect eigenvalues (and more generally resonances) of $\mcl(\delta)$. This results in the following proposition --- again, see \cite{AveryGarenaux} for more details in the case of the extended Fisher-KPP equation. 
\begin{prop}
	The equation $(\mcl(\delta) - \gamma^2) u = 0$ has a bounded solution if and only if $E(\gamma, \delta) = 0$. In particular, $\mcl(\delta)$ has no eigenvalues in a neighborhood of the origin, and no resonance at $\lambda = 0$ for $\delta$ sufficiently small. 
\end{prop}

\begin{proof}[Proof of Theorem \ref{t: robustness}]
	We have already shown that Hypotheses \ref{hyp: spreading speed} and \ref{hyp: front existence} hold for $\delta$ small, and that there is no resonance at the origin, nor any eigenvalues bifurcating from the essential spectrum. Since the left Fredholm border is given by the zero sets of an algebraic curve $\lambda \mapsto d^-(\lambda, i k)$ for $k \in \R$, Hypothesis \ref{hyp: stable on left} holds for $\delta$ small as well. The eigenvalue problem away from the essential spectrum is a regular perturbation problem, and so for $\delta$ small there cannot be any unstable eigenvalues and Hypothesis \ref{hyp: resonance} holds in full as well. 
\end{proof}

\section{Examples and discussion}\label{s: discussion}
We discuss examples and limitations of our results. 

\noindent \textbf{Second-order equations.} The simplest application of our results is to the Fisher-KPP equation
\begin{align}
u_t = u_{xx} + f(u),\qquad  f(0) = f(1) = 0,\quad   f'(0) > 0,\quad f'(1) < 0.\label{e: ex scalar}
\end{align}
Spectral stability in the sense of Hypothesis \ref{hyp: resonance} holds when, for instance, $0 < f(u) \leq f'(0) u$ for $u \in (0,1)$ \cite{aronson, Sattinger}. As mentioned in the introduction, front selection results for this equation allow for large classes of initial data such as 
% 
% there are a variety of front propagation results available for this equation using either probabilistic methods \cite{Bramson1, Bramson2} and comparison principle arguments \cite{Lau, Comparison1, Comparison2, Comparison3}. These results apply to step function initial data, or 
compactly supported perturbations of the step function --- the ``diffusive tail'' does not have to be baked into the initial data.  However, these results mostly  restrict to positive solutions. In this regard, we believe that restricting to the class of initial data with well-formed ``diffusive tail'' is not merely a technical limitation. Considering for example  \eqref{e: ex scalar} with the bistable nonlinearity $f(u) = u - u^3$ with an additional negative invasion front connecting $-1$ to $0$, one can ask for a description of selection basins for both positive and negative fronts.  Considering ``step-like'' initial data $u_0$ such that $u_0 (x) \equiv 1$ for $x \leq 0$ and $u_0$ is strongly localized on the right, we expect that the long time behavior is determined by the form of $u_0$ as $x \to \infty$. If the initial data has a \textit{negative} diffusive tail, the solution will develop a roughly stationary \textit{kink} between the two stable states $\pm 1$, while $u = -1$ is spreading into the unstable state $u = 0$. The tail dependence is more dramatic if one considers strongly asymmetric cubics $f(u)=(u-a)(1-u^2)$, with \hl{$a = 1 - \eps$, $\eps$ small,} such that the invasion $-1$ to $a$ is pushed, while the $1$-to-$a$-invasion remains pulled \cite{HadelerRothe} \hl{(see below for further discussion of pushed fronts)}. In this case, not only the selected state in the wake but even the {propagation speed} depend in subtle ways on tail behavior since both pushed and pulled fronts are ``selected''. Again, well-developed Gaussian tails appear to select front speed and  state in the wake in the sense that open classes of initial conditions with such tails converge to the corresponding front. In summary, a description of the boundaries of the mutual basins of attraction is in many ways a question of global dynamics, which to our knowledge has not been explored. 

% In this line of reasoning, even in systems with comparison principles, selection of fronts depends in subtle ways on tail behavior of the initial data. 
% Well-developed Gaussian tails guarantee selection in an open neighborhood and in open classes of equations. Modifying Gaussian tails, for instance through a simple cut-off of the original exponential front tail, or by changing the sign of the Gaussian tail, are large perturbations in our sense, and numerical evidence indicates that they can in fact lead to selection of different fronts. 

\noindent \textbf{The extended Fisher-KPP equation.}
The ideas laid out in Section \ref{s: robustness} were previously developed  and used in \cite{AveryGarenaux} to show that Hypotheses \ref{hyp: spreading speed} through \ref{hyp: resonance} hold for the extended Fisher-KPP equation 
\begin{align}
u_t = -\delta^2 u_{xxxx} + u_{xx} + f(u),\label{e:efkpp}
\end{align}
for $\delta$ small.  The perturbation in $\delta$ is singular such that, compared to the analysis in Section \ref{s: robustness}, an additional regularization step is required in \cite{AveryGarenaux} \hl{in order to perturb from the second order Fisher-KPP equation to \eqref{e:efkpp}. Once this regularization step has been carried out, establishing that \eqref{e:efkpp} satisfies Hypotheses \ref{hyp: spreading speed} through \ref{hyp: resonance} for some $\delta > 0$, Theorem \ref{t: robustness} applies to show that all nearby fourth order equations satisfy these assumptions.} Equation \eqref{e:efkpp} is of interest due to its role in describing the behavior of solutions to reaction-diffusion systems near certain higher co-dimension bifurcations \cite{RottschaferDoelman}, and its ability to interpolate between ``simple'' invasion fronts and more complex phenomena \cite{DeeSaarloos}. This example also highlights that our approach is not restricted to second order equations and does not rely on the presence of a comparison principle. The set of fourth order equations to which our results apply is therefore both non-empty and open in the sense of Theorem \ref{t: robustness}. We expect that a perturbation analysis analogous to \cite{AveryGarenaux} would carry over to perturbations of higher order, $-\delta^2 (i\partial_x)^{2m}u$.

\noindent \textbf{Systems of equations.}
We have focused here on a general framework for scalar equations for simplicity and clarity of presentation. However, we expect that our methods can be used to prove analogous results in systems of equations satisfying appropriate versions of our assumptions. In particular, the linear spreading speed analysis extends readily and yields diffusive dynamics in the leading edge: one finds an associated exponential weight and an associated eigenvector, which can be used to construct Gaussian tails. In the case where the linearization in the leading edge is diagonal, this Gaussian tail would be confined to one component in the system, a case which occurs in particular in Lotka-Volterra systems as considered by Faye and Holzer \cite{FayeHolzerLotKaVolterra}, who obtained sharp local stability of pulled fronts in this system.

\noindent \textbf{Pattern-forming systems.} A common exception to our results are oscillatory pulled fronts, with pinched double roots $\lambda_*=i\omega_*$, $\nu_*\in \C\setminus \R$, and selected states in the wake that are not exponentially stable. Both difficulties combine in pattern-forming systems such as the Swift-Hohenberg, the Cahn-Hilliard, or phase-field equations \cite{vanSaarloosReview,colleteckmann,sch,mesuro}. We expect that the key mechanism of front selection, exploited here, through matching of a diffusive tail with the main front profile can be adapted  \cite{vanSaarloosII, vanSaarloosReview}. However, even results on asymptotic stability of fronts appear to be known only for speeds above the linear spreading speed \cite{SchneiderEckmann}. Related but somewhat simpler examples arise in pattern-forming mechanisms with a stationary mode $\lambda_*=0,\nu_*\in\R$, such as the Ginzburg-Landau modulation approximation to Swift-Hohenberg, where stability is known up to the critical speed \cite{EckmannWayne}, or the FitzHugh-Nagumo equation \cite{cartersch}. Our results here do not directly apply to these fronts, despite stationary invasion, due to the diffusive stability of the patterns in the wake, so that Hypothesis \ref{hyp: stable on left} is not satisfied. However, we expect this to be mostly a technical issue, rather than a fundamental obstacle to extending the analysis here to these cases.

% In many interesting systems, the invasion process is driven by a simple pinched double root in a single equation, with the other equations in some sense ``slaved'' to these dynamics, and in such cases one should be able to readily modify our ansatz to incorporate appropriate perturbations in the other components and establish a version of Theorem \ref{t: main} in this setting. This is the case, for instance, for the pulled fronts in the Lotka-Volterra system considered by Faye and Holzer \cite{FayeHolzerLotKaVolterra}, in which they obtained sharp local stability results analogous to \cite[Theorem 1]{AveryScheel} for this system. 

\noindent \textbf{Pushed fronts and other instabilities of the front.}
Hypothesis \ref{hyp: resonance} requires absence of unstable eigenvalues. Indeed, as mentioned above, nonlinearities can create instabilities of fronts propagating at the linear speed and lead to the selection of faster \emph{pushed fronts}. Simple explicit examples occur in the cubic family $u_t=u_{xx}+u(u+\delta)(1-\delta-u)$, where fronts connecting $1-\delta$ to $0$ are stable for $1/3<\delta\leq 1/2$ but unstable with a single unstable eigenvalue for $0<\delta< 1/3$. In this latter regime, invasion is faster than the linear spreading, mediated by a steeper, pushed front. \hl{Compared to pulled fronts, it is much simpler to establish selection of pushed fronts, i.e. that pushed fronts attract open classes of steep initial data. Perturbations that cut off the tail of a pushed front are small in a space with an exponential weight that pushes the essential spectrum strictly into the left half plane, allowing selection to be established with a classical stability argument \cite{Sattinger}}. At the transition, in this simple example at $\delta=1/3$, the linearization at the front propagating with the linear spreading speed possesses a resonance at $\lambda=0$, violating Hypothesis  \ref{hyp: resonance}. \hl{In this boundary case, it is conjectured \cite{vanSaarloosReview} that the front still propagates with speed $c_*$, but with the $-\frac{3}{2 \eta_*} \log t$ delay replaced by $-\frac{1}{2 \eta_*} \log t$ due to the stronger effect of the nonlinearity.}

\hl{Beyond pushed fronts, one also observes destabilization against complex eigenvalues or against essential spectrum induced by instabilities in the wake. All of those lead to more complex dynamics, for instance oscillatory invasion despite a zero frequency pinched double root. Specific examples are studied in  \cite{hs14} in simple problems with comparison principles. Generalizing these examples to complex coefficient amplitude equations, one readily finds scenarios where invasion dynamics are periodic or even chaotic, despite predicted simple stationary dynamics from the pinched double root analysis; compare also the complexity of invasion dyanmics in the complex Ginzburg-Landau equation as described for instance in \cite{vanSaarloosReview,aransonkramer,Sherratt10890}. We are not aware of selection results in the presence of oscillatory or chaotic dynamics.}

\noindent \textbf{Necessity of Hypothesis \ref{hyp: spreading speed}.}
While we expect that our assumptions and results hold in open classes of systems of several equations, we caution that there are examples of selected fronts that do not satisfy our assumptions. Spreading in these examples relies on different pointwise instability mechanisms, which may preclude front stability in any fixed exponential weight or even necessitate linear speed selection criteria different from the pinched double root criterion \cite{FayeHolzerScheel}. Moreover, stability analysis and numerical evidence strongly suggest that such selection mechanisms occur in open classes of equations. 

Specifically, such anomalous modes of invasion were first studied in systems of two coupled Fisher-KPP equations \cite{HolzerAnomalous,HolzerAnomalous2,FayeHolzerScheel}, where a pointwise decaying second component accelerates the propagation in the first component through a weakly decaying exponential tail. This mechanism was interpreted in \cite{FayeHolzerScheel} more broadly as a spreading behavior mediated by \emph{resonant couplings}, present in open classes of equations. Such resonances can preclude stability of fronts in exponentially weighted spaces when the pinched double root criterion gives the correct spreading speed: the associated front is shown to be asymptotically stable in a model problem in \cite{FayeHolzerScheelSiemer} and strong numerical evidence indicates that it is selected in our sense. \hl{Resonant couplings} can also give rise to spreading speeds and selected fronts that are not predicted by pinched double roots  but rather by other resonances in the complex dispersion relation. From this perspective, \hl{simple} pinched double roots are branched 1:1 resonances. Unbranched 1:1 resonances are non-generic but occur in \cite{HolzerAnomalous,HolzerAnomalous2}. The phenomena in  \cite{FayeHolzerScheel,FayeHolzerScheelSiemer} are induced by  1:2 resonances. 

Spreading speeds in these examples can be reliably determined from a formal \hl{linear} marginal stability \hl{criterion} for resonant modes \cite{FayeHolzerScheel}. However, even formulating a marginal stability conjecture for selection of nonlinear fronts in these examples is challenging since the invasion process crucially relies on modes that exhibit pointwise temporal decay in the frame with the observed spreading speed. It is then not clear how to formulate linear (or nonlinear) marginal stability near such a time-dependent profile.

\bibliographystyle{abbrv}

\bibliography{references}
\end{document}